\documentclass[reqno,a4paper,twoside]{amsart}

\usepackage{enumitem}

\setenumerate{label=\textnormal{(\arabic*)}}

\usepackage{amsmath,amssymb,dsfont,verbatim,mathtools,bm,geometry,fge,endnotes}

\usepackage{booktabs}
\newcommand{\ra}[1]{\renewcommand{\arraystretch}{#1}}
\usepackage{multirow}
\usepackage[latin1]{inputenc}
\usepackage[raggedright]{titlesec}
\usepackage{tikz}
\usetikzlibrary{decorations.pathreplacing}
\usepackage{float,subfig}
\usepackage{amsrefs}
\usepackage[foot]{amsaddr}

\titleformat{\chapter}[display]
{\normalfont\huge\bfseries}{\chaptertitlename\\thechapter}{20pt}{\Huge}
\titleformat{\section}
{\normalfont\Large\bfseries\center}{\thesection}{1em}{}
\titleformat{\subsection}
{\normalfont\large\bfseries}{\thesubsection}{1em}{}
\titleformat{\subsubsection}[runin]
{\normalfont\normalsize\bfseries}{\thesubsubsection}{1em}{}
\titleformat{\paragraph}[runin]
{\normalfont\normalsize\bfseries}{\theparagraph}{1em}{}
\titleformat{\subparagraph}[runin]
{\normalfont\normalsize\bfseries}{\thesubparagraph}{1em}{}
\titlespacing*{\chapter} {0pt}{50pt}{40pt}
\titlespacing*{\section} {0pt}{3.5ex plus 1ex minus .2ex}{2.3ex plus .2ex}
\titlespacing*{\subsection} {0pt}{3.25ex plus 1ex minus .2ex}{1.5ex plus .2ex}
\titlespacing*{\subsubsection}{0pt}{3.25ex plus 1ex minus .2ex}{1.5ex plus .2ex}
\titlespacing*{\paragraph} {0pt}{3.25ex plus 1ex minus .2ex}{1em}
\titlespacing*{\subparagraph} {\parindent}{3.25ex plus 1ex minus .2ex}{1em}
\usepackage{titletoc}
\contentsmargin{2.55em}
\dottedcontents{section}[1.5em]{}{2.8em}{1pc}
\dottedcontents{subsection}[3.7em]{}{3.4em}{1pc}
\dottedcontents{subsubsection}[7.6em]{}{3.2em}{1pc}
\dottedcontents{paragraph}[10.3em]{}{3.2em}{1pc}

\newcommand{\hs}{\hspace{-0.5pt}}

\def\xcirc{\hs\circ\hs}

\newtheorem{theorem}{Theorem}[section]

\newtheorem{proposition}[theorem]{Proposition}
\newtheorem{corollary}[theorem]{Corollary}

\theoremstyle{definition}

\theoremstyle{remark}
\newtheorem{remark}[theorem]{Remark}

\DeclareMathOperator{\Aut}{Aut}
\DeclareMathOperator{\siq}{SIQC}
\DeclareMathOperator{\CH}{CH}
\DeclareMathOperator{\ord}{ord}
\DeclareMathOperator{\J}{J}
\DeclareMathOperator{\Jac}{Jac}
\DeclareMathOperator{\ide}{id}
\DeclareMathOperator{\Ext}{Ext}
\DeclareMathOperator{\GL}{GL}
\DeclareMathOperator{\Supp}{Supp}
\DeclareMathOperator{\Z}{Z}
\DeclareMathOperator{\en}{en}
\DeclareMathOperator{\En}{En}
\DeclareMathOperator{\factors}{factors}
\DeclareMathOperator{\N}{N}
\DeclareMathOperator{\Ss}{S}
\DeclareMathOperator{\st}{st}
\DeclareMathOperator{\St}{St}
\DeclareMathOperator{\Cone}{Cone}
\DeclareMathOperator{\PE}{PE}
\DeclareMathOperator{\val}{val}
\DeclareMathOperator{\dir}{dir}
\DeclareMathOperator{\Dirsup}{Valsup}
\DeclareMathOperator{\Dirinf}{Valinf}
\DeclareMathOperator{\Succ}{Succ}
\DeclareMathOperator{\Pred}{Pred}
\DeclareMathOperator{\Dir}{Dir}
\DeclareMathOperator{\HH}{H}
\DeclareMathOperator{\linspan}{linspan}
\DeclareMathOperator{\lcm}{lcm}
\DeclareMathOperator{\End}{End}
\newcommand{\ov}{\overline}
\newcommand{\ot}{\otimes}
\newcommand{\wh}{\widehat}
\newcommand{\wt}{\widetilde}
\newcommand{\ep}{\epsilon}
\newcommand{\De}{\Delta}
\newcommand{\urho}{\underline{\rho}}
\newcommand{\usigma}{\underline{\sigma}}
\newcommand{\uDir}{\underline{\Dir}}
\newcommand{\BigFig}[1]{\parbox{12pt}{\Huge #1}}
\newcommand{\BigZero}{\BigFig{0}}
\newcommand{\dpu}{\mathbin{:}}
\renewcommand{\theequation}{\thesection.\arabic{equation}}

\begin{document}
\title[Intersection of two longest paths not a separator]{Minimal graph in which the intersection of two longest paths is not a separator}

\author[Juan Gutierrez]{Juan Gutierrez$^{1}$}
\address{$^1$ Departamento de Ciencia de la Computaci\'on
Universidad de Ingenier\'ia y Tecnolog\'ia (UTEC)}
\email{jgutierreza@utec.edu.pe}

\author[Christian Valqui]{Christian Valqui$^{2,}$$^3$}
\address{$^2$Pontificia Universidad Cat\'olica del Per\'u, Secci\'on Matem\'aticas, PUCP, Av. Universitaria 1801, San Miguel, Lima 32, Per\'u.}
\address{$^3$Instituto de Matem\'atica y Ciencias Afines (IMCA) Calle Los Bi\'ologos 245. Urb San C\'esar.
La Molina, Lima 12, Per\'u.}
\email{cvalqui@pucp.edu.pe}
\thanks{Christian Valqui was supported by PUCP-DGI-CAP-2020-818.}

\subjclass[2020]{primary 05C38; secondary 05C45}
\keywords{Hippchen's conjecture, three longest paths, traceable graph, intersection of longest paths}

\maketitle

\begin{abstract}
We prove that for a connected simple graph $G$ with $n\le 10$ vertices, and two longest paths
$C$ and $D$ in $G$, the intersection of vertex sets $V(C)\cap V(D)$ is a separator. This shows that the graph found previously with $n=11$, in which
the complement of the intersection of vertex sets $V(C)\cap V(D)$ of two longest paths is connected, is minimal.
\end{abstract}

\tableofcontents

\section{Introduction}

In~\cite{Gr} the following result was established. Let $k\in\{3,4,5\}$ and let $G$ be a 2-connected graph.
Suppose that $C$ and $D$ are two longest cycles of $G$ with $V(C)\ne V(D)$, meeting in a set $W$ of exactly
$k$ vertices. Then $W$ is a separator of $G$ (called an articulation set in~\cite{Gr}), which means that the
complement is not connected.   In~\cite{ST}  the
same result is proved for $k\in \{6,7 \}$. The result cannot be true for $k=8$, since the Petersen graph
has two longest cycles $C$ and $D$ of length 9 with $V(C)\ne V(D)$, such that the intersection has 8 vertices,
and $G\setminus(V(C)\cap V(D))$ is connected.

Using methods developed in order to approach the Hippchen conjecture, a path version of some of these results was proved recently
in~\cite{GV}*{Theorem 5.7 and Corollary 5.8}.
 Assume that $P$ and $Q$ are two longest paths in a simple graph $G$. If $V(Q)\ne V(P)$ and $V(P)\cap V(Q)$ has
cardinality $\ell \le 5$, then
it is a separator. Moreover,  if $V(Q)\ne V(P)$ and $n=|V(G)|\le 7$ then $V(Q)\cap V(P)$ is a separator. The following question was
raised in~\cite{GV}: Which are the maximal $\ell$ and $n$ such that the above results remain true?
Consider the following graph with 11 vertices, that has two longest path $P$ and $Q$ of length 9, which satisfy $V(Q)\ne V(P)$ and
moreover, the complement of $V(Q)\cap V(P)$ is connected. Since $\# (V(P)\cap V(Q))=9$ we have $n=11$ and $\ell=9$.

  \begin{tikzpicture}[scale=1]
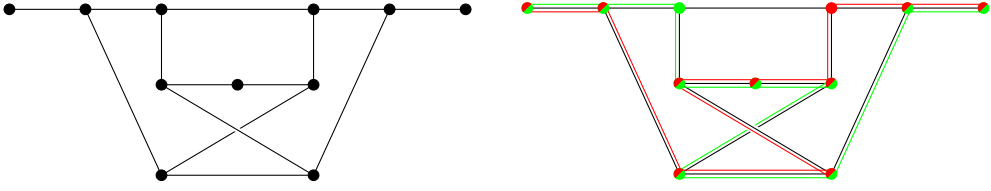
  \draw(3,0.2) node {Simple graph, $n=11$ vertices};
\draw(6,2.3) node {};
\filldraw [black]  (0,3)    circle (2pt)
[black]  (2,0.8)    circle (2pt)
[black]  (1,3)    circle (2pt)
[black]  (2,2)    circle (2pt)
[black]  (2,3)    circle (2pt)
[black]  (3,2)    circle (2pt)
[black]  (4,2)    circle (2pt)
[black]  (4,3)    circle (2pt)
[black]  (4,0.8)    circle (2pt)
[black]  (5,3)    circle (2pt)
[black]  (6,3)    circle (2pt);
\draw[-] (0,3)--(1,3);
\draw[-] (1,3)--(2,3);
\draw[-] (2,0.8)--(1,3);
\draw[-] (2,0.8)--(4,2);
\draw[-] (2,2)--(2,3);
\draw[-] (2,2)--(4,2);
\draw[-,white,line width=2pt] (2.5,1.7)--(3.5,1.1);
\draw[-] (2,2)--(4,0.8);
\draw[-] (4,2)--(4,3);
\draw[-] (4,3)--(5,3);
\draw[-] (5,3)--(6,3);
\draw[-] (4,0.8)--(5,3);
\draw[-] (2,3)--(4,3);
\draw[-] (2,0.8)--(4,0.8);
\draw (2,3.5) node {};
\end{tikzpicture}
  \begin{tikzpicture}[scale=1]
  \draw(3,0.2) node {Two longest paths, $\ell=\#(V(P)\cap V(Q))=9$};
\draw(6,2.3) node {};
\draw[-] (0,3)--(1,3);
\draw[-] (1,3)--(2,3);
\draw[-] (2,0.8)--(1,3);
\draw[-] (2,0.8)--(4,2);
\draw[-] (2,2)--(2,3);
\draw[-] (2,2)--(4,2);
\draw[-] (4,2)--(4,3);
\draw[-] (4,3)--(5,3);
\draw[-] (5,3)--(6,3);
\draw[-] (4,0.8)--(5,3);
\draw[-] (2,3)--(4,3);
\draw[-] (2,0.8)--(4,0.8);
\draw[-,red] (0,2.95)--(1,2.95);
\draw[-,red] (2.03,0.83)--(1.03,3.03);
\draw[-,red] (2,2.05)--(4,2.05);
\draw[-,green] (2,0.85)--(4,2.05);
\draw[-,white,line width=3pt] (2.5,1.68)--(3.5,1.08);
\draw[-] (2,2)--(4,0.8);
\draw[-,red] (2,1.95)--(4,0.75);
\draw[-,red] (3.95,2)--(3.95,3);
\draw[-,red] (4,3.05)--(5,3.05);
\draw[-,red] (5,3.05)--(6,3.05);
\draw[-,red] (2,0.85)--(4,0.85);
\draw[-,green] (0,3.05)--(1,3.05);
\draw[-,green] (1,3.05)--(2,3.05);
\draw[-,green] (2,0.75)--(4,0.75);
\draw[-,green] (1.95,2)--(1.95,3);
\draw[-,green] (2,1.95)--(4,1.95);
\draw[-,green] (5,2.95)--(6,2.95);
\draw[-,green] (4.05,0.8)--(5.05,3);
\filldraw[red]  (4,3)    circle (2pt);
\filldraw [green]  (0,3)    circle (2pt)
[green]  (2,0.8)    circle (2pt)
[green]  (1,3)    circle (2pt)
[green]  (2,2)    circle (2pt)
[green]  (2,3)    circle (2pt)
[green]  (3,2)    circle (2pt)
[green]  (4,2)    circle (2pt)
[green]  (4,0.8)    circle (2pt)
[green]  (5,3)    circle (2pt)
[green]  (6,3)    circle (2pt);
\draw[red,fill=red] (0.05,3.05) arc (45:225:.07cm);
\draw[red,fill=red] (1.05,3.05) arc (45:225:.07cm);
\draw[red,fill=red] (2.05,0.85) arc (45:225:.07cm);
\draw[red,fill=red] (2.05,2.05) arc (45:225:.07cm);
\draw[red,fill=red] (3.05,2.05) arc (45:225:.07cm);
\draw[red,fill=red] (4.05,0.85) arc (45:225:.07cm);
\draw[red,fill=red] (4.05,2.05) arc (45:225:.07cm);
\draw[red,fill=red] (5.05,3.05) arc (45:225:.07cm);
\draw[red,fill=red] (6.05,3.05) arc (45:225:.07cm);
\draw (2,3.5) node {};
\end{tikzpicture}

\noindent This shows that $n_{max}\le 10$ and $\ell_{max}\le 8$, and combining this with the results of~\cite{GV},
we obtain that $7\le n_{max} \le 10$ and $5\le \ell_{max} \le 8$.

In this article we will prove that $n_{max}=10$, as we have announced in~\cite{GV}.
In order to prove that $n_{max}=10$ we assume that there exists a graph $G$ with $n=V(G)\in\{8,9,10\}$, such that
$V(Q)\ne V(P)$ and such that $G\setminus(V(P)\cap V(Q))$ is connected, and we will arrive at a contradiction.
If $n=\ell +2$, then $\# (V(P)\setminus V(Q))=1$, $\# (V(Q)\setminus V(P))=1$ and $V(G)=V(P)\cup V(Q)$. This yields three cases that
we will discard in section~\ref{seccion P' =1}:
\begin{itemize}
  \item $\ell=6$, $n=8$,
  \item $\ell=7$, $n=9$,
  \item $\ell=8$, $n=10$.
\end{itemize}

If $\# (V(P)\setminus V(Q))=1$ and $\# (V(Q)\setminus V(P))=1$ but $n> \ell+2$, then we write $\{p_0\}=(V(P)\setminus V(Q))$,
$\{q_0\}=(V(Q)\setminus V(P))$ and set $V_0=V(G)\setminus (V(P)\cup V(Q))$. We build a new graph $G_1$, deleting $V_0$ and
 all edges incident with vertices in $V_0$, and then adding a new edge connecting $p_0$ with $q_0$. In this new graph $P$
and $Q$ are still longest paths, $V(Q)\ne V(P)$ and $G_1\setminus(V(P)\cap V(Q))$ is connected, and we have the same
$\ell=|V(P)\cap V(Q)|$, but now we are in the case $n=\ell +2$.

Since $\# (V(P)\setminus V(Q))=\# (V(Q)\setminus V(P))$, the only remaining case is
\begin{itemize}
  \item $\ell=6$ and $n=10$, $\# (V(P)\setminus V(Q))=2$ and $\# (V(Q)\setminus V(P))=2$,
\end{itemize}
 which we will discard in section~\ref{seccion 6 mas 4}.

\section{The case $n=\ell +2$}
\label{seccion P' =1}
In this section we will we assume that there exists a graph $G$ with $n=V(G)\in\{8,9,10\}$ and two longest paths
$P$ and $Q$ with $\ell=\#(V(P)\cap V(Q))=n-2$, such that
$V(Q)\ne V(P)$ and $G\setminus(V(P)\cap V(Q))$ is connected.

We set $P'=V(P)\setminus V(Q)$ and $Q'=V(Q)\setminus V(P)$.
In this case we know already that $\# P'=1$, $\# Q'=1$ and $V(G)=V(P)\cup V(Q)$. We write $V(P)=\{p_1,\dots,p_{n-1}\}$, and assume that
in the path $P$ the vertices $p_i$ and $p_{i+1}$ are consecutive. Clearly $P'\ne \{p_1\}$ and $P'\ne \{p_{n-1}\}$. Otherwise, since there is an edge
from $P'$ to $Q'$, we could expand the path $P$ to a Hamiltonian path.

\begin{remark}\label{remark distancia minima}
Write $V(Q)\setminus V(P)=\{q\}$. By the same (symmetric) argument as above $q$ cannot be
and endpoint of $Q$. Note also that $q$ cannot connect directly with $p_1$ or $p_{n-1}$.
If $P'=\{p_i\}$, then by assumption there is an edge connecting $q$ with $p_i$. Since $q$ is not an endpoint of $Q$,
there are two vertices $p_j,p_k$ in $V(Q)\cap V(P)$ such that $q$ connects directly to them, and we can and will assume that $j<k$.
Note that
$$
|i-j|,|k-i|,|j-k|>1.
$$
In fact, if $|j-i|=1$, then we can replace the edge connecting $p_i$ with $p_j$ by the path $p_i q p_j$ in the path $P$ and obtain a longer path.
Similarly, if $|k-i|=1$ then we can replace the edge connecting $p_i$ with $p_k$ in $P$ by the path $p_i q p_k$, and obtain a longer path, and
if $|k-j|=1$ then we can replace the edge connecting $p_k$ with $p_j$ in $P$ by the path $p_k q p_j$, and obtain a longer path.
\end{remark}

\begin{proposition}\label{prop no 2 ni n-2}
  We can assume that $P'\ne \{p_2\}$ (and by symmetry $P'\ne \{p_{n-2}\}$).
\end{proposition}

\begin{proof}
  Assume that $V(P)\setminus V(Q)=P'=\{p_2\}$. We will use the lollipop method in order to replace the path $P$
  with another longest path $\widetilde P$ with $V(\widetilde P)\ne V(Q)$ such that $G\setminus(V(P)\cap V(Q))$ is connected, and such that
  $V(\widetilde{P})\setminus V(Q)\ne \{\tilde p_2\},\{\tilde p_{n-2}\}$.

   By Remark~\ref{remark distancia minima} the vertex $q$ connects directly with
  vertices $p_j$ and $p_k$. We can assume $j<k$ and then, again by Remark~\ref{remark distancia minima}, we have $5<k<n-1$. Set
  $$
  \widetilde P:=p_{k-1}p_{k-2}\dots p_{3} p_2 q p_k p_{k+1}\dots p_{n-2}p_{n-1}.
  $$

\noindent  \begin{tikzpicture}[scale=0.68]
  \draw(4.5,-0.2) node {$q$ connects with $p_j$ and $p_k$};
\draw(6,2.3) node {};
\draw[-] (0,1)--(1,3);
\draw[-] (1,3)--(2,1);
\draw[-] (2,1)--(3,1);
\draw[-] (3,1)--(4,1);
\draw[-] (5,1)--(6,1);
\draw[-] (6,1)--(7,1);
\draw[-] (8,1)--(9,1);
\draw[-,red] (3,1)--(4.5,3);
\draw[-,red] (4.5,3)--(6,1);
\draw[-,dotted] (1,3)..controls(2.75,3.5)..(4.5,3);
\filldraw[red]  (4.5,3)    circle (2pt);
\filldraw [black]  (0,1)    circle (2pt)
[black]  (1,3)    circle (2pt)
[black]  (2,1)    circle (2pt)
[black]  (3,1)    circle (2pt)
[black]  (4,1)    circle (2pt)
[black]  (5,1)    circle (2pt)
[black]  (6,1)    circle (2pt)
[black]  (7,1)    circle (2pt)
[black]  (8,1)    circle (2pt)
[black]  (9,1)    circle (2pt);
\draw (0,0.6) node {$p_1$};
\draw (1,3.4) node {$p_2$};
\draw (2,0.6) node {$p_3$};
\draw (3,0.6) node {$p_j$};
\draw (5,0.6) node {$p_{k-1}$};
\draw (6,0.6) node {$p_k$};
\draw (9,0.6) node {$p_{n-1}$};
\draw (4.5,3.4) node {$q$};
\draw (4.5,1) node {$\dots$};
\draw (7.5,1) node {$\dots$};
\draw (10.5,1) node {};
\end{tikzpicture}
\noindent  \begin{tikzpicture}[scale=0.68]
  \draw(4.5,-0.3) node {The path $\widetilde P$ in blue};
\draw(6,2.3) node {};
\draw[-] (0,1)--(1,3);
\draw[-,line width=2pt,cyan] (1,3)--(2,1);
\draw[-,line width=2pt,cyan] (2,1)--(3,1);
\draw[-,line width=2pt,cyan] (3,1)--(4,1);
\draw[-,line width=2pt,cyan] (4,1)--(5,1);
\draw[-] (5,1)--(6,1);
\draw[-,line width=2pt,cyan] (6,1)--(7,1);
\draw[-,line width=2pt,cyan] (7,1)--(8,1);
\draw[-,line width=2pt,cyan] (8,1)--(9,1);
\draw[-,red] (3,1)--(4.5,3);
\draw[-,line width=2pt,cyan] (4.5,3)--(6,1);
\draw[-,line width=2pt,cyan] (1,3)..controls(2.75,3.5)..(4.5,3);
\filldraw[red]  (4.5,3)    circle (2pt);
\filldraw [black]  (0,1)    circle (2pt)
[black]  (1,3)    circle (2pt)
[black]  (2,1)    circle (2pt)
[black]  (3,1)    circle (2pt)
[black]  (4,1)    circle (2pt)
[black]  (6,1)    circle (2pt)
[black]  (7,1)    circle (2pt)
[black]  (8,1)    circle (2pt);
\filldraw[cyan]  (9,1)    circle (2.5pt)
[cyan]  (5,1)    circle (2.5pt);
\draw (0,0.6) node {$p_1$};
\draw (1,3.4) node {$p_2$};
\draw (2,0.6) node {$p_3$};
\draw (3,0.6) node {$p_j$};
\draw (5,0.6) node {$p_{k-1}$};
\draw (6,0.6) node {$p_k$};
\draw (9,0.6) node {$p_{n-1}$};
\draw (4.5,3.4) node {$q$};
\draw (4.5,1) node {$\dots$};
\draw (7.5,1) node {$\dots$};
\end{tikzpicture}

  Then $\widetilde P$ is also a longest path with $n-1$ vertices, $V(\widetilde{P})\setminus V(Q)=\{p_2\}$ is connected with
  $\{p_1\}=V(Q)\setminus V(\widetilde P)$, and so $G\setminus (V(\widetilde P)\cap V(Q)$ is connected. Moreover,
  since $\tilde p_1=p_{k-1}$, $\tilde p_2=p_{k-2}$ and $k>5$ we have $k-2>3>2$, and so
  $$
  V(\widetilde{P})\setminus V(Q)=\{p_2\}\ne\{p_{k-2}\}= \{\tilde p_2\}.
  $$
  Since $n>4$ we have $2<n-2$, and so
  $$
  V(\widetilde{P})\setminus V(Q)=\{p_2\}\ne\{p_{n-2}\}= \{\tilde p_{n-2}\},
  $$
  which concludes the proof.
\end{proof}

There are some pairs of vertices of $P$ that cannot be connected by an edge of $Q$ without generating a Hamiltonian path.
The following proposition collects some of the forbidden pairs in $\{p_1,p_{i-1},p_{i+1},p_{j-1},p_{j+1},p_{k-1},p_{k+1},p_{n-1}\}$

\begin{proposition} \label{forbidden pairs}
  The following pairs cannot be connected by an edge in $Q$.
  \begin{enumerate}
    \item $(p_{i-1},p_{i+1})$, $(p_1,p_{n-1})$, $(p_1,p_{i+1})$, $(p_{i-1},p_{n-1})$.
    \item For $x\in\{j,k\}$, the pairs $(p_{1},p_{x+1})$, $(p_{x-1},p_{n-1})$, $(p_{i-1},p_{x-1})$, $(p_{i+1},p_{x+1})$.
    \item For $x\in\{j,k\}$ with $x>i$, the pair $(p_1,p_{x-1})$ is forbidden, and for $x\in\{j,k\}$ with $x<i$, the pair $(p_{x+1},p_{n-1})$
    is forbidden.
    \item If $j<i$, the pair $(p_1,p_{i-1})$ is forbidden, and if $k>i$, the pair $(p_{i+1},p_{n-1})$ is forbidden.
  \end{enumerate}
\end{proposition}

\begin{proof}
  \noindent (1). There cannot be an edge of $Q$ connecting $p_{i-1}$ with $p_{i+1}$, since then we could replace that edge
  with the path $p_{i-1}p_i p_{i+1}$ in $Q$ and obtain a longer path. The Hamiltonian paths in each of following three diagrams show that
  none of the pairs $(p_1,p_{i+1})$, $(p_1,p_{n-1})$, $(p_{i-1},p_{n-1})$ are allowed.

\noindent
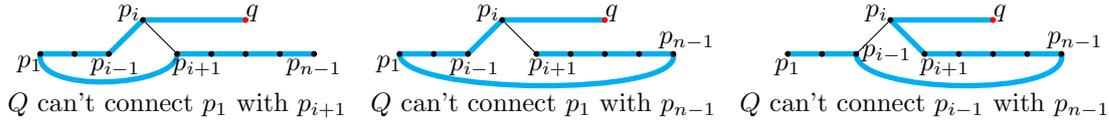
\begin{tikzpicture}[scale=0.45]
\draw(6,2.3) node {};
  \draw(4,-0.5) node {$Q$ can't connect $p_1$ with $p_{i+1}$};
\draw[-,line width=2pt,cyan] (0,1)--(1,1);
\draw[-,line width=2pt,cyan] (1,1)--(2,1);
\draw[-,line width=2pt,cyan] (2,1)--(3,2);
\draw[-] (3,2)--(4,1);
\draw[-,line width=2pt,cyan] (4,1)--(5,1);
\draw[-,line width=2pt,cyan] (5,1)--(6,1);
\draw[-,line width=2pt,cyan] (6,1)--(7,1);
\draw[-,line width=2pt,cyan] (7,1)--(8,1);
\draw[-,line width=2pt,cyan] (3,2)--(6,2);
\draw[-,line width=2pt,cyan] (0,1)..controls(-0.2,-0.2)and(4,0)..(4,1);
\filldraw[red]  (6,2)    circle (2pt);
\filldraw [black]  (0,1)    circle (2pt)
[black]  (1,1)    circle (2pt)
[black]  (2,1)    circle (2pt)
[black]  (3,2)    circle (2pt)
[black]  (4,1)    circle (2pt)
[black]  (5,1)    circle (2pt)
[black]  (6,1)    circle (2pt)
[black]  (7,1)    circle (2pt)
[black]  (8,1)    circle (2pt);
\draw (-0.3,0.7) node {$p_1$};
\draw (2.2,0.6) node {$p_{i-1}$};
\draw (2.6,2.2) node {$p_i$};
\draw (4.6,0.6) node {$p_{i+1}$};
\draw (8,0.6) node {$p_{n-1}$};
\draw (6.2,2.2) node {$q$};
\draw (8.5,3) node {};
\end{tikzpicture}
\begin{tikzpicture}[scale=0.45]
\draw(6,2.3) node {};
\draw(4.2,-0.5) node {$Q$ can't connect $p_1$ with $p_{n-1}$};
\draw[-,line width=2pt,cyan] (0,1)--(1,1);
\draw[-,line width=2pt,cyan] (1,1)--(2,1);
\draw[-,line width=2pt,cyan] (2,1)--(3,2);
\draw[-] (3,2)--(4,1);
\draw[-,line width=2pt,cyan] (4,1)--(5,1);
\draw[-,line width=2pt,cyan] (5,1)--(6,1);
\draw[-,line width=2pt,cyan] (6,1)--(7,1);
\draw[-,line width=2pt,cyan] (7,1)--(8,1);
\draw[-,line width=2pt,cyan] (3,2)--(6,2);
\draw[-,line width=2pt,cyan] (0,1)..controls(0,-0.2)and(8.3,-0.3)..(8,1);
\filldraw[red]  (6,2)    circle (2pt);
\filldraw [black]  (0,1)    circle (2pt)
[black]  (1,1)    circle (2pt)
[black]  (2,1)    circle (2pt)
[black]  (3,2)    circle (2pt)
[black]  (4,1)    circle (2pt)
[black]  (5,1)    circle (2pt)
[black]  (6,1)    circle (2pt)
[black]  (7,1)    circle (2pt)
[black]  (8,1)    circle (2pt);
\draw (-0.3,0.7) node {$p_1$};
\draw (2.2,0.6) node {$p_{i-1}$};
\draw (2.6,2.2) node {$p_i$};
\draw (4.5,0.6) node {$p_{i+1}$};
\draw (8.4,1.4) node {$p_{n-1}$};
\draw (6.2,2.2) node {$q$};
\draw (8.5,3) node {};
\end{tikzpicture}
\begin{tikzpicture}[scale=0.45]
\draw(6,2.3) node {};
\draw(4,-0.5) node {$Q$ can't connect $p_{i-1}$ with $p_{n-1}$};
\draw[-,line width=2pt,cyan] (0,1)--(1,1);
\draw[-,line width=2pt,cyan] (1,1)--(2,1);
\draw[-] (2,1)--(3,2);
\draw[-,line width=2pt,cyan] (3,2)--(4,1);
\draw[-,line width=2pt,cyan] (4,1)--(5,1);
\draw[-,line width=2pt,cyan] (5,1)--(6,1);
\draw[-,line width=2pt,cyan] (6,1)--(7,1);
\draw[-,line width=2pt,cyan] (7,1)--(8,1);
\draw[-,line width=2pt,cyan] (3,2)--(6,2);
\draw[-,line width=2pt,cyan] (2,1)..controls(2,-0.2)and(8.3,-0.3)..(8,1);
\filldraw[red]  (6,2)    circle (2pt);
\filldraw [black]  (0,1)    circle (2pt)
[black]  (1,1)    circle (2pt)
[black]  (2,1)    circle (2pt)
[black]  (3,2)    circle (2pt)
[black]  (4,1)    circle (2pt)
[black]  (5,1)    circle (2pt)
[black]  (6,1)    circle (2pt)
[black]  (7,1)    circle (2pt)
[black]  (8,1)    circle (2pt);
\draw (0,0.6) node {$p_1$};
\draw (2.9,1) node {$p_{i-1}$};
\draw (2.6,2.2) node {$p_i$};
\draw (4.6,0.6) node {$p_{i+1}$};
\draw (8.4,1.4) node {$p_{n-1}$};
\draw (6.2,2.2) node {$q$};
\draw (8.5,3) node {};
\end{tikzpicture}

  \noindent (2). The path $q p_x p_{x-1} p_{x-2}\dots p_2 p_1 p_{x+1} p_{x+2}\dots p_{n-1}$ shows that $(p_{1},p_{x+1})$ is forbidden, and
  the path $p_1 p_2 \dots p_{x-2}p_{x-1}p_{n-1} p_{n-2}\dots p_{x+2} p_{x+1} p_x q$ discards the pair $(p_{x-1},p_{n-1})$.

  If $x<i$, then the Hamiltonian paths in the following two diagrams show that $(p_{x-1},p_{i-1})$, $(p_{x+1},p_{i+1})$ are not allowed in that case,

\noindent
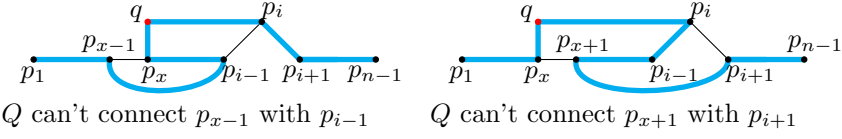
\begin{tikzpicture}[scale=0.5]
\draw(6,2.3) node {};
  \draw(4,-0.5) node {$Q$ can't connect $p_{x-1}$ with $p_{i-1}$};
\draw[-,line width=2pt,cyan] (0,1)--(2,1);
\draw[-,line width=2pt,cyan] (3,1)--(5,1);
\draw[-,line width=2pt,cyan] (3,1)--(3,2);
\draw[-] (2,1)--(3,1);
\draw[-] (5,1)--(6,2);
\draw[-,line width=2pt,cyan] (6,2)--(7,1);
\draw[-,line width=2pt,cyan] (7,1)--(9,1);
\draw[-,line width=2pt,cyan] (7,1)--(8,1);
\draw[-,line width=2pt,cyan] (3,2)--(6,2);
\draw[-,line width=2pt,cyan] (2,1)..controls(1.8,-0.2)and(5,0)..(5,1);
\filldraw[red]  (3,2)    circle (2pt);
\filldraw [black]  (0,1)    circle (2pt)
[black]  (2,1)    circle (2pt)
[black]  (3,1)    circle (2pt)
[black]  (5,1)    circle (2pt)
[black]  (6,2)    circle (2pt)
[black]  (7,1)    circle (2pt)
[black]  (9,1)    circle (2pt);
\draw (0,0.6) node {$p_1$};
\draw (2,1.4) node {$p_{x-1}$};
\draw (3.2,0.6) node {$p_{x}$};
\draw (5.6,0.6) node {$p_{i-1}$};
\draw (6.3,2.3) node {$p_i$};
\draw (7.2,0.6) node {$p_{i+1}$};
\draw (9,0.6) node {$p_{n-1}$};
\draw (2.7,2.2) node {$q$};
\draw (8.5,3) node {};
\end{tikzpicture}
\begin{tikzpicture}[scale=0.5]
\draw(6,2.3) node {};
  \draw(4,-0.5) node {$Q$ can't connect $p_{x+1}$ with $p_{i+1}$};
\draw[-,line width=2pt,cyan] (0,1)--(2,1);
\draw[-,line width=2pt,cyan] (3,1)--(5,1);
\draw[-,line width=2pt,cyan] (2,1)--(2,2);
\draw[-] (2,1)--(3,1);
\draw[-] (6,2)--(7,1);
\draw[-,line width=2pt,cyan] (5,1)--(6,2);
\draw[-,line width=2pt,cyan] (7,1)--(9,1);
\draw[-,line width=2pt,cyan] (7,1)--(8,1);
\draw[-,line width=2pt,cyan] (2,2)--(6,2);
\draw[-,line width=2pt,cyan] (3,1)..controls(2.8,-0.2)and(7,0)..(7,1);
\filldraw[red]  (2,2)    circle (2pt);
\filldraw [black]  (0,1)    circle (2pt)
[black]  (2,1)    circle (2pt)
[black]  (3,1)    circle (2pt)
[black]  (5,1)    circle (2pt)
[black]  (6,2)    circle (2pt)
[black]  (7,1)    circle (2pt)
[black]  (9,1)    circle (2pt);
\draw (0,0.6) node {$p_1$};
\draw (2,0.6) node {$p_{x}$};
\draw (3.2,1.4) node {$p_{x+1}$};
\draw (5.6,0.6) node {$p_{i-1}$};
\draw (6.3,2.3) node {$p_i$};
\draw (7.6,0.6) node {$p_{i+1}$};
\draw (9.3,1.4) node {$p_{n-1}$};
\draw (1.7,2.2) node {$q$};
\draw (8.5,3) node {};
\end{tikzpicture}

\noindent  and the Hamiltonian paths in the following two diagrams show that $(p_{i-1},p_{x-1})$, $(p_{i+1},p_{x+1})$ are impossible if $x>i$,
  which concludes item~(2).

  \noindent
  
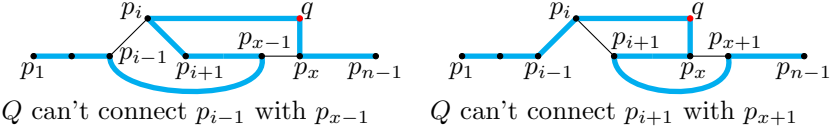
\begin{tikzpicture}[scale=0.5]
\draw(6,2.3) node {};
\draw(4,-0.5) node {$Q$ can't connect $p_{i-1}$ with $p_{x-1}$};
\draw[-,line width=2pt,cyan] (0,1)--(1,1);
\draw[-,line width=2pt,cyan] (1,1)--(2,1);
\draw[-] (2,1)--(3,2);
\draw[-,line width=2pt,cyan] (3,2)--(4,1);
\draw[-,line width=2pt,cyan] (4,1)--(5,1);
\draw[-,line width=2pt,cyan] (5,1)--(6,1);
\draw[-] (6,1)--(7,1);
\draw[-,line width=2pt,cyan] (7,1)--(9,1);
\draw[-,line width=2pt,cyan] (7,2)--(7,1);
\draw[-,line width=2pt,cyan] (3,2)--(7,2);
\draw[-,line width=2pt,cyan] (2,1)..controls(2,-0.2)and(6.3,-0.3)..(6,1);
\filldraw[red]  (7,2)    circle (2pt);
\filldraw [black]  (0,1)    circle (2pt)
[black]  (1,1)    circle (2pt)
[black]  (2,1)    circle (2pt)
[black]  (3,2)    circle (2pt)
[black]  (4,1)    circle (2pt)
[black]  (6,1)    circle (2pt)
[black]  (7,1)    circle (2pt)
[black]  (9,1)    circle (2pt);
\draw (0,0.6) node {$p_1$};
\draw (2.9,1) node {$p_{i-1}$};
\draw (2.6,2.2) node {$p_i$};
\draw (4.4,0.6) node {$p_{i+1}$};
\draw (6.1,1.4) node {$p_{x-1}$};
\draw (7.2,0.6) node {$p_x$};
\draw (9,0.6) node {$p_{n-1}$};
\draw (7.2,2.2) node {$q$};
\draw (8.5,3) node {};
\end{tikzpicture}
  \begin{tikzpicture}[scale=0.5]
\draw(6,2.3) node {};
\draw(4,-0.5) node {$Q$ can't connect $p_{i+1}$ with $p_{x+1}$};
\draw[-,line width=2pt,cyan] (0,1)--(1,1);
\draw[-,line width=2pt,cyan] (1,1)--(2,1);
\draw[-] (4,1)--(3,2);
\draw[-,line width=2pt,cyan] (3,2)--(2,1);
\draw[-,line width=2pt,cyan] (4,1)--(5,1);
\draw[-,line width=2pt,cyan] (5,1)--(6,1);
\draw[-] (6,1)--(7,1);
\draw[-,line width=2pt,cyan] (7,1)--(9,1);
\draw[-,line width=2pt,cyan] (6,2)--(6,1);
\draw[-,line width=2pt,cyan] (3,2)--(6,2);
\draw[-,line width=2pt,cyan] (4,1)..controls(4,-0.2)and(7.3,-0.3)..(7,1);
\filldraw[red]  (6,2)    circle (2pt);
\filldraw [black]  (0,1)    circle (2pt)
[black]  (1,1)    circle (2pt)
[black]  (2,1)    circle (2pt)
[black]  (3,2)    circle (2pt)
[black]  (4,1)    circle (2pt)
[black]  (6,1)    circle (2pt)
[black]  (7,1)    circle (2pt)
[black]  (9,1)    circle (2pt);
\draw (0,0.6) node {$p_1$};
\draw (2.3,0.6) node {$p_{i-1}$};
\draw (2.6,2.2) node {$p_i$};
\draw (4.6,1.4) node {$p_{i+1}$};
\draw (6.1,0.6) node {$p_{x}$};
\draw (7.2,1.4) node {$p_{x+1}$};
\draw (9,0.6) node {$p_{n-1}$};
\draw (6.2,2.2) node {$q$};
\draw (8.5,3) node {};
\end{tikzpicture}

  \noindent (3) The Hamiltonian paths in the following two diagrams show that for
  $x\in\{j,k\}$ with $x>i$, the pair $(p_1,p_{x-1})$ is forbidden, and for $x\in\{j,k\}$ with $x<i$, the pair $(p_{x+1},p_{n-1})$
    is forbidden.

  \noindent
  
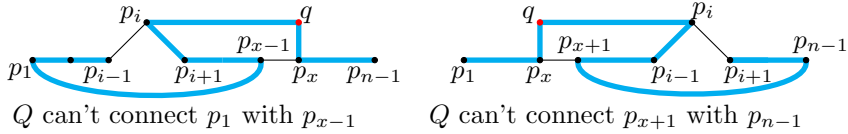
\begin{tikzpicture}[scale=0.5]
\draw(6,2.3) node {};
\draw(4,-0.5) node {$Q$ can't connect $p_{1}$ with $p_{x-1}$};
\draw[-,line width=2pt,cyan] (0,1)--(1,1);
\draw[-,line width=2pt,cyan] (1,1)--(2,1);
\draw[-] (2,1)--(3,2);
\draw[-,line width=2pt,cyan] (3,2)--(4,1);
\draw[-,line width=2pt,cyan] (4,1)--(5,1);
\draw[-,line width=2pt,cyan] (5,1)--(6,1);
\draw[-] (6,1)--(7,1);
\draw[-,line width=2pt,cyan] (7,1)--(9,1);
\draw[-,line width=2pt,cyan] (7,2)--(7,1);
\draw[-,line width=2pt,cyan] (3,2)--(7,2);
\draw[-,line width=2pt,cyan] (0,1)..controls(0,-0.2)and(6.3,-0.3)..(6,1);
\filldraw[red]  (7,2)    circle (2pt);
\filldraw [black]  (0,1)    circle (2pt)
[black]  (1,1)    circle (2pt)
[black]  (2,1)    circle (2pt)
[black]  (3,2)    circle (2pt)
[black]  (4,1)    circle (2pt)
[black]  (6,1)    circle (2pt)
[black]  (7,1)    circle (2pt)
[black]  (9,1)    circle (2pt);
\draw (-0.3,0.7) node {$p_1$};
\draw (2,0.6) node {$p_{i-1}$};
\draw (2.6,2.2) node {$p_i$};
\draw (4.4,0.6) node {$p_{i+1}$};
\draw (6.1,1.4) node {$p_{x-1}$};
\draw (7.2,0.6) node {$p_x$};
\draw (9,0.6) node {$p_{n-1}$};
\draw (7.2,2.2) node {$q$};
\draw (8.5,3) node {};
\end{tikzpicture}
\begin{tikzpicture}[scale=0.5]
\draw(6,2.3) node {};
  \draw(4,-0.5) node {$Q$ can't connect $p_{x+1}$ with $p_{n-1}$};
\draw[-,line width=2pt,cyan] (0,1)--(2,1);
\draw[-,line width=2pt,cyan] (3,1)--(5,1);
\draw[-,line width=2pt,cyan] (2,1)--(2,2);
\draw[-] (2,1)--(3,1);
\draw[-] (6,2)--(7,1);
\draw[-,line width=2pt,cyan] (5,1)--(6,2);
\draw[-,line width=2pt,cyan] (7,1)--(9,1);
\draw[-,line width=2pt,cyan] (7,1)--(8,1);
\draw[-,line width=2pt,cyan] (2,2)--(6,2);
\draw[-,line width=2pt,cyan] (3,1)..controls(2.8,-0.2)and(9.2,-0.2)..(9,1);
\filldraw[red]  (2,2)    circle (2pt);
\filldraw [black]  (0,1)    circle (2pt)
[black]  (2,1)    circle (2pt)
[black]  (3,1)    circle (2pt)
[black]  (5,1)    circle (2pt)
[black]  (6,2)    circle (2pt)
[black]  (7,1)    circle (2pt)
[black]  (9,1)    circle (2pt);
\draw (0,0.6) node {$p_1$};
\draw (2,0.6) node {$p_{x}$};
\draw (3.2,1.4) node {$p_{x+1}$};
\draw (5.6,0.6) node {$p_{i-1}$};
\draw (6.3,2.3) node {$p_i$};
\draw (7.5,0.6) node {$p_{i+1}$};
\draw (9.4,1.4) node {$p_{n-1}$};
\draw (1.7,2.2) node {$q$};
\draw (8.5,3) node {};
\end{tikzpicture}

  \noindent (4) The Hamiltonian paths in the following two diagrams show that when $j<i$, the pair $(p_1,p_{i-1})$ is forbidden,
   and if $k>i$, the pair $(p_{i+1},p_{n-1})$ is forbidden, which concludes the proof.

\noindent
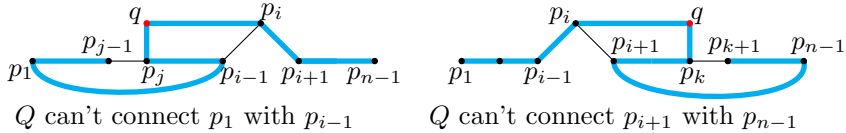
\begin{tikzpicture}[scale=0.5]
\draw(6,2.3) node {};
  \draw(4,-0.5) node {$Q$ can't connect $p_{1}$ with $p_{i-1}$};
\draw[-,line width=2pt,cyan] (0,1)--(2,1);
\draw[-,line width=2pt,cyan] (3,1)--(5,1);
\draw[-,line width=2pt,cyan] (3,1)--(3,2);
\draw[-] (2,1)--(3,1);
\draw[-] (5,1)--(6,2);
\draw[-,line width=2pt,cyan] (6,2)--(7,1);
\draw[-,line width=2pt,cyan] (7,1)--(9,1);
\draw[-,line width=2pt,cyan] (7,1)--(8,1);
\draw[-,line width=2pt,cyan] (3,2)--(6,2);
\draw[-,line width=2pt,cyan] (0,1)..controls(0,-0.2)and(5,0)..(5,1);
\filldraw[red]  (3,2)    circle (2pt);
\filldraw [black]  (0,1)    circle (2pt)
[black]  (2,1)    circle (2pt)
[black]  (3,1)    circle (2pt)
[black]  (5,1)    circle (2pt)
[black]  (6,2)    circle (2pt)
[black]  (7,1)    circle (2pt)
[black]  (9,1)    circle (2pt);
\draw (-0.3,0.7) node {$p_1$};
\draw (2,1.4) node {$p_{j-1}$};
\draw (3.2,0.6) node {$p_{j}$};
\draw (5.6,0.6) node {$p_{i-1}$};
\draw (6.3,2.3) node {$p_i$};
\draw (7.2,0.6) node {$p_{i+1}$};
\draw (9,0.6) node {$p_{n-1}$};
\draw (2.7,2.2) node {$q$};
\draw (8.5,3) node {};
\end{tikzpicture}
  \begin{tikzpicture}[scale=0.5]
\draw(6,2.3) node {};
\draw(4,-0.5) node {$Q$ can't connect $p_{i+1}$ with $p_{n-1}$};
\draw[-,line width=2pt,cyan] (0,1)--(1,1);
\draw[-,line width=2pt,cyan] (1,1)--(2,1);
\draw[-] (4,1)--(3,2);
\draw[-,line width=2pt,cyan] (3,2)--(2,1);
\draw[-,line width=2pt,cyan] (4,1)--(5,1);
\draw[-,line width=2pt,cyan] (5,1)--(6,1);
\draw[-] (6,1)--(7,1);
\draw[-,line width=2pt,cyan] (7,1)--(9,1);
\draw[-,line width=2pt,cyan] (6,2)--(6,1);
\draw[-,line width=2pt,cyan] (3,2)--(6,2);
\draw[-,line width=2pt,cyan] (4,1)..controls(4,-0.2)and(9.3,-0.3)..(9,1);
\filldraw[red]  (6,2)    circle (2pt);
\filldraw [black]  (0,1)    circle (2pt)
[black]  (1,1)    circle (2pt)
[black]  (2,1)    circle (2pt)
[black]  (3,2)    circle (2pt)
[black]  (4,1)    circle (2pt)
[black]  (6,1)    circle (2pt)
[black]  (7,1)    circle (2pt)
[black]  (9,1)    circle (2pt);
\draw (0,0.6) node {$p_1$};
\draw (2.3,0.6) node {$p_{i-1}$};
\draw (2.6,2.2) node {$p_i$};
\draw (4.6,1.4) node {$p_{i+1}$};
\draw (6.1,0.6) node {$p_{k}$};
\draw (7.2,1.4) node {$p_{k+1}$};
\draw (9.4,1.4) node {$p_{n-1}$};
\draw (6.2,2.2) node {$q$};
\draw (8.5,3) node {};
\end{tikzpicture}
\end{proof}

\subsection{The case $\ell=6$ and $n=8$}

Remember that $\{p_i\}=V(P)\setminus V(Q)$ and $\{q\}=V(Q)\setminus V(P)$. Since $i\notin\{1,n-1=7\}$ and by Proposition~\ref{prop no 2 ni n-2}
$i\notin\{2,n-2=6\}$, we know that $i\in\{3,4,5\}$.
By Remark~\ref{remark distancia minima} there exist $j,k$ such that $1<j<k<n-1$, and $|i-j|,|k-i|,|j-k|>1$.
The case $i=3$ is impossible, since then $3=i<j-1$ and $j+1<k$, and so $5=i+2<j+1<k$, thus $7\le k$, which contradicts $k<n-1=7$.
By symmetry, the case $i=5$ is also impossible.

Finally, we discard the case $i=4$. If $i<j<k$, then $4<j-1<k-3$ leads to $k\ge 8$, which is impossible, and if
$j<k<i$, then $j<k-1<i-3=1$ leads to the contradiction $j<0$. Thus $j<i<k$, and so $j=2$, $i=4$ and $k=6$ is the only case left.

So assume that $i=4$, $j=2$ and $k=6$. We know that $p_3$ cannot be an endpoint of $Q$, since then we could continue $Q$ with the edge $p_3p_4$.

  Thus two edges of the path $Q$ connect $p_3$ directly with two vertices in the set $\{p_1,p_2,p_5,p_6,p_7\}$.
  By Proposition~\ref{forbidden pairs}(1) an edge of $Q$ cannot connect $p_3=p_{i-1}$ with $p_{i+1}=p_5$ nor $p_3$ with $p_{n-1}=p_7$,
  and by item~(4) of the same proposition an edge of $Q$ cannot connect $p_3=p_{i-1}$ with $p_1$, since $j=2<i$.

Thus two edges of the path $Q$ connect $p_3$ with $p_2$ and $p_3$ with $p_6$, which implies that $Q$ contains the 4-cycle $p_2 p_3 p_6 q p_3$,
a contradiction that discards
the case $i=4$ and finishes the proof of the case $\ell=6$ and $n=8$.

\noindent 
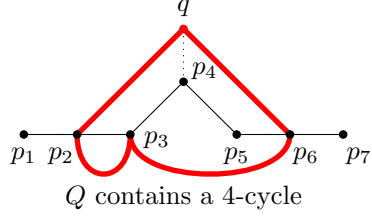
\begin{tikzpicture}[scale=0.7]
  \draw(3,-0.2) node {$Q$ contains a 4-cycle};
\draw(-6,2.3) node {};
\draw[-] (0,1)--(1,1);
\draw[-] (1,1)--(2,1);
\draw[-] (2,1)--(3,2);
\draw[-] (3,2)--(4,1);
\draw[-] (4,1)--(5,1);
\draw[-] (5,1)--(6,1);
\draw[-,red,line width=2pt] (1,1)--(3,3);
\draw[-,red,line width=2pt] (3,3)--(5,1);
\draw[-,red,line width=2pt] (1,1)..controls(1,0)and(2,0)..(2,1);
\draw[-,red,line width=2pt] (2,1)..controls(2,0)and(5,0)..(5,1);
\draw[-,dotted] (3,2)..controls(3,2.5)..(3,3);
\filldraw[red]  (3,3)    circle (2pt);
\filldraw [black]  (0,1)    circle (2pt)
[black]  (1,1)    circle (2pt)
[black]  (2,1)    circle (2pt)
[black]  (3,2)    circle (2pt)
[black]  (4,1)    circle (2pt)
[black]  (5,1)    circle (2pt)
[black]  (6,1)    circle (2pt);
\draw (0,0.6) node {$p_1$};
\draw (0.7,0.6) node {$p_2$};
\draw (2.5,0.9) node {$p_3$};
\draw (3.4,2.2) node {$p_4$};
\draw (4,0.6) node {$p_{5}$};
\draw (5.3,0.6) node {$p_6$};
\draw (6.3,0.6) node {$p_7$};
\draw (3,3.4) node {$q$};
\draw (10.5,1) node {};
\end{tikzpicture}

\subsection{The case $\ell=7$ and $n=9$}
Remember that $\{p_i\}=V(P)\setminus V(Q)$ and $\{q\}=V(Q)\setminus V(P)$. Since $i\notin\{1,n-1=8\}$ and by Proposition~\ref{prop no 2 ni n-2}
$i\notin\{2,n-2=7\}$, we know that $i\in\{3,4,5,6\}$. By symmetry it suffices to discard the cases $i=3$ and $i=4$.
There are two vertices $p_j,p_k$ in $V(Q)\cap V(P)$ that are connected with $q$ via edges of $Q$, with $j<k$.

\begin{itemize}
  \item Assume that $i=3$.
\end{itemize}

By Remark~\ref{remark distancia minima} we know that $1<j,k<n-1$ and $|i-j|,|k-i|,|j-k|>1$.
The case $j<i=3$ is impossible, since then $j<i-1=2$ leads to $j=1$, which is impossible. Hence $i<j<k$, and the only possibility is
$$
i=3,j=5,k=7.
$$
We know that $p_4$ cannot be an endpoint of $Q$, since then we could continue $Q$ with the edge $p_4p_3$.
Thus two edges of the path $Q$ must connect $p_4$ directly with two vertices in the set $\{p_1,p_2,p_5,p_6,p_7,p_8\}$.
  By Proposition~\ref{forbidden pairs} an edge of $Q$ cannot connect $p_4=p_{i+1}$ with $p_1$, $p_2=p_{i-1}$, $p_6=p_{j+1}$ nor $p_8=p_{n-1}$.

Thus two edges of the path $Q$ connect $p_4$ with $p_5$ and $p_4$ with $p_7$, which implies that $Q$ contains the 4-cycle $p_4 p_5 q p_7 p_4$,
a contradiction that discards
the case $i=3$.

\noindent  
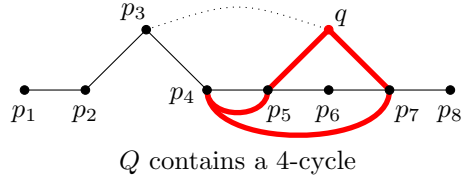
\begin{tikzpicture}[scale=0.8]
  \draw(3.5,-0.2) node {$Q$ contains a 4-cycle};
\draw(-6,2.3) node {};
\draw[-] (0,1)--(1,1);
\draw[-] (1,1)--(2,2);
\draw[-] (2,2)--(3,1);
\draw[-] (3,1)--(4,1);
\draw[-] (4,1)--(5,1);
\draw[-] (5,1)--(6,1);
\draw[-] (6,1)--(7,1);
\draw[-,red,line width=2pt] (4,1)--(5,2);
\draw[-,red,line width=2pt] (5,2)--(6,1);
\draw[-,red,line width=2pt] (3,1)..controls(3,0)and(6,0)..(6,1);
\draw[-,red,line width=2pt] (3,1)..controls(3,0.5)and(4,0.5)..(4,1);
\draw[-,dotted] (2,2)..controls(3.5,2.5)..(5,2);
\filldraw[red]  (5,2)    circle (2pt);
\filldraw [black]  (0,1)    circle (2pt)
[black]  (1,1)    circle (2pt)
[black]  (2,2)    circle (2pt)
[black]  (3,1)    circle (2pt)
[black]  (4,1)    circle (2pt)
[black]  (5,1)    circle (2pt)
[black]  (6,1)    circle (2pt)
[black]  (7,1)    circle (2pt);
\draw (0,0.6) node {$p_1$};
\draw (1,0.6) node {$p_2$};
\draw (1.8,2.3) node {$p_3$};
\draw (2.6,0.9) node {$p_4$};
\draw (4.2,0.6) node {$p_{5}$};
\draw (5,0.6) node {$p_6$};
\draw (6.3,0.6) node {$p_7$};
\draw (7,0.6) node {$p_8$};
\draw (5.2,2.2) node {$q$};
\draw (8.5,3) node {};
\end{tikzpicture}

\begin{itemize}
  \item Assume that $i=4$.
\end{itemize}

By Remark~\ref{remark distancia minima} we know that $1<j,k<n-1=8$ and $|i-j|,|k-i|,|j-k|>1$.
The case $i=4<j$ can be discarded, since then $6\le j$ and $8\le k$, which is impossible.
Hence $j<i<k$, and the only two possibilities are
$$
j=2,i=4,k=6\quad\text{or}\quad j=2,i=4,k=7.
$$

\begin{itemize}
  \item[$\bullet\bullet$] Assume first that $(i,j,k)=(4,2,6)$.
\end{itemize}
We know that $p_5$ cannot be an endpoint of $Q$, since then we could continue $Q$ with the edge $p_5p_4$.

  Thus two edges of the path $Q$ connect $p_5$ directly with two vertices in the set
  $$
  (V(P)\cap V(Q))\setminus\{p_5\}=\{p_1,p_2,p_3,p_6,p_7,p_8\}.
  $$
  By Proposition~\ref{forbidden pairs} an edge of $Q$ cannot connect $p_5=p_{i+1}$ with $p_1$, $p_3=p_{i-1}$, $p_7=p_{j+1}$ nor $p_8=p_{n-1}$.

Thus two edges of the path $Q$ connect $p_5$ with $p_2$ and $p_5$ with $p_6$, which implies that $Q$ contains the 4-cycle $p_2 p_5 p_6 q p_2$,
a contradiction that discards
the case $i=4$, $j=2$ and $k=6$.

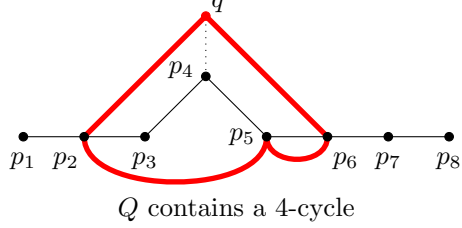
\begin{tikzpicture}[scale=0.8]
  \draw(3.5,-0.2) node {$Q$ contains a 4-cycle};
\draw(-6,2.3) node {};
\draw[-] (0,1)--(1,1);
\draw[-] (1,1)--(2,1);
\draw[-] (2,1)--(3,2);
\draw[-] (3,2)--(4,1);
\draw[-] (4,1)--(5,1);
\draw[-] (5,1)--(6,1);
\draw[-] (6,1)--(7,1);
\draw[-,red,line width=2pt] (1,1)--(3,3);
\draw[-,red,line width=2pt] (3,3)--(5,1);
\draw[-,red,line width=2pt] (4,1)..controls(4,0)and(1,0)..(1,1);
\draw[-,red,line width=2pt] (4,1)..controls(4,0.5)and(5,0.5)..(5,1);
\draw[-,dotted] (3,2)--(3,3);
\filldraw[red]  (3,3)    circle (2pt);
\filldraw [black]  (0,1)    circle (2pt)
[black]  (1,1)    circle (2pt)
[black]  (2,1)    circle (2pt)
[black]  (3,2)    circle (2pt)
[black]  (4,1)    circle (2pt)
[black]  (5,1)    circle (2pt)
[black]  (6,1)    circle (2pt)
[black]  (7,1)    circle (2pt);
\draw (0,0.6) node {$p_1$};
\draw (0.7,0.6) node {$p_2$};
\draw (2,0.6) node {$p_3$};
\draw (2.6,2.1) node {$p_4$};
\draw (3.6,1) node {$p_{5}$};
\draw (5.3,0.6) node {$p_6$};
\draw (6,0.6) node {$p_7$};
\draw (7,0.6) node {$p_8$};
\draw (3.2,3.2) node {$q$};
\draw (8.5,3) node {};
\end{tikzpicture}

\begin{itemize}
  \item[$\bullet\bullet$] Now assume that $(i,j,k)=(4,2,7)$.
\end{itemize}

We know that $p_3$ cannot be an endpoint of $Q$, since then we could continue $Q$ with the edge $p_3p_4$.

  Thus two edges of the path $Q$ connect $p_3$ directly with two vertices in the set
  $$
  (V(P)\cap V(Q))\setminus\{p_3\}=\{p_1,p_2,p_5,p_6,p_7,p_8\}.
  $$
   By Proposition~\ref{forbidden pairs} an edge of $Q$ cannot connect $p_3=p_{i-1}$ with $p_1$, $p_5=p_{i+1}$, $p_6=p_{k-1}$ nor $p_8=p_{n-1}$.

Thus two edges of the path $Q$ connect $p_3$ with $p_2$ and $p_3$ with $p_7$, which implies that $Q$ contains the 4-cycle $p_2 p_3 p_7 q p_2$,
a contradiction that discards
the case $i=4$, $j=2$ and $k=7$, finishing the case $i=4$ and thus we have proved that $\ell=7$ and $n=9$ is impossible.

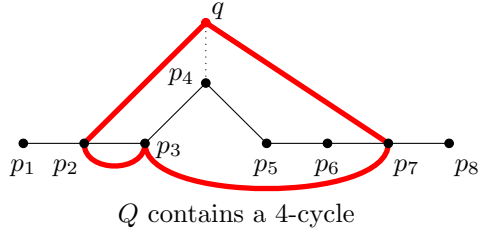
\begin{tikzpicture}[scale=0.8]
  \draw(3.5,-0.2) node {$Q$ contains a 4-cycle};
\draw(-6,2.3) node {};
\draw[-] (0,1)--(1,1);
\draw[-] (1,1)--(2,1);
\draw[-] (2,1)--(3,2);
\draw[-] (3,2)--(4,1);
\draw[-] (4,1)--(5,1);
\draw[-] (5,1)--(6,1);
\draw[-] (6,1)--(7,1);
\draw[-,red,line width=2pt] (1,1)--(3,3);
\draw[-,red,line width=2pt] (3,3)--(6,1);
\draw[-,red,line width=2pt] (1,1)..controls(1,0.5)and(2,0.5)..(2,1);
\draw[-,red,line width=2pt] (2,1)..controls(2,0)and(6,0)..(6,1);
\draw[-,dotted] (3,2)--(3,3);
\filldraw[red]  (3,3)    circle (2pt);
\filldraw [black]  (0,1)    circle (2pt)
[black]  (1,1)    circle (2pt)
[black]  (2,1)    circle (2pt)
[black]  (3,2)    circle (2pt)
[black]  (4,1)    circle (2pt)
[black]  (5,1)    circle (2pt)
[black]  (6,1)    circle (2pt)
[black]  (7,1)    circle (2pt);
\draw (0,0.6) node {$p_1$};
\draw (0.7,0.6) node {$p_2$};
\draw (2.4,0.9) node {$p_3$};
\draw (2.6,2.1) node {$p_4$};
\draw (4,0.6) node {$p_{5}$};
\draw (5,0.6) node {$p_6$};
\draw (6.3,0.6) node {$p_7$};
\draw (7.3,0.6) node {$p_8$};
\draw (3.2,3.2) node {$q$};
\draw (8.5,3) node {};
\end{tikzpicture}

\subsection{The case $\ell=8$ and $n=10$}
Remember that $\{p_i\}=V(P)\setminus V(Q)$ and $\{q\}=V(Q)\setminus V(P)$. Since $i\notin\{1,n-1=9\}$ and by Proposition~\ref{prop no 2 ni n-2}
$i\notin\{2,n-2=8\}$, we know that $i\in\{3,4,5,6,7\}$. By symmetry it suffices to discard the cases $i=3$, $i=4$ and $i=5$.
There are two vertices $p_j,p_k$ in $V(Q)\cap V(P)$ that are connected with $q$ via edges of $Q$, and we can and will assume that $j<k$.

\begin{itemize}
  \item Assume that $i=3$.
\end{itemize}

By Remark~\ref{remark distancia minima} we know that $1<j,k<n-1=9$ and $|i-j|,|k-i|,|j-k|>1$.
The case $j<i=3$ is impossible, since then $j<i-1=2$ leads to $j=1$, which we already discarded. Hence $i<j<k$, and the only three
possibilities are
$$
(i,j,k)=(3,5,7),\quad (i,j,k)=(3,5,8)\quad \text{and} \quad (i,j,k)=(3,6,8)
$$
\begin{itemize}
  \item[$\bullet\bullet$] Assume that $(i,j,k)=(3,5,7)$.
\end{itemize}

 We know that $p_4$ cannot be an endpoint of $Q$, since then we could continue $Q$ with the edge $p_4p_3$.

  Thus two edges of the path $Q$ connect $p_4$ directly with two vertices in the set
  $$
  (V(P)\cap V(Q))\setminus\{p_3\}=\{p_1,p_2,p_5,p_6,p_7,p_8,p_9\}.
  $$
   By Proposition~\ref{forbidden pairs} an edge of $Q$ cannot connect $p_4=p_{i+1}$ with $p_1$, $p_2=p_{i-1}$, $p_6=p_{j+1}$,
   $p_8=p_{k+1}$ nor $p_9=p_{n-1}$.

Thus two edges of the path $Q$ connect $p_4$ with $p_5$ and $p_4$ with $p_7$, which implies that $Q$ contains the 4-cycle $p_4 p_5 q p_7  p_4$,
a contradiction that discards
the case $i=3$, $j=5$ and $k=7$.

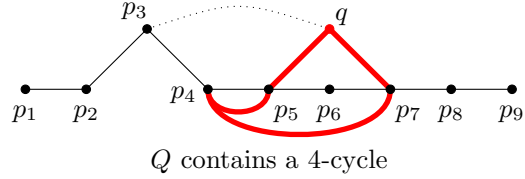
\begin{tikzpicture}[scale=0.8]
  \draw(4,-0.2) node {$Q$ contains a 4-cycle};
\draw(-6,2.3) node {};
\draw[-] (0,1)--(1,1);
\draw[-] (1,1)--(2,2);
\draw[-] (2,2)--(3,1);
\draw[-] (3,1)--(4,1);
\draw[-] (4,1)--(5,1);
\draw[-] (5,1)--(6,1);
\draw[-] (6,1)--(7,1);
\draw[-] (7,1)--(8,1);
\draw[-,line width=2pt,red] (4,1)--(5,2);
\draw[-,line width=2pt,red] (5,2)--(6,1);
\draw[-,line width=2pt,red] (3,1)..controls(3,0.5)and(4,0.5)..(4,1);
\draw[-,line width=2pt,red] (3,1)..controls(3,0)and(6,0)..(6,1);
\draw[-,dotted] (2,2)..controls(3.5,2.5)..(5,2);
\filldraw[red]  (5,2)    circle (2pt);
\filldraw [black]  (0,1)    circle (2pt)
[black]  (1,1)    circle (2pt)
[black]  (2,2)    circle (2pt)
[black]  (3,1)    circle (2pt)
[black]  (4,1)    circle (2pt)
[black]  (5,1)    circle (2pt)
[black]  (6,1)    circle (2pt)
[black]  (7,1)    circle (2pt)
[black]  (8,1)    circle (2pt);
\draw (0,0.6) node {$p_1$};
\draw (1,0.6) node {$p_2$};
\draw (1.8,2.3) node {$p_3$};
\draw (2.6,0.9) node {$p_4$};
\draw (4.3,0.6) node {$p_{5}$};
\draw (5,0.6) node {$p_6$};
\draw (6.3,0.6) node {$p_7$};
\draw (7,0.6) node {$p_8$};
\draw (8,0.6) node {$p_9$};
\draw (5.2,2.2) node {$q$};
\draw (8.5,3) node {};
\end{tikzpicture}

\begin{itemize}
  \item[$\bullet\bullet$] Assume that $(i,j,k)=(3,5,8)$.
\end{itemize}

We know that $p_4$ cannot be an endpoint of $Q$, since then we could continue $Q$ with the edge $p_4p_3$.
  Thus two edges of the path $Q$ connect $p_4$ directly with two vertices in the set
  $$
  (V(P)\cap V(Q))\setminus\{p_3\}=\{p_1,p_2,p_5,p_6,p_7,p_8,p_9\}.
  $$
 By Proposition~\ref{forbidden pairs} an edge of $Q$ cannot connect $p_4=p_{i+1}$ with $p_1$, $p_2=p_{i-1}$, $p_6=p_{j+1}$ nor $p_9=p_{n-1}$.

  The first diagram shows that if an edge of $Q$ connects $p_4$ with $p_7$, then there is a (Hamiltonian) path of length 9,
which contradicts that $P$ and $Q$ of length $n-2=8$ are longest paths.

\noindent
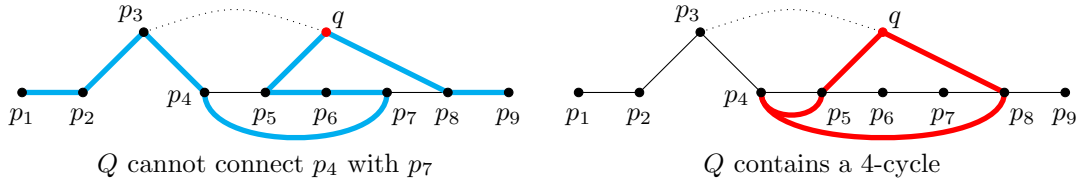
\begin{tikzpicture}[scale=0.8]
  \draw(4,-0.2) node {$Q$ cannot connect $p_4$ with $p_7$};
\draw(6,2.3) node {};
\draw[-,line width=2pt,cyan] (0,1)--(1,1);
\draw[-,line width=2pt,cyan] (1,1)--(2,2);
\draw[-,line width=2pt,cyan] (2,2)--(3,1);
\draw[-] (3,1)--(4,1);
\draw[-,line width=2pt,cyan] (4,1)--(5,1);
\draw[-,line width=2pt,cyan] (5,1)--(6,1);
\draw[-] (6,1)--(7,1);
\draw[-,line width=2pt,cyan] (7,1)--(8,1);
\draw[-,line width=2pt,cyan] (4,1)--(5,2);
\draw[-,line width=2pt,cyan] (5,2)--(7,1);
\draw[-,line width=2pt,cyan] (3,1)..controls(3,0)and(6,0)..(6,1);
\draw[-,dotted] (2,2)..controls(3.5,2.5)..(5,2);
\filldraw[red]  (5,2)    circle (2pt);
\filldraw [black]  (0,1)    circle (2pt)
[black]  (1,1)    circle (2pt)
[black]  (2,2)    circle (2pt)
[black]  (3,1)    circle (2pt)
[black]  (4,1)    circle (2pt)
[black]  (5,1)    circle (2pt)
[black]  (6,1)    circle (2pt)
[black]  (7,1)    circle (2pt)
[black]  (8,1)    circle (2pt);
\draw (0,0.6) node {$p_1$};
\draw (1,0.6) node {$p_2$};
\draw (1.8,2.3) node {$p_3$};
\draw (2.6,0.9) node {$p_4$};
\draw (4,0.6) node {$p_{5}$};
\draw (5,0.6) node {$p_6$};
\draw (6.3,0.6) node {$p_7$};
\draw (7,0.6) node {$p_8$};
\draw (8,0.6) node {$p_9$};
\draw (5.2,2.2) node {$q$};
\draw (8.5,3) node {};
\end{tikzpicture}
\begin{tikzpicture}[scale=0.8]
  \draw(4,-0.2) node {$Q$ contains a 4-cycle};
\draw(6,2.3) node {};
\draw[-] (0,1)--(1,1);
\draw[-] (1,1)--(2,2);
\draw[-] (2,2)--(3,1);
\draw[-] (3,1)--(4,1);
\draw[-] (4,1)--(5,1);
\draw[-] (5,1)--(6,1);
\draw[-] (6,1)--(7,1);
\draw[-] (7,1)--(8,1);
\draw[-,line width=2pt,red] (4,1)--(5,2);
\draw[-,line width=2pt,red] (5,2)--(7,1);
\draw[-,line width=2pt,red] (3,1)..controls(3,0.5)and(4,0.5)..(4,1);
\draw[-,line width=2pt,red] (3,1)..controls(3,0)and(7,0)..(7,1);
\draw[-,dotted] (2,2)..controls(3.5,2.5)..(5,2);
\filldraw[red]  (5,2)    circle (2pt);
\filldraw [black]  (0,1)    circle (2pt)
[black]  (1,1)    circle (2pt)
[black]  (2,2)    circle (2pt)
[black]  (3,1)    circle (2pt)
[black]  (4,1)    circle (2pt)
[black]  (5,1)    circle (2pt)
[black]  (6,1)    circle (2pt)
[black]  (7,1)    circle (2pt)
[black]  (8,1)    circle (2pt);
\draw (0,0.6) node {$p_1$};
\draw (1,0.6) node {$p_2$};
\draw (1.8,2.3) node {$p_3$};
\draw (2.6,0.9) node {$p_4$};
\draw (4.3,0.6) node {$p_{5}$};
\draw (5,0.6) node {$p_6$};
\draw (6,0.6) node {$p_7$};
\draw (7.3,0.6) node {$p_8$};
\draw (8,0.6) node {$p_9$};
\draw (5.2,2.2) node {$q$};
\draw (8.5,3) node {};
\end{tikzpicture}

Thus two edges of the path $Q$ connect $p_4$ with $p_5$ and $p_4$ with $p_8$, which implies that $Q$ contains the 4-cycle $p_4 p_5 q p_8  p_4$,
a contradiction that discards
the case $i=3$, $j=5$ and $k=8$.

\begin{itemize}
  \item[$\bullet\bullet$] Assume that $(i,j,k)=(3,6,8)$.
\end{itemize}

In this case it doesn't suffice to analyze the connections of a single vertex as in all the previous cases,
and we have to consider the connections of $p_7$ and $p_9$.
Proposition~\ref{forbidden pairs} and the Hamiltonian path in the following diagram show that an edge of $Q$ cannot connect $p_7$ with one of $\{p_1,p_2,p_4,p_5,p_9\}$,
and that an edge of $Q$ cannot connect $p_9$ with one of $\{p_1,p_2,p_4,p_5,p_7\}$.

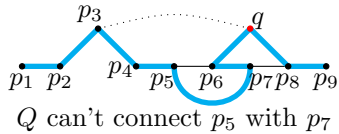
\noindent \begin{tikzpicture}[scale=0.5]
  \draw(4,-0.4) node {$Q$ can't connect $p_5$ with $p_7$};
\draw(6,2.3) node {};
\draw[-,line width=2pt,cyan] (0,1)--(1,1);
\draw[-,line width=2pt,cyan] (1,1)--(2,2);
\draw[-,line width=2pt,cyan] (2,2)--(3,1);
\draw[-,line width=2pt,cyan] (3,1)--(4,1);
\draw[-] (4,1)--(5,1);
\draw[-,line width=2pt,cyan] (5,1)--(6,1);
\draw[-] (6,1)--(7,1);
\draw[-,line width=2pt,cyan] (7,1)--(8,1);
\draw[-,line width=2pt,cyan] (5,1)--(6,2);
\draw[-,line width=2pt,cyan] (6,2)--(7,1);
\draw[-,line width=2pt,cyan] (4,1)..controls(4,-0.3)and(6,-0.3)..(6,1);
\draw[-,dotted] (2,2)..controls(4,2.5)..(6,2);
\filldraw[red]  (6,2)    circle (2pt);
\filldraw [black]  (0,1)    circle (2pt)
[black]  (1,1)    circle (2pt)
[black]  (2,2)    circle (2pt)
[black]  (3,1)    circle (2pt)
[black]  (4,1)    circle (2pt)
[black]  (5,1)    circle (2pt)
[black]  (6,1)    circle (2pt)
[black]  (7,1)    circle (2pt)
[black]  (8,1)    circle (2pt);
\draw (0,0.6) node {$p_1$};
\draw (1,0.6) node {$p_2$};
\draw (1.8,2.4) node {$p_3$};
\draw (2.6,0.9) node {$p_4$};
\draw (3.7,0.6) node {$p_{5}$};
\draw (5,0.6) node {$p_6$};
\draw (6.3,0.6) node {$p_7$};
\draw (7,0.6) node {$p_8$};
\draw (8,0.6) node {$p_9$};
\draw (6.2,2.2) node {$q$};
\draw (8.5,3) node {};
\end{tikzpicture}

Thus $p_7$ can be connected only with $p_6$ and $p_8$ by edges of $Q$, and similarly
$p_9$ can be connected only with $p_6$ and $p_8$ by edges of $Q$.

Hence, if $p_7$ is not an endpoint of $Q$, then $Q$ contains the 4-cycle $p_6 p_7 p_8 q p_6$, and similarly,
if $p_9$ is not an endpoint of $Q$, then $Q$ contains the 4-cycle $p_6 p_9 p_8 q p_6$.

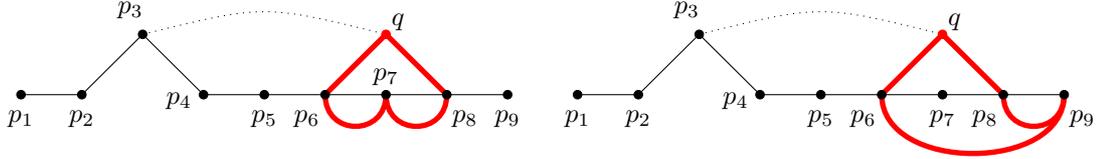
\noindent \begin{tikzpicture}[scale=0.8]
  \draw(4,-0.4) node {A 4-cycle, if $p_7$ is not an endpoint of $Q$};
\draw(6,2.3) node {};
\draw[-] (0,1)--(1,1);
\draw[-] (1,1)--(2,2);
\draw[-] (2,2)--(3,1);
\draw[-] (3,1)--(4,1);
\draw[-] (4,1)--(5,1);
\draw[-] (5,1)--(6,1);
\draw[-] (6,1)--(7,1);
\draw[-] (7,1)--(8,1);
\draw[-,line width=2pt,red] (5,1)--(6,2);
\draw[-,line width=2pt,red] (6,2)--(7,1);
\draw[-,line width=2pt,red] (5,1)..controls(5,0.3)and(6,0.3)..(6,1);
\draw[-,line width=2pt,red] (6,1)..controls(6,0.3)and(7,0.3)..(7,1);
\draw[-,dotted] (2,2)..controls(4,2.5)..(6,2);
\filldraw[red]  (6,2)    circle (2pt);
\filldraw [black]  (0,1)    circle (2pt)
[black]  (1,1)    circle (2pt)
[black]  (2,2)    circle (2pt)
[black]  (3,1)    circle (2pt)
[black]  (4,1)    circle (2pt)
[black]  (5,1)    circle (2pt)
[black]  (6,1)    circle (2pt)
[black]  (7,1)    circle (2pt)
[black]  (8,1)    circle (2pt);
\draw (0,0.6) node {$p_1$};
\draw (1,0.6) node {$p_2$};
\draw (1.8,2.4) node {$p_3$};
\draw (2.6,0.9) node {$p_4$};
\draw (4,0.6) node {$p_{5}$};
\draw (4.7,0.6) node {$p_6$};
\draw (6,1.3) node {$p_7$};
\draw (7.3,0.6) node {$p_8$};
\draw (8,0.6) node {$p_9$};
\draw (6.2,2.2) node {$q$};
\draw (8.5,3) node {};
\end{tikzpicture}
\begin{tikzpicture}[scale=0.8]
  \draw(4,-0.4) node {A 4-cycle, if $p_9$ is not an endpoint of $Q$};
\draw(6,2.3) node {};
\draw[-] (0,1)--(1,1);
\draw[-] (1,1)--(2,2);
\draw[-] (2,2)--(3,1);
\draw[-] (3,1)--(4,1);
\draw[-] (4,1)--(5,1);
\draw[-] (5,1)--(6,1);
\draw[-] (6,1)--(7,1);
\draw[-] (7,1)--(8,1);
\draw[-,line width=2pt,red] (5,1)--(6,2);
\draw[-,line width=2pt,red] (6,2)--(7,1);
\draw[-,line width=2pt,red] (5,1)..controls(5,-0.3)and(8,-0.3)..(8,1);
\draw[-,line width=2pt,red] (7,1)..controls(7,0.3)and(8,0.3)..(8,1);
\draw[-,dotted] (2,2)..controls(4,2.5)..(6,2);
\filldraw[red]  (6,2)    circle (2pt);
\filldraw [black]  (0,1)    circle (2pt)
[black]  (1,1)    circle (2pt)
[black]  (2,2)    circle (2pt)
[black]  (3,1)    circle (2pt)
[black]  (4,1)    circle (2pt)
[black]  (5,1)    circle (2pt)
[black]  (6,1)    circle (2pt)
[black]  (7,1)    circle (2pt)
[black]  (8,1)    circle (2pt);
\draw (0,0.6) node {$p_1$};
\draw (1,0.6) node {$p_2$};
\draw (1.8,2.4) node {$p_3$};
\draw (2.6,0.9) node {$p_4$};
\draw (4,0.6) node {$p_{5}$};
\draw (4.7,0.6) node {$p_6$};
\draw (6,0.6) node {$p_7$};
\draw (6.7,0.6) node {$p_8$};
\draw (8.3,0.6) node {$p_9$};
\draw (6.2,2.2) node {$q$};
\draw (8.5,3) node {};
\end{tikzpicture}

So $p_7$ and $p_9$ are the endpoints of $Q$, and each can be connected by $Q$ only with $p_6$ or $p_8$.
Then either $p_7$ connects with $p_6$ and $p_9$ with $p_8$, which leads to $Q=p_7p_6qp_8p_9$, and $Q$ has length 4, or
$p_7$ connects with $p_8$ and $p_9$ with $p_6$, which leads to $Q=p_7p_8qp_6p_9$, and $Q$ has length 4, too.

\noindent \begin{tikzpicture}[scale=0.8]
  \draw(4,-0.4) node {$p_7$ and $p_9$ are endpoints of $Q$, $L(Q)=4$};
\draw(6,2.3) node {};
\draw[-] (0,1)--(1,1);
\draw[-] (1,1)--(2,2);
\draw[-] (2,2)--(3,1);
\draw[-] (3,1)--(4,1);
\draw[-] (4,1)--(5,1);
\draw[-] (5,1)--(6,1);
\draw[-] (6,1)--(7,1);
\draw[-] (7,1)--(8,1);
\draw[-,line width=2pt,red] (5,1)--(6,2);
\draw[-,line width=2pt,red] (6,2)--(7,1);
\draw[-,line width=2pt,red] (5,1)..controls(5,-0.3)and(8,-0.3)..(8,1);
\draw[-,line width=2pt,red] (6,1)..controls(6,0.3)and(7,0.3)..(7,1);
\draw[-,dotted] (2,2)..controls(4,2.5)..(6,2);
\filldraw[red]  (6,2)    circle (2pt);
\filldraw [black]  (0,1)    circle (2pt)
[black]  (1,1)    circle (2pt)
[black]  (2,2)    circle (2pt)
[black]  (3,1)    circle (2pt)
[black]  (4,1)    circle (2pt)
[black]  (5,1)    circle (2pt)
[black]  (6,1)    circle (2pt)
[black]  (7,1)    circle (2pt)
[black]  (8,1)    circle (2pt);
\draw (0,0.6) node {$p_1$};
\draw (1,0.6) node {$p_2$};
\draw (1.8,2.4) node {$p_3$};
\draw (2.6,0.9) node {$p_4$};
\draw (4,0.6) node {$p_{5}$};
\draw (4.7,0.6) node {$p_6$};
\draw (5.7,0.6) node {$p_7$};
\draw (7.3,0.6) node {$p_8$};
\draw (8.3,0.6) node {$p_9$};
\draw (6.2,2.2) node {$q$};
\draw (8.5,3) node {};
\end{tikzpicture}
\begin{tikzpicture}[scale=0.8]
  \draw(4,-0.4) node {$p_7$ and $p_9$ are endpoints of $Q$, $L(Q)=4$};
\draw(6,2.3) node {};
\draw[-] (0,1)--(1,1);
\draw[-] (1,1)--(2,2);
\draw[-] (2,2)--(3,1);
\draw[-] (3,1)--(4,1);
\draw[-] (4,1)--(5,1);
\draw[-] (5,1)--(6,1);
\draw[-] (6,1)--(7,1);
\draw[-] (7,1)--(8,1);
\draw[-,line width=2pt,red] (5,1)--(6,2);
\draw[-,line width=2pt,red] (6,2)--(7,1);
\draw[-,line width=2pt,red] (5,1)..controls(5,0.3)and(6,0.3)..(6,1);
\draw[-,line width=2pt,red] (7,1)..controls(7,0.3)and(8,0.3)..(8,1);
\draw[-,dotted] (2,2)..controls(4,2.5)..(6,2);
\filldraw[red]  (6,2)    circle (2pt);
\filldraw [black]  (0,1)    circle (2pt)
[black]  (1,1)    circle (2pt)
[black]  (2,2)    circle (2pt)
[black]  (3,1)    circle (2pt)
[black]  (4,1)    circle (2pt)
[black]  (5,1)    circle (2pt)
[black]  (6,1)    circle (2pt)
[black]  (7,1)    circle (2pt)
[black]  (8,1)    circle (2pt);
\draw (0,0.6) node {$p_1$};
\draw (1,0.6) node {$p_2$};
\draw (1.8,2.4) node {$p_3$};
\draw (2.6,0.9) node {$p_4$};
\draw (4,0.6) node {$p_{5}$};
\draw (4.7,0.6) node {$p_6$};
\draw (6,1.3) node {$p_7$};
\draw (6.7,0.6) node {$p_8$};
\draw (8.3,0.6) node {$p_9$};
\draw (6.2,2.2) node {$q$};
\draw (8.5,3) node {};
\end{tikzpicture}

This contradiction discards the case $i=3$, $j=6$, $k=8$, and finishes the case $i=3$.

\begin{itemize}
  \item Assume that $i=4$.
\end{itemize}

By Remark~\ref{remark distancia minima} and assumption we know that $1<j<k<n-1=9$ and $|i-j|,|k-i|,|j-k|>1$.
The case $j<i=4$ yields $j=2$, since then $j<i-1=3$. Here we have the three cases $k=6$, $k=7$ and $k=8$.

\noindent 
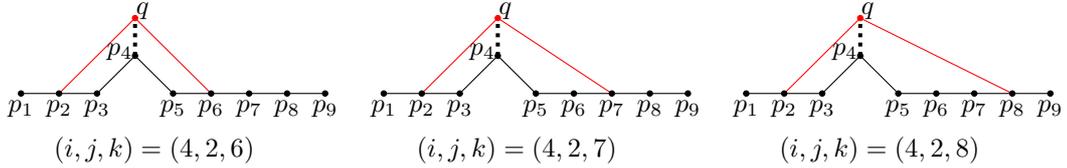
\begin{tikzpicture}[scale=0.5]
  \draw(3.5,-0.5) node {$(i,j,k)=(4,2,6)$};
\draw(6,2.3) node {};
\draw[-] (0,1)--(1,1);
\draw[-] (1,1)--(2,1);
\draw[-] (2,1)--(3,2);
\draw[-] (3,2)--(4,1);
\draw[-] (4,1)--(5,1);
\draw[-] (5,1)--(6,1);
\draw[-] (6,1)--(7,1);
\draw[-] (7,1)--(8,1);
\draw[-,red] (1,1)--(3,3);
\draw[-,red] (3,3)--(5,1);
\draw[-,line width=1.5pt,dotted] (3,2)--(3,3);
\filldraw[red]  (3,3)    circle (2pt);
\filldraw [black]  (0,1)    circle (2pt)
[black]  (1,1)    circle (2pt)
[black]  (2,1)    circle (2pt)
[black]  (3,2)    circle (2pt)
[black]  (4,1)    circle (2pt)
[black]  (5,1)    circle (2pt)
[black]  (6,1)    circle (2pt)
[black]  (7,1)    circle (2pt)
[black]  (8,1)    circle (2pt);
\draw (0,0.6) node {$p_1$};
\draw (1,0.6) node {$p_2$};
\draw (2,0.6) node {$p_3$};
\draw (2.6,2.1) node {$p_4$};
\draw (4,0.6) node {$p_{5}$};
\draw (5,0.6) node {$p_6$};
\draw (6,0.6) node {$p_7$};
\draw (7,0.6) node {$p_8$};
\draw (8,0.6) node {$p_9$};
\draw (3.2,3.2) node {$q$};
\draw (8.5,3) node {};
\end{tikzpicture}
\begin{tikzpicture}[scale=0.5]
  \draw(3.5,-0.5) node {$(i,j,k)=(4,2,7)$};
\draw(6,2.3) node {};
\draw[-] (0,1)--(1,1);
\draw[-] (1,1)--(2,1);
\draw[-] (2,1)--(3,2);
\draw[-] (3,2)--(4,1);
\draw[-] (4,1)--(5,1);
\draw[-] (5,1)--(6,1);
\draw[-] (6,1)--(7,1);
\draw[-] (7,1)--(8,1);
\draw[-,red] (1,1)--(3,3);
\draw[-,red] (3,3)--(6,1);
\draw[-,line width=1.5pt,dotted] (3,2)--(3,3);
\filldraw[red]  (3,3)    circle (2pt);
\filldraw [black]  (0,1)    circle (2pt)
[black]  (1,1)    circle (2pt)
[black]  (2,1)    circle (2pt)
[black]  (3,2)    circle (2pt)
[black]  (4,1)    circle (2pt)
[black]  (5,1)    circle (2pt)
[black]  (6,1)    circle (2pt)
[black]  (7,1)    circle (2pt)
[black]  (8,1)    circle (2pt);
\draw (0,0.6) node {$p_1$};
\draw (1,0.6) node {$p_2$};
\draw (2,0.6) node {$p_3$};
\draw (2.6,2.1) node {$p_4$};
\draw (4,0.6) node {$p_{5}$};
\draw (5,0.6) node {$p_6$};
\draw (6,0.6) node {$p_7$};
\draw (7,0.6) node {$p_8$};
\draw (8,0.6) node {$p_9$};
\draw (3.2,3.2) node {$q$};
\draw (8.5,3) node {};
\end{tikzpicture}
\begin{tikzpicture}[scale=0.5]
  \draw(3.5,-0.5) node {$(i,j,k)=(4,2,8)$};
\draw(6,2.3) node {};
\draw[-] (0,1)--(1,1);
\draw[-] (1,1)--(2,1);
\draw[-] (2,1)--(3,2);
\draw[-] (3,2)--(4,1);
\draw[-] (4,1)--(5,1);
\draw[-] (5,1)--(6,1);
\draw[-] (6,1)--(7,1);
\draw[-] (7,1)--(8,1);
\draw[-,red] (1,1)--(3,3);
\draw[-,red] (3,3)--(7,1);
\draw[-,line width=1.5pt,dotted] (3,2)--(3,3);
\filldraw[red]  (3,3)    circle (2pt);
\filldraw [black]  (0,1)    circle (2pt)
[black]  (1,1)    circle (2pt)
[black]  (2,1)    circle (2pt)
[black]  (3,2)    circle (2pt)
[black]  (4,1)    circle (2pt)
[black]  (5,1)    circle (2pt)
[black]  (6,1)    circle (2pt)
[black]  (7,1)    circle (2pt)
[black]  (8,1)    circle (2pt);
\draw (0,0.6) node {$p_1$};
\draw (1,0.6) node {$p_2$};
\draw (2,0.6) node {$p_3$};
\draw (2.6,2.1) node {$p_4$};
\draw (4,0.6) node {$p_{5}$};
\draw (5,0.6) node {$p_6$};
\draw (6,0.6) node {$p_7$};
\draw (7,0.6) node {$p_8$};
\draw (8,0.6) node {$p_9$};
\draw (3.2,3.2) node {$q$};
\draw (8.5,3) node {};
\end{tikzpicture}

If $j>i=4$, then the only possibility is $(i,j,k)=(4,6,8)$.
${\vcenter
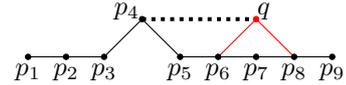
{\hbox{
\begin{tikzpicture}[scale=0.5]
\draw(6,2.3) node {};
\draw[-] (0,1)--(1,1);
\draw[-] (1,1)--(2,1);
\draw[-] (2,1)--(3,2);
\draw[-] (3,2)--(4,1);
\draw[-] (4,1)--(5,1);
\draw[-] (5,1)--(6,1);
\draw[-] (6,1)--(7,1);
\draw[-] (7,1)--(8,1);
\draw[-,red] (5,1)--(6,2);
\draw[-,red] (6,2)--(7,1);
\draw[-,line width=1.5pt,dotted] (3,2)--(6,2);
\filldraw[red]  (6,2)    circle (2pt);
\filldraw [black]  (0,1)    circle (2pt)
[black]  (1,1)    circle (2pt)
[black]  (2,1)    circle (2pt)
[black]  (3,2)    circle (2pt)
[black]  (4,1)    circle (2pt)
[black]  (5,1)    circle (2pt)
[black]  (6,1)    circle (2pt)
[black]  (7,1)    circle (2pt)
[black]  (8,1)    circle (2pt);
\draw (0,0.6) node {$p_1$};
\draw (1,0.6) node {$p_2$};
\draw (2,0.6) node {$p_3$};
\draw (2.6,2.2) node {$p_4$};
\draw (4,0.6) node {$p_{5}$};
\draw (5,0.6) node {$p_6$};
\draw (6,0.6) node {$p_7$};
\draw (7,0.6) node {$p_8$};
\draw (8,0.6) node {$p_9$};
\draw (6.2,2.2) node {$q$};
\draw (8.5,3) node {};
\end{tikzpicture}}}}$

\begin{itemize}
  \item[$\bullet\bullet$] Assume that $(i,j,k)=(4,2,6)$.
\end{itemize}

In this case it doesn't suffice to analyze the connections of a single vertex,
we have to consider the connections of $p_3$ and $p_5$.
Proposition~\ref{forbidden pairs} shows that an edge of $Q$ cannot connect $p_3$ with one of $\{p_1,p_5,p_7,p_9\}$, nor can it connect
 $p_5$ with one of $\{p_1,p_3,p_7,p_9\}$.
 Moreover, neither $p_3$ nor $p_5$ can be an endpoint of $Q$, since then we could extend $Q$ connecting $p_4$.
Hence each of $p_3$ and $p_5$ connects with two of $\{ p_2,p_6,p_8\}$, and since only one of them can connect with $p_2$, and
only one of them can connect with $p_6$, we have two possibilities: The first is that $p_3$ connects with $p_2$ and $p_8$; and $p_5$ connects with
$p_6$ and $p_8$. In the second one $p_3$ connects with $p_6$ and $p_8$; and $p_5$ connects with
$p_2$ and $p_8$. In both cases we obtain a 6-cycle contained in $Q$, which is impossible and discards the case $(i,j,k)=(4,2,6)$.

\noindent
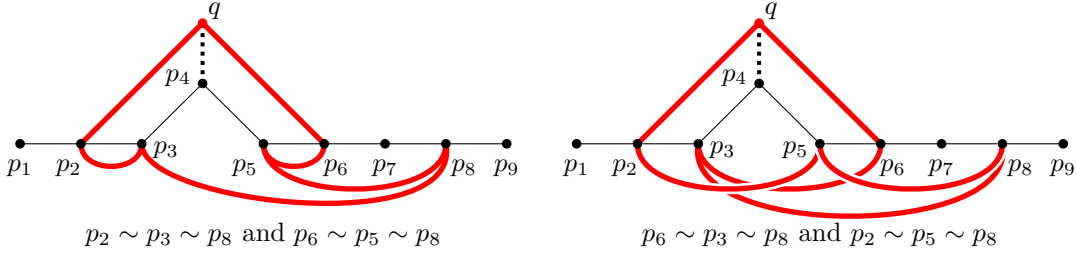
\begin{tikzpicture}[scale=0.8]
  \draw(4,-0.5) node {$p_2\sim p_3\sim p_8$ and $p_6\sim p_5\sim p_8$};
\draw(6,2.3) node {};
\draw[-] (0,1)--(1,1);
\draw[-] (1,1)--(2,1);
\draw[-] (2,1)--(3,2);
\draw[-] (3,2)--(4,1);
\draw[-] (4,1)--(5,1);
\draw[-] (5,1)--(6,1);
\draw[-] (6,1)--(7,1);
\draw[-] (7,1)--(8,1);
\draw[-,red,line width=2pt] (1,1)--(3,3);
\draw[-,red,line width=2pt] (3,3)--(5,1);
\draw[-,red,line width=2pt] (4,1)..controls(4,0)and(7,0)..(7,1);
\draw[-,red,line width=2pt] (2,1)..controls(2,0)and(7.3,-0.6)..(7,1);
\draw[-,red,line width=2pt] (1,1)..controls(1,0.5)and(2,0.5)..(2,1);
\draw[-,red,line width=2pt] (4,1)..controls(4,0.5)and(5,0.5)..(5,1);
\draw[-,line width=1.5pt,dotted] (3,2)--(3,3);
\filldraw[red]  (3,3)    circle (2pt);
\filldraw [black]  (0,1)    circle (2pt)
[black]  (1,1)    circle (2pt)
[black]  (2,1)    circle (2pt)
[black]  (3,2)    circle (2pt)
[black]  (4,1)    circle (2pt)
[black]  (5,1)    circle (2pt)
[black]  (6,1)    circle (2pt)
[black]  (7,1)    circle (2pt)
[black]  (8,1)    circle (2pt);
\draw (0,0.6) node {$p_1$};
\draw (0.8,0.6) node {$p_2$};
\draw (2.4,0.9) node {$p_3$};
\draw (2.6,2.1) node {$p_4$};
\draw (3.7,0.6) node {$p_{5}$};
\draw (5.2,0.6) node {$p_6$};
\draw (6,0.6) node {$p_7$};
\draw (7.3,0.6) node {$p_8$};
\draw (8,0.6) node {$p_9$};
\draw (3.2,3.2) node {$q$};
\draw (8.5,3) node {};
\end{tikzpicture}
\begin{tikzpicture}[scale=0.8]
  \draw(4,-0.5) node {$p_6\sim p_3\sim p_8$ and $p_2\sim p_5\sim p_8$};
\draw(6,2.3) node {};
\draw[-] (0,1)--(1,1);
\draw[-] (1,1)--(2,1);
\draw[-] (2,1)--(3,2);
\draw[-] (3,2)--(4,1);
\draw[-] (4,1)--(5,1);
\draw[-] (5,1)--(6,1);
\draw[-] (6,1)--(7,1);
\draw[-] (7,1)--(8,1);
\draw[-,red,line width=2pt] (1,1)--(3,3);
\draw[-,red,line width=2pt] (3,3)--(5,1);
\draw[-,red,line width=2pt] (2,1)..controls(1.7,-0.6)and(7.3,-0.6)..(7,1);
\draw[-,red,line width=2pt] (2,1)..controls(2,0)and(5,0)..(5,1);
\draw[-,white,line width=4pt] (1,1)..controls(1,0)and(4,0)..(4,1);
\draw[-,red,line width=2pt] (1,1)..controls(1,0)and(4,0)..(4,1);
\draw[-,white,line width=4pt] (4,1)..controls(4,0)and(7,0)..(7,1);
\draw[-,red,line width=2pt] (4,1)..controls(4,0)and(7,0)..(7,1);
\draw[-,line width=1.5pt,dotted] (3,2)--(3,3);
\filldraw[red]  (3,3)    circle (2pt);
\filldraw [black]  (0,1)    circle (2pt)
[black]  (1,1)    circle (2pt)
[black]  (2,1)    circle (2pt)
[black]  (3,2)    circle (2pt)
[black]  (4,1)    circle (2pt)
[black]  (5,1)    circle (2pt)
[black]  (6,1)    circle (2pt)
[black]  (7,1)    circle (2pt)
[black]  (8,1)    circle (2pt);
\draw (0,0.6) node {$p_1$};
\draw (0.8,0.6) node {$p_2$};
\draw (2.4,0.9) node {$p_3$};
\draw (2.6,2.1) node {$p_4$};
\draw (3.6,0.9) node {$p_{5}$};
\draw (5.2,0.6) node {$p_6$};
\draw (6,0.6) node {$p_7$};
\draw (7.3,0.6) node {$p_8$};
\draw (8,0.6) node {$p_9$};
\draw (3.2,3.2) node {$q$};
\draw (8.5,3) node {};
\end{tikzpicture}

\begin{itemize}
  \item[$\bullet\bullet$] Assume that $(i,j,k)=(4,2,7)$.
\end{itemize}

We know that $p_3$ cannot be an endpoint of $Q$, since then we could continue $Q$ with the edge $p_3p_4$.

  Thus two edges of the path $Q$ connect $p_3$ directly with two vertices in the set
  $$
  (V(P)\cap V(Q))\setminus\{p_3\}=\{p_1,p_2,p_5,p_6,p_7,p_8,p_9\}.
  $$

  By Proposition~\ref{forbidden pairs} an edge of $Q$ cannot connect $p_3=p_{i-1}$ with $p_1$, $p_5=p_{i+1}$, $p_6=p_{k-1}$,
  $p_8=p_{k+1}$ nor $p_9=p_{n-1}$.

Thus two edges of the path $Q$ connect $p_3$ with $p_2$ and $p_3$ with $p_7$, which implies that $Q$ contains the 4-cycle $p_2 p_3  p_7 q p_2$,
a contradiction that discards
the case $i=4$, $j=2$ and $k=7$.
\begin{tikzpicture}[scale=0.8]
  \draw(4,-0.5) node {$Q$ contains a 4-cycle};
\draw(6,2.3) node {};
\draw[-] (0,1)--(1,1);
\draw[-] (1,1)--(2,1);
\draw[-] (2,1)--(3,2);
\draw[-] (3,2)--(4,1);
\draw[-] (4,1)--(5,1);
\draw[-] (5,1)--(6,1);
\draw[-] (6,1)--(7,1);
\draw[-] (7,1)--(8,1);
\draw[-,red,line width=2pt] (1,1)--(3,3);
\draw[-,red,line width=2pt] (3,3)--(6,1);
\draw[-,red,line width=2pt] (2,1)..controls(2,0)and(6.2,-0.2)..(6,1);
\draw[-,red,line width=2pt] (1,1)..controls(1,0.5)and(2,0.5)..(2,1);
\draw[-] (3,2)--(3,3);
\filldraw[red]  (3,3)    circle (2pt);
\filldraw [black]  (0,1)    circle (2pt)
[black]  (1,1)    circle (2pt)
[black]  (2,1)    circle (2pt)
[black]  (3,2)    circle (2pt)
[black]  (4,1)    circle (2pt)
[black]  (5,1)    circle (2pt)
[black]  (6,1)    circle (2pt)
[black]  (7,1)    circle (2pt)
[black]  (8,1)    circle (2pt);
\draw (0,0.6) node {$p_1$};
\draw (0.9,0.6) node {$p_2$};
\draw (2.4,0.9) node {$p_3$};
\draw (2.6,2.1) node {$p_4$};
\draw (4,0.6) node {$p_{5}$};
\draw (5,0.6) node {$p_6$};
\draw (6.3,0.6) node {$p_7$};
\draw (7,0.6) node {$p_8$};
\draw (8,0.6) node {$p_9$};
\draw (3.2,3.2) node {$q$};
\draw (8.5,3) node {};
\end{tikzpicture}

\begin{itemize}
  \item[$\bullet\bullet$] Assume that $(i,j,k)=(4,2,8)$.
\end{itemize}

In this case it doesn't suffice to analyze the connections of a single vertex or the connections of two vertices,
we have to consider the connections of $p_1$, $p_3$ and $p_9$.
By Proposition~\ref{forbidden pairs} there cannot be an edge of $Q$ connecting $p_1$ with one of $\{p_3,p_5,p_7,p_9\}$,
nor an edge of $Q$ connecting $p_3$ with one of $\{p_1,p_5,p_7,p_9\}$,
nor an edge of $Q$ connecting $p_9$ with one of $\{p_1,p_3,p_5,p_7\}$.

We know that $p_3$ cannot be an endpoint of $Q$, since then we could continue $Q$ with the edge $p_3p_4$. Thus $p_3$ must connect with two of
$\{p_2,p_6,p_8\}$. But it cannot connect with $p_2$ and $p_8$, since then we would have a 4-cycle contained in $Q$. So it connects with
$p_6$ and one of $p_2$ or $p_8$.

If $p_3$ connects with $p_2$, then one of $\{p_1,p_9\}$ connects with $p_6$ and the other with $p_8$. Hence, both
$p_1$ and $p_9$ must be endpoints, and in each case, $p_1\sim p_6$, $p_9\sim p_8$ or  $p_1\sim p_8$, $p_9\sim p_6$,
we obtain that $Q$ has length 6, which is a contradiction.

\noindent 
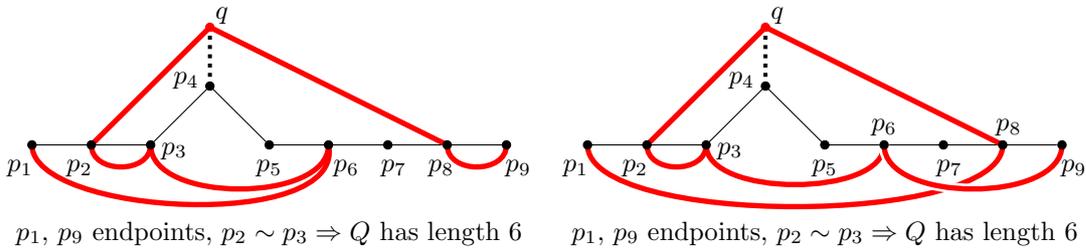
\begin{tikzpicture}[scale=0.78]
\draw(4,-0.5) node {$p_1$, $p_9$ endpoints, $p_2\sim p_3\Rightarrow Q$ has length 6};
\draw(6,2.3) node {};
\draw[-] (0,1)--(1,1);
\draw[-] (1,1)--(2,1);
\draw[-] (2,1)--(3,2);
\draw[-] (3,2)--(4,1);
\draw[-] (4,1)--(5,1);
\draw[-] (5,1)--(6,1);
\draw[-] (6,1)--(7,1);
\draw[-] (7,1)--(8,1);
\draw[-,red,line width=2pt] (1,1)--(3,3);
\draw[-,red,line width=2pt] (3,3)--(7,1);
\draw[-,red,line width=2pt] (7,1)..controls(7,0.5)and(8,0.5)..(8,1);
\draw[-,red,line width=2pt] (1,1)..controls(1,0.5)and(2,0.5)..(2,1);
\draw[-,red,line width=2pt] (0,1)..controls(0,-0.2)and(5.3,-0.5)..(5,1);
\draw[-,red,line width=2pt] (2,1)..controls(2,0)and(5,0)..(5,1);
\draw[-,line width=1.5pt,dotted] (3,2)--(3,3);
\filldraw[red]  (3,3)    circle (2pt);
\filldraw [black]  (0,1)    circle (2pt)
[black]  (1,1)    circle (2pt)
[black]  (2,1)    circle (2pt)
[black]  (3,2)    circle (2pt)
[black]  (4,1)    circle (2pt)
[black]  (5,1)    circle (2pt)
[black]  (6,1)    circle (2pt)
[black]  (7,1)    circle (2pt)
[black]  (8,1)    circle (2pt);
\draw (-0.2,0.6) node {$p_1$};
\draw (0.8,0.6) node {$p_2$};
\draw (2.4,0.9) node {$p_3$};
\draw (2.6,2.1) node {$p_4$};
\draw (4,0.6) node {$p_{5}$};
\draw (5.3,0.6) node {$p_6$};
\draw (6.1,0.6) node {$p_7$};
\draw (6.9,0.6) node {$p_8$};
\draw (8.2,0.6) node {$p_9$};
\draw (3.2,3.2) node {$q$};
\draw (8.5,3) node {};
\end{tikzpicture}
\begin{tikzpicture}[scale=0.78]
\draw(4,-0.5) node {$p_1$, $p_9$ endpoints, $p_2\sim p_3\Rightarrow Q$ has length 6};
\draw(6,2.3) node {};
\draw[-] (0,1)--(1,1);
\draw[-] (1,1)--(2,1);
\draw[-] (2,1)--(3,2);
\draw[-] (3,2)--(4,1);
\draw[-] (4,1)--(5,1);
\draw[-] (5,1)--(6,1);
\draw[-] (6,1)--(7,1);
\draw[-] (7,1)--(8,1);
\draw[-,red,line width=2pt] (1,1)--(3,3);
\draw[-,red,line width=2pt] (3,3)--(7,1);
\draw[-,red,line width=2pt] (1,1)..controls(1,0.5)and(2,0.5)..(2,1);
\draw[-,red,line width=2pt] (0,1)..controls(0,-0.4)and(7,-0.4)..(7,1);
\draw[-,red,line width=2pt] (2,1)..controls(2,0.1)and(5,0.1)..(5,1);
\draw[-,white,line width=4pt] (5,1)..controls(5,0)and(8,0)..(8,1);
\draw[-,red,line width=2pt] (5,1)..controls(5,0)and(8,0)..(8,1);
\draw[-,line width=1.5pt,dotted] (3,2)--(3,3);
\filldraw[red]  (3,3)    circle (2pt);
\filldraw [black]  (0,1)    circle (2pt)
[black]  (1,1)    circle (2pt)
[black]  (2,1)    circle (2pt)
[black]  (3,2)    circle (2pt)
[black]  (4,1)    circle (2pt)
[black]  (5,1)    circle (2pt)
[black]  (6,1)    circle (2pt)
[black]  (7,1)    circle (2pt)
[black]  (8,1)    circle (2pt);
\draw (-0.2,0.6) node {$p_1$};
\draw (0.8,0.6) node {$p_2$};
\draw (2.4,0.9) node {$p_3$};
\draw (2.6,2.1) node {$p_4$};
\draw (4,0.6) node {$p_{5}$};
\draw (5,1.3) node {$p_6$};
\draw (6.1,0.6) node {$p_7$};
\draw (7.1,1.3) node {$p_8$};
\draw (8.2,0.6) node {$p_9$};
\draw (3.2,3.2) node {$q$};
\draw (8.5,3) node {};
\end{tikzpicture}

If $p_3$ connects with $p_8$, then one of $\{p_1,p_9\}$ connects with $p_2$ and the other with $p_6$. Hence, both
$p_1$ and $p_9$ must be endpoints, and in each case, $p_1\sim p_6$, $p_9\sim p_2$ or  $p_1\sim p_2$, $p_9\sim p_6$,
we obtain that $Q$ has length 6, which is a contradiction.

\noindent \begin{tikzpicture}[scale=0.78]
\draw(4,-0.5) node {$p_1$, $p_9$ endpoints, $p_3\sim p_8\Rightarrow Q$ has length 6};
\draw(6,2.3) node {};
\draw[-] (0,1)--(1,1);
\draw[-] (1,1)--(2,1);
\draw[-] (2,1)--(3,2);
\draw[-] (3,2)--(4,1);
\draw[-] (4,1)--(5,1);
\draw[-] (5,1)--(6,1);
\draw[-] (6,1)--(7,1);
\draw[-] (7,1)--(8,1);
\draw[-,red,line width=2pt] (1,1)--(3,3);
\draw[-,red,line width=2pt] (3,3)--(7,1);
\draw[-,red,line width=2pt] (1,1)..controls(1,-0.5)and(8,-0.5)..(8,1);
\draw[-,red,line width=2pt] (2,1)..controls(2,0)and(7,0)..(7,1);
\draw[-,red,line width=2pt] (2,1)..controls(2,0.5)and(5,0.5)..(5,1);
\draw[-,white,line width=4pt] (0,1)..controls(0,-0.5)and(5.3,-0.5)..(5,1);
\draw[-,red,line width=2pt] (0,1)..controls(0,-0.5)and(5.3,-0.5)..(5,1);
\draw[-,line width=1.5pt,dotted] (3,2)--(3,3);
\filldraw[red]  (3,3)    circle (2pt);
\filldraw [black]  (0,1)    circle (2pt)
[black]  (1,1)    circle (2pt)
[black]  (2,1)    circle (2pt)
[black]  (3,2)    circle (2pt)
[black]  (4,1)    circle (2pt)
[black]  (5,1)    circle (2pt)
[black]  (6,1)    circle (2pt)
[black]  (7,1)    circle (2pt)
[black]  (8,1)    circle (2pt);
\draw (-0.2,0.6) node {$p_1$};
\draw (0.8,0.6) node {$p_2$};
\draw (2.4,0.9) node {$p_3$};
\draw (2.6,2.1) node {$p_4$};
\draw (3.6,0.9) node {$p_{5}$};
\draw (5.3,0.6) node {$p_6$};
\draw (6.1,0.6) node {$p_7$};
\draw (7.1,1.3) node {$p_8$};
\draw (8.2,0.6) node {$p_9$};
\draw (3.2,3.2) node {$q$};
\draw (8.5,3) node {};
\end{tikzpicture}
\begin{tikzpicture}[scale=0.78]
\draw(4,-0.5) node {$p_1$, $p_9$ endpoints, $p_3\sim p_8\Rightarrow Q$ has length 6};
\draw(6,2.3) node {};
\draw[-] (0,1)--(1,1);
\draw[-] (1,1)--(2,1);
\draw[-] (2,1)--(3,2);
\draw[-] (3,2)--(4,1);
\draw[-] (4,1)--(5,1);
\draw[-] (5,1)--(6,1);
\draw[-] (6,1)--(7,1);
\draw[-] (7,1)--(8,1);
\draw[-,red,line width=2pt] (1,1)--(3,3);
\draw[-,red,line width=2pt] (3,3)--(7,1);
\draw[-,red,line width=2pt] (2,1)..controls(2,0)and(7,0)..(7,1);
\draw[-,red,line width=2pt] (2,1)..controls(2,0.5)and(5,0.5)..(5,1);
\draw[-,red,line width=2pt] (0,1)..controls(0,0.5)and(1,0.5)..(1,1);
\draw[-,white,line width=4pt] (5,1)..controls(5,0)and(8,0)..(8,1);
\draw[-,red,line width=2pt] (5,1)..controls(5,0)and(8,0)..(8,1);
\draw[-,line width=1.5pt,dotted] (3,2)--(3,3);
\filldraw[red]  (3,3)    circle (2pt);
\filldraw [black]  (0,1)    circle (2pt)
[black]  (1,1)    circle (2pt)
[black]  (2,1)    circle (2pt)
[black]  (3,2)    circle (2pt)
[black]  (4,1)    circle (2pt)
[black]  (5,1)    circle (2pt)
[black]  (6,1)    circle (2pt)
[black]  (7,1)    circle (2pt)
[black]  (8,1)    circle (2pt);
\draw (-0.2,0.6) node {$p_1$};
\draw (0.8,1.3) node {$p_2$};
\draw (1.7,0.6) node {$p_3$};
\draw (2.6,2.1) node {$p_4$};
\draw (3.6,0.9) node {$p_{5}$};
\draw (5,1.3) node {$p_6$};
\draw (6.1,0.6) node {$p_7$};
\draw (7.1,1.3) node {$p_8$};
\draw (8.2,0.6) node {$p_9$};
\draw (3.2,3.2) node {$q$};
\draw (8.5,3) node {};
\end{tikzpicture}

This discards the case $(i,j,k)=(4,2,8)$.

\begin{itemize}
  \item[$\bullet\bullet$] Assume that $(i,j,k)=(4,6,8)$.
\end{itemize}

In this case it doesn't suffice to analyze the connections of a single vertex or the connections of two vertices,
we have to consider the connections of $p_5$, $p_7$ and $p_9$.
By Proposition~\ref{forbidden pairs} there cannot be an edge of $Q$ connecting $p_5$ with one of $\{p_1,p_3,p_7,p_9\}$,
nor an edge of $Q$ connecting $p_7$ with one of $\{p_1,p_3,p_5,p_9\}$,
nor an edge of $Q$ connecting $p_9$ with one of $\{p_1,p_3,p_5,p_7\}$.

We know that $p_5$ cannot be an endpoint of $Q$, since then we could continue $Q$ with the edge $p_5p_4$. Thus $p_5$ must be connected by an edge of
$Q$ with two of
$\{p_2,p_6,p_8\}$. But it cannot be connected with $p_6$ and $p_8$ at the same time,
 since then we would have a 4-cycle contained in $Q$. So it is connected with
$p_2$ and one of $p_6$ or $p_8$.

If $p_5$ is connected with $p_6$, then one of $\{p_7,p_9\}$ is connected with $p_2$ and the other with $p_8$. Hence, both
$p_7$ and $p_9$ must be endpoints, and in each case, $p_7\sim p_2$, $p_9\sim p_8$ or  $p_7\sim p_8$, $p_9\sim p_2$,
we obtain that $Q$ has length 6, which is a contradiction.

\noindent
\begin{tikzpicture}[scale=0.8]
\draw(6,2.3) node {};
\draw(4,-0.5) node {$p_7$, $p_9$ endpoints, $p_5\sim p_6\Rightarrow Q$ has length 6};%
\draw[-] (0,1)--(1,1);
\draw[-] (1,1)--(2,1);
\draw[-] (2,1)--(3,2);
\draw[-] (3,2)--(4,1);
\draw[-] (4,1)--(5,1);
\draw[-] (5,1)--(6,1);
\draw[-] (6,1)--(7,1);
\draw[-] (7,1)--(8,1);
\draw[-,red,line width=2pt] (5,1)--(6,2);
\draw[-,red,line width=2pt] (6,2)--(7,1);
\draw[-] (3,2)--(6,2);
\draw[-,red,line width=2pt] (1,1)..controls(0.7,-0.3)and(6.3,-0.3)..(6,1);
\draw[-,red,line width=2pt] (1,1)..controls(1,0)and(4,0)..(4,1);
\draw[-,red,line width=2pt] (4,1)..controls(4,0.5)and(5,0.5)..(5,1);
\draw[-,red,line width=2pt] (7,1)..controls(7,0.5)and(8,0.5)..(8,1);
\filldraw[red]  (6,2)    circle (2pt);
\filldraw [black]  (0,1)    circle (2pt)
[black]  (1,1)    circle (2pt)
[black]  (2,1)    circle (2pt)
[black]  (3,2)    circle (2pt)
[black]  (4,1)    circle (2pt)
[black]  (5,1)    circle (2pt)
[black]  (6,1)    circle (2pt)
[black]  (7,1)    circle (2pt)
[black]  (8,1)    circle (2pt);
\draw (0,0.6) node {$p_1$};
\draw (0.7,0.6) node {$p_2$};
\draw (2,0.6) node {$p_3$};
\draw (2.6,2.2) node {$p_4$};
\draw (3.6,0.9) node {$p_{5}$};
\draw (5.2,0.6) node {$p_6$};
\draw (6.3,0.6) node {$p_7$};
\draw (6.9,0.6) node {$p_8$};
\draw (8.2,0.6) node {$p_9$};
\draw (6.2,2.2) node {$q$};
\draw (8.5,3) node {};
\end{tikzpicture}
\begin{tikzpicture}[scale=0.8]
\draw(6,2.3) node {};
\draw(4,-0.5) node {$p_7$, $p_9$ endpoints, $p_5\sim p_6\Rightarrow Q$ has length 6};%
\draw[-] (0,1)--(1,1);
\draw[-] (1,1)--(2,1);
\draw[-] (2,1)--(3,2);
\draw[-] (3,2)--(4,1);
\draw[-] (4,1)--(5,1);
\draw[-] (5,1)--(6,1);
\draw[-] (6,1)--(7,1);
\draw[-] (7,1)--(8,1);
\draw[-,red,line width=2pt] (5,1)--(6,2);
\draw[-,red,line width=2pt] (6,2)--(7,1);
\draw[-] (3,2)--(6,2);
\draw[-,red,line width=2pt] (1,1)..controls(0.7,-0.4)and(8.3,-0.3)..(8,1);
\draw[-,red,line width=2pt] (1,1)..controls(1,0.2)and(4,0.2)..(4,1);
\draw[-,red,line width=2pt] (4,1)..controls(4,0.5)and(5,0.5)..(5,1);
\draw[-,red,line width=2pt] (6,1)..controls(6,0.5)and(7,0.5)..(7,1);
\filldraw[red]  (6,2)    circle (2pt);
\filldraw [black]  (0,1)    circle (2pt)
[black]  (1,1)    circle (2pt)
[black]  (2,1)    circle (2pt)
[black]  (3,2)    circle (2pt)
[black]  (4,1)    circle (2pt)
[black]  (5,1)    circle (2pt)
[black]  (6,1)    circle (2pt)
[black]  (7,1)    circle (2pt)
[black]  (8,1)    circle (2pt);
\draw (0,0.6) node {$p_1$};
\draw (0.7,0.6) node {$p_2$};
\draw (2,0.7) node {$p_3$};
\draw (2.6,2.2) node {$p_4$};
\draw (3.6,0.9) node {$p_{5}$};
\draw (5.2,0.6) node {$p_6$};
\draw (5.8,0.6) node {$p_7$};
\draw (7.2,0.6) node {$p_8$};
\draw (8.3,0.6) node {$p_9$};
\draw (6.2,2.2) node {$q$};
\draw (8.5,3) node {};
\end{tikzpicture}

If $p_5$ is connected with $p_8$, then one of $\{p_7,p_9\}$ is connected with $p_2$ and the other with $p_6$. Hence, both
$p_7$ and $p_9$ must be endpoints, and in each case, $p_7\sim p_2$, $p_9\sim p_6$ or  $p_7\sim p_6$, $p_9\sim p_2$,
we obtain that $Q$ has length 6, which is a contradiction.

\noindent
\begin{tikzpicture}[scale=0.8]
\draw(6,2.3) node {};
\draw(4,-0.5) node {$p_7$, $p_9$ endpoints, $p_5\sim p_8\Rightarrow Q$ has length 6};%
\draw[-] (0,1)--(1,1);
\draw[-] (1,1)--(2,1);
\draw[-] (2,1)--(3,2);
\draw[-] (3,2)--(4,1);
\draw[-] (4,1)--(5,1);
\draw[-] (5,1)--(6,1);
\draw[-] (6,1)--(7,1);
\draw[-] (7,1)--(8,1);
\draw[-,red,line width=2pt] (5,1)--(6,2);
\draw[-,red,line width=2pt] (6,2)--(7,1);
\draw[-] (3,2)--(6,2);
\draw[-,red,line width=2pt] (1,1)..controls(1,0)and(4,0)..(4,1);
\draw[-,red,line width=2pt] (4,1)..controls(4,0)and(7,0)..(7,1);
\draw[-,white,line width=4pt] (5,1)..controls(5,0)and(8,0)..(8,1);
\draw[-,red,line width=2pt] (5,1)..controls(5,0)and(8,0)..(8,1);
\draw[-,white,line width=4pt] (1,1)..controls(0.7,-0.3)and(6.3,-0.3)..(6,1);
\draw[-,red,line width=2pt] (1,1)..controls(0.7,-0.3)and(6.3,-0.3)..(6,1);
\filldraw[red]  (6,2)    circle (2pt);
\filldraw [black]  (0,1)    circle (2pt)
[black]  (1,1)    circle (2pt)
[black]  (2,1)    circle (2pt)
[black]  (3,2)    circle (2pt)
[black]  (4,1)    circle (2pt)
[black]  (5,1)    circle (2pt)
[black]  (6,1)    circle (2pt)
[black]  (7,1)    circle (2pt)
[black]  (8,1)    circle (2pt);
\draw (0,0.6) node {$p_1$};
\draw (0.7,0.6) node {$p_2$};
\draw (2,0.6) node {$p_3$};
\draw (2.6,2.2) node {$p_4$};
\draw (3.6,0.9) node {$p_{5}$};
\draw (4.7,0.6) node {$p_6$};
\draw (6.3,0.6) node {$p_7$};
\draw (7.3,0.6) node {$p_8$};
\draw (8.2,0.6) node {$p_9$};
\draw (6.2,2.2) node {$q$};
\draw (8.5,3) node {};
\end{tikzpicture}
\begin{tikzpicture}[scale=0.8]
\draw(6,2.3) node {};
\draw(4,-0.5) node {$p_7$, $p_9$ endpoints, $p_5\sim p_8\Rightarrow Q$ has length 6};%
\draw[-] (0,1)--(1,1);
\draw[-] (1,1)--(2,1);
\draw[-] (2,1)--(3,2);
\draw[-] (3,2)--(4,1);
\draw[-] (4,1)--(5,1);
\draw[-] (5,1)--(6,1);
\draw[-] (6,1)--(7,1);
\draw[-] (7,1)--(8,1);
\draw[-,red,line width=2pt] (5,1)--(6,2);
\draw[-,red,line width=2pt] (6,2)--(7,1);
\draw[-] (3,2)--(6,2);
\draw[-,red,line width=2pt] (1,1)..controls(0.7,-0.4)and(8.3,-0.3)..(8,1);
\draw[-,red,line width=2pt] (1,1)..controls(1,0.2)and(4,0.2)..(4,1);
\draw[-,red,line width=2pt] (5,1)..controls(5,0.5)and(6,0.5)..(6,1);
\draw[-,red,line width=2pt] (4,1)..controls(4,0.1)and(7,0.1)..(7,1);
\filldraw[red]  (6,2)    circle (2pt);
\filldraw [black]  (0,1)    circle (2pt)
[black]  (1,1)    circle (2pt)
[black]  (2,1)    circle (2pt)
[black]  (3,2)    circle (2pt)
[black]  (4,1)    circle (2pt)
[black]  (5,1)    circle (2pt)
[black]  (6,1)    circle (2pt)
[black]  (7,1)    circle (2pt)
[black]  (8,1)    circle (2pt);
\draw (0,0.6) node {$p_1$};
\draw (0.7,0.6) node {$p_2$};
\draw (2,0.7) node {$p_3$};
\draw (2.6,2.2) node {$p_4$};
\draw (3.6,0.9) node {$p_{5}$};
\draw (4.8,0.7) node {$p_6$};
\draw (6.2,0.7) node {$p_7$};
\draw (7.2,0.6) node {$p_8$};
\draw (8.3,0.6) node {$p_9$};
\draw (6.2,2.2) node {$q$};
\draw (8.5,3) node {};
\end{tikzpicture}

This discards the case $(i,j,k)=(4,6,8)$ and finishes the case $i=4$.

\begin{itemize}
  \item Assume that $i=5$.
\end{itemize}

By Remark~\ref{remark distancia minima} and assumption we know that $1<j<k<n-1=9$ and $|i-j|,|k-i|,|j-k|>1$.
The case $i<j$ and $k<i$ are impossible, hence $j<i<k$, and we have four cases
$$
(i,j,k)\in\{(5,2,7),(5,2,8),(5,3,7),(5,3,8)\}.
$$
But the cases $(i,j,k)=(5,2,7)$ and $(i,j,k)=(5,3,8)$ are symmetric, so we have to discard only the first three cases.

\noindent 
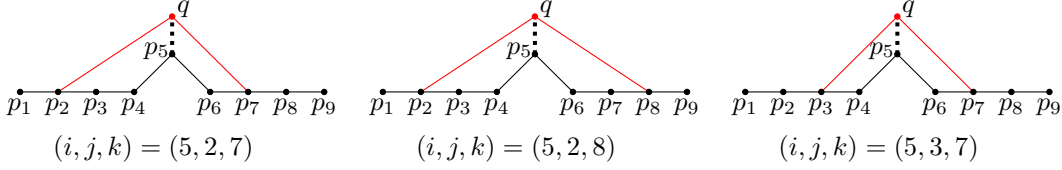
\begin{tikzpicture}[scale=0.5]
  \draw(3.5,-0.5) node {$(i,j,k)=(5,2,7)$};
\draw(6,2.3) node {};
\draw[-] (0,1)--(1,1);
\draw[-] (1,1)--(2,1);
\draw[-] (2,1)--(3,1);
\draw[-] (3,1)--(4,2);
\draw[-] (4,2)--(5,1);
\draw[-] (5,1)--(6,1);
\draw[-] (6,1)--(7,1);
\draw[-] (7,1)--(8,1);
\draw[-,red] (1,1)--(4,3);
\draw[-,red] (4,3)--(6,1);
\draw[-,line width=1.5pt,dotted] (4,2)--(4,3);
\filldraw[red]  (4,3)    circle (2pt);
\filldraw [black]  (0,1)    circle (2pt)
[black]  (1,1)    circle (2pt)
[black]  (2,1)    circle (2pt)
[black]  (3,1)    circle (2pt)
[black]  (4,2)    circle (2pt)
[black]  (5,1)    circle (2pt)
[black]  (6,1)    circle (2pt)
[black]  (7,1)    circle (2pt)
[black]  (8,1)    circle (2pt);
\draw (0,0.6) node {$p_1$};
\draw (1,0.6) node {$p_2$};
\draw (2,0.6) node {$p_3$};
\draw (3,0.6) node {$p_4$};
\draw (3.6,2.1) node {$p_{5}$};
\draw (5,0.6) node {$p_6$};
\draw (6,0.6) node {$p_7$};
\draw (7,0.6) node {$p_8$};
\draw (8,0.6) node {$p_9$};
\draw (4.3,3.2) node {$q$};
\draw (8.5,3) node {};
\end{tikzpicture}
\begin{tikzpicture}[scale=0.5]
  \draw(3.5,-0.5) node {$(i,j,k)=(5,2,8)$};
\draw(6,2.3) node {};
\draw[-] (0,1)--(1,1);
\draw[-] (1,1)--(2,1);
\draw[-] (2,1)--(3,1);
\draw[-] (3,1)--(4,2);
\draw[-] (4,2)--(5,1);
\draw[-] (5,1)--(6,1);
\draw[-] (6,1)--(7,1);
\draw[-] (7,1)--(8,1);
\draw[-,red] (1,1)--(4,3);
\draw[-,red] (4,3)--(7,1);
\draw[-,line width=1.5pt,dotted] (4,2)--(4,3);
\filldraw[red]  (4,3)    circle (2pt);
\filldraw [black]  (0,1)    circle (2pt)
[black]  (1,1)    circle (2pt)
[black]  (2,1)    circle (2pt)
[black]  (3,1)    circle (2pt)
[black]  (4,2)    circle (2pt)
[black]  (5,1)    circle (2pt)
[black]  (6,1)    circle (2pt)
[black]  (7,1)    circle (2pt)
[black]  (8,1)    circle (2pt);
\draw (0,0.6) node {$p_1$};
\draw (1,0.6) node {$p_2$};
\draw (2,0.6) node {$p_3$};
\draw (3,0.6) node {$p_4$};
\draw (3.6,2.1) node {$p_{5}$};
\draw (5,0.6) node {$p_6$};
\draw (6,0.6) node {$p_7$};
\draw (7,0.6) node {$p_8$};
\draw (8,0.6) node {$p_9$};
\draw (4.3,3.2) node {$q$};
\draw (8.5,3) node {};
\end{tikzpicture}
\begin{tikzpicture}[scale=0.5]
  \draw(3.5,-0.5) node {$(i,j,k)=(5,3,7)$};
\draw(6,2.3) node {};
\draw[-] (0,1)--(1,1);
\draw[-] (1,1)--(2,1);
\draw[-] (2,1)--(3,1);
\draw[-] (3,1)--(4,2);
\draw[-] (4,2)--(5,1);
\draw[-] (5,1)--(6,1);
\draw[-] (6,1)--(7,1);
\draw[-] (7,1)--(8,1);
\draw[-,red] (2,1)--(4,3);
\draw[-,red] (4,3)--(6,1);
\draw[-,line width=1.5pt,dotted] (4,2)--(4,3);
\filldraw[red]  (4,3)    circle (2pt);
\filldraw [black]  (0,1)    circle (2pt)
[black]  (1,1)    circle (2pt)
[black]  (2,1)    circle (2pt)
[black]  (3,1)    circle (2pt)
[black]  (4,2)    circle (2pt)
[black]  (5,1)    circle (2pt)
[black]  (6,1)    circle (2pt)
[black]  (7,1)    circle (2pt)
[black]  (8,1)    circle (2pt);
\draw (0,0.6) node {$p_1$};
\draw (1,0.6) node {$p_2$};
\draw (2,0.6) node {$p_3$};
\draw (3,0.6) node {$p_4$};
\draw (3.6,2.1) node {$p_{5}$};
\draw (5,0.6) node {$p_6$};
\draw (6,0.6) node {$p_7$};
\draw (7,0.6) node {$p_8$};
\draw (8,0.6) node {$p_9$};
\draw (4.3,3.2) node {$q$};
\draw (8.5,3) node {};
\end{tikzpicture}

\begin{itemize}
  \item[$\bullet\bullet$] Assume that $(i,j,k)=(5,2,7)$.
\end{itemize}

We know that $p_6$ cannot be an endpoint of $Q$, since then we could continue $Q$ with the edge $p_6p_5$.
  Thus two edges of the path $Q$ connect $p_6$ directly with two vertices in the set
  $$
  (V(P)\cap V(Q))\setminus\{p_3\}=\{p_1,p_2,p_3,p_4,p_7,p_8,p_9\}.
  $$
  By Proposition~\ref{forbidden pairs} an edge of $Q$ cannot connect $p_6=p_{i+1}$ with $p_1$, $p_3=p_{j+1}$, $p_4=p_{i-1}$,
  $p_8=p_{k+1}$ nor $p_9=p_{n-1}$.

Thus two edges of the path $Q$ connect $p_6$ with $p_2$ and $p_6$ with $p_7$, which implies that $Q$ contains the 4-cycle $p_2 p_6  p_7 q p_2$,
a contradiction that discards
the case $i=5$, $j=2$ and $k=7$.

\begin{tikzpicture}[scale=0.8]
  \draw(4,-0.5) node {$Q$ contains a 4-cycle};
\draw(6,2.3) node {};
\draw[-] (0,1)--(1,1);
\draw[-] (1,1)--(2,1);
\draw[-] (2,1)--(3,1);
\draw[-] (3,1)--(4,2);
\draw[-] (4,2)--(5,1);
\draw[-] (5,1)--(6,1);
\draw[-] (6,1)--(7,1);
\draw[-] (7,1)--(8,1);
\draw[-,red,line width=2pt] (1,1)--(4,3);
\draw[-,red,line width=2pt] (4,3)--(6,1);
\draw[-,line width=1.5,dotted] (4,2)--(4,3);
\draw[-,red,line width=2pt] (1,1)..controls(1,0)and(5,0)..(5,1);
\draw[-,red,line width=2pt] (5,1)..controls(5,0.5)and(6,0.5)..(6,1);
\filldraw[red]  (4,3)    circle (2pt);
\filldraw [black]  (0,1)    circle (2pt)
[black]  (1,1)    circle (2pt)
[black]  (2,1)    circle (2pt)
[black]  (3,1)    circle (2pt)
[black]  (4,2)    circle (2pt)
[black]  (5,1)    circle (2pt)
[black]  (6,1)    circle (2pt)
[black]  (7,1)    circle (2pt)
[black]  (8,1)    circle (2pt);
\draw (0,0.6) node {$p_1$};
\draw (0.7,0.6) node {$p_2$};
\draw (2,0.6) node {$p_3$};
\draw (3,0.6) node {$p_4$};
\draw (3.6,2.1) node {$p_{5}$};
\draw (4.6,0.9) node {$p_6$};
\draw (6.3,0.6) node {$p_7$};
\draw (7,0.6) node {$p_8$};
\draw (8,0.6) node {$p_9$};
\draw (4.2,3.2) node {$q$};
\draw (8.5,3) node {};
\end{tikzpicture}

\begin{itemize}
  \item[$\bullet\bullet$] Assume that $(i,j,k)=(5,3,7)$.
\end{itemize}

We know that $p_6$ cannot be an endpoint of $Q$, since then we could continue $Q$ with the edge $p_6p_5$.
  Thus two edges of the path $Q$ connect $p_6$ directly with two vertices in the set
  $$
  (V(P)\cap V(Q))\setminus\{p_3\}=\{p_1,p_2,p_3,p_4,p_7,p_8,p_9\}.
  $$
 Proposition~\ref{forbidden pairs}  shows that
 an edge of $Q$ cannot connect $p_6=p_{i+1}$ with $p_1$, $p_4=p_{i-1}$,
  $p_8=p_{k+1}$ nor $p_9=p_{n-1}$;
and the Hamiltonian path in the first diagram shows that
 an edge of $Q$ cannot connect $p_6$ with $p_2$.

\noindent \begin{tikzpicture}[scale=0.8]
  \draw(4,-0.5) node {$Q$ can't connect $p_2$ with $p_6$};
\draw(6,2.3) node {};
\draw[-,line width=2pt,cyan] (0,1)--(1,1);
\draw[-] (1,1)--(2,1);
\draw[-,line width=2pt,cyan] (2,1)--(3,1);
\draw[-,line width=2pt,cyan] (3,1)--(4,2);
\draw[-,line width=2pt,cyan] (4,2)--(5,1);
\draw[-] (5,1)--(6,1);
\draw[-,line width=2pt,cyan] (6,1)--(7,1);
\draw[-,line width=2pt,cyan] (7,1)--(8,1);
\draw[-,line width=2pt,cyan] (2,1)--(4,3);
\draw[-,line width=2pt,cyan] (4,3)--(6,1);
\draw[-,line width=1.5pt,dotted] (4,2)--(4,3);
\draw[-,line width=2pt,cyan] (1,1)..controls(1,-0.2)and(5,0)..(5,1);
\filldraw[red]  (4,3)    circle (2pt);
\filldraw [black]  (0,1)    circle (2pt)
[black]  (1,1)    circle (2pt)
[black]  (2,1)    circle (2pt)
[black]  (3,1)    circle (2pt)
[black]  (4,2)    circle (2pt)
[black]  (5,1)    circle (2pt)
[black]  (6,1)    circle (2pt)
[black]  (7,1)    circle (2pt)
[black]  (8,1)    circle (2pt);
\draw (0,0.6) node {$p_1$};
\draw (0.7,0.6) node {$p_2$};
\draw (2,0.6) node {$p_3$};
\draw (3,0.6) node {$p_4$};
\draw (3.6,2.1) node {$p_{5}$};
\draw (5.3,0.6) node {$p_6$};
\draw (6,0.6) node {$p_7$};
\draw (7,0.6) node {$p_8$};
\draw (8,0.6) node {$p_9$};
\draw (4.2,3.2) node {$q$};
\draw (8.5,3) node {};
\end{tikzpicture}
\begin{tikzpicture}[scale=0.8]
  \draw(4,-0.5) node {$Q$ contains a 4-cycle};
\draw(6,2.3) node {};
\draw[-] (0,1)--(1,1);
\draw[-] (1,1)--(2,1);
\draw[-] (2,1)--(3,1);
\draw[-] (3,1)--(4,2);
\draw[-] (4,2)--(5,1);
\draw[-] (5,1)--(6,1);
\draw[-] (6,1)--(7,1);
\draw[-] (7,1)--(8,1);
\draw[-,red,line width=2pt] (2,1)--(4,3);
\draw[-,red,line width=2pt] (4,3)--(6,1);
\draw[-,line width=1.5,dotted] (4,2)--(4,3);
\draw[-,red,line width=2pt] (2,1)..controls(2,0)and(5,0)..(5,1);
\draw[-,red,line width=2pt] (5,1)..controls(5,0.5)and(6,0.5)..(6,1);
\filldraw[red]  (4,3)    circle (2pt);
\filldraw [black]  (0,1)    circle (2pt)
[black]  (1,1)    circle (2pt)
[black]  (2,1)    circle (2pt)
[black]  (3,1)    circle (2pt)
[black]  (4,2)    circle (2pt)
[black]  (5,1)    circle (2pt)
[black]  (6,1)    circle (2pt)
[black]  (7,1)    circle (2pt)
[black]  (8,1)    circle (2pt);
\draw (0,0.6) node {$p_1$};
\draw (1,0.6) node {$p_2$};
\draw (1.7,0.6) node {$p_3$};
\draw (3,0.6) node {$p_4$};
\draw (3.6,2.1) node {$p_{5}$};
\draw (4.6,0.9) node {$p_6$};
\draw (6.3,0.6) node {$p_7$};
\draw (7,0.6) node {$p_8$};
\draw (8,0.6) node {$p_9$};
\draw (4.2,3.2) node {$q$};
\draw (8.5,3) node {};
\end{tikzpicture}

Thus two edges of the path $Q$ connect $p_6$ with $p_3$ and $p_6$ with $p_7$, which implies that $Q$ contains the 4-cycle $p_3 p_6  p_7 q p_3$,
a contradiction that discards
the case $i=5$, $j=3$ and $k=7$.

\begin{itemize}
  \item[$\bullet\bullet$] Assume that $(i,j,k)=(5,2,8)$.
\end{itemize}

In this case we consider the connections of $p_1$ and $p_9$. By Proposition~\ref{forbidden pairs}
an edge of $Q$ cannot connect $p_1$ with one of $\{p_3,p_4,p_6,p_7,p_9\}$,
nor can it connect $p_9$ with one of $\{p_1,p_3,p_4,p_6,p_7\}$.

Thus $p_1$ can be connected only to $p_2$ or $p_8$; and similarly $p_9$ can be connected only to $p_2$ or $p_8$.
Hence, one of $\{p_1,p_9\}$ is connected with $p_2$ and the other with $p_8$ and both
$p_1$ and $p_9$ must be endpoints. In each case, $p_1\sim p_2$, $p_9\sim p_8$ or  $p_1\sim p_8$, $p_9\sim p_2$,
we obtain that $Q$ has length 4, which is a contradiction.

\noindent 
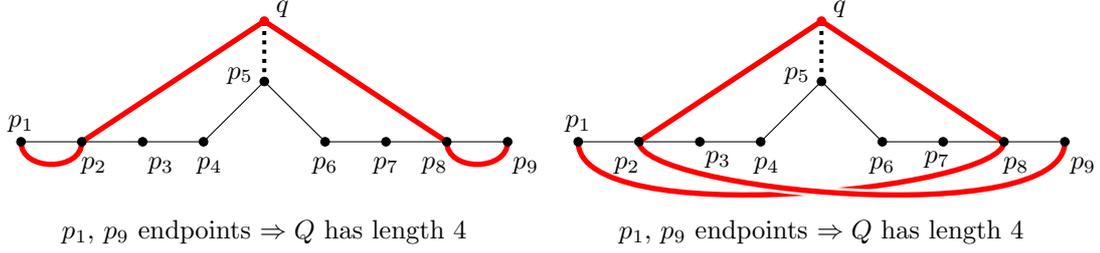
\begin{tikzpicture}[scale=0.8]
  \draw(4,-0.5) node {$p_1$, $p_9$ endpoints $\Rightarrow Q$ has length 4};
\draw(6,2.3) node {};
\draw[-] (0,1)--(1,1);
\draw[-] (1,1)--(2,1);
\draw[-] (2,1)--(3,1);
\draw[-] (3,1)--(4,2);
\draw[-] (4,2)--(5,1);
\draw[-] (5,1)--(6,1);
\draw[-] (6,1)--(7,1);
\draw[-] (7,1)--(8,1);
\draw[-,red,line width=2pt] (1,1)--(4,3);
\draw[-,red,line width=2pt] (4,3)--(7,1);
\draw[-,dotted,line width=1.5pt] (4,2)--(4,3);
\draw[-,red,line width=2pt] (0,1)..controls(0,0.5)and(1,0.5)..(1,1);
\draw[-,red,line width=2pt] (7,1)..controls(7,0.5)and(8,0.5)..(8,1);
\filldraw[red]  (4,3)    circle (2pt);
\filldraw [black]  (0,1)    circle (2pt)
[black]  (1,1)    circle (2pt)
[black]  (2,1)    circle (2pt)
[black]  (3,1)    circle (2pt)
[black]  (4,2)    circle (2pt)
[black]  (5,1)    circle (2pt)
[black]  (6,1)    circle (2pt)
[black]  (7,1)    circle (2pt)
[black]  (8,1)    circle (2pt);
\draw (0,1.3) node {$p_1$};
\draw (1.2,0.6) node {$p_2$};
\draw (2.3,0.6) node {$p_3$};
\draw (3.1,0.6) node {$p_4$};
\draw (3.6,2.1) node {$p_{5}$};
\draw (5,0.6) node {$p_6$};
\draw (6,0.6) node {$p_7$};
\draw (6.8,0.6) node {$p_8$};
\draw (8.3,0.6) node {$p_9$};
\draw (4.3,3.2) node {$q$};
\draw (8.5,3) node {};
\end{tikzpicture}
\begin{tikzpicture}[scale=0.8]
  \draw(4,-0.5) node {$p_1$, $p_9$ endpoints $\Rightarrow Q$ has length 4};
\draw(6,2.3) node {};
\draw[-] (0,1)--(1,1);
\draw[-] (1,1)--(2,1);
\draw[-] (2,1)--(3,1);
\draw[-] (3,1)--(4,2);
\draw[-] (4,2)--(5,1);
\draw[-] (5,1)--(6,1);
\draw[-] (6,1)--(7,1);
\draw[-] (7,1)--(8,1);
\draw[-,red,line width=2pt] (1,1)--(4,3);
\draw[-,red,line width=2pt] (4,3)--(7,1);
\draw[-,dotted,line width=1.5pt] (4,2)--(4,3);
\draw[-,red,line width=2pt] (0,1)..controls(0,-0.5)and(7,0.2)..(7,1);
\draw[-,white,line width=4pt] (1,1)..controls(1,0.2)and(8,-0.5)..(8,1);
\draw[-,red,line width=2pt] (1,1)..controls(1,0.2)and(8,-0.5)..(8,1);
\filldraw[red]  (4,3)    circle (2pt);
\filldraw [black]  (0,1)    circle (2pt)
[black]  (1,1)    circle (2pt)
[black]  (2,1)    circle (2pt)
[black]  (3,1)    circle (2pt)
[black]  (4,2)    circle (2pt)
[black]  (5,1)    circle (2pt)
[black]  (6,1)    circle (2pt)
[black]  (7,1)    circle (2pt)
[black]  (8,1)    circle (2pt);
\draw (0,1.3) node {$p_1$};
\draw (0.8,0.6) node {$p_2$};
\draw (2.3,0.7) node {$p_3$};
\draw (3.1,0.6) node {$p_4$};
\draw (3.6,2.1) node {$p_{5}$};
\draw (5,0.6) node {$p_6$};
\draw (5.9,0.7) node {$p_7$};
\draw (7.2,0.6) node {$p_8$};
\draw (8.3,0.6) node {$p_9$};
\draw (4.3,3.2) node {$q$};
\draw (8.5,3) node {};
\end{tikzpicture}

This contradiction discards the case $(i,j,k)=(5,2,8)$, finishing the case $i=5$ and thus we have proved that $\ell=8$ and $n=10$ is impossible.

\section{The case $\ell=6$ and $n=10$}
\label{seccion 6 mas 4}
In this section we will discard the case $\ell=6$ and $n=10$. We will assume that there exists a minimal graph $G$ with two
longest paths $P$ and $Q$, such that $|V(P)\cap V(Q)|=6$, $n(G)=10$ and so $|V(P)|=|V(Q)|=8$, and we will show that this is impossible.

For any set of vertices $A$ in $G$, we denote by $N(A)$ the set of neighbors of vertices in $A$,
that is $N(A)=\{u \in V(G): \mbox{there exists a neighbor of }u \mbox{ in } A\}$.

Let \begin{align*}
     & P'=V(P) \setminus V(Q), && Q'=V(Q) \setminus V(P), & S=V(P) \cap V(Q), \\
     & P''=S \cap N(P'), && Q''=S \cap N(Q').  &
    \end{align*}
Then $|S|=6$, and $|P'|=|Q'|=2$.

\begin{proposition}
We may assume that both $P$ and $Q$ have no endpoints in $P'$ and $Q'$ respectively.
\end{proposition}

\begin{proof}
Suppose for a moment that $P$ has one extreme in $P'$, and let $p$ be that extreme.
As $P' \cup	 Q'$ is connected, $p$ has at
least one neighbor in $P' \cup Q'$. If such a neighbor is in $Q'$, then we can add this vertex to $P$, obtaining a path longer than $P$, a contradiction. Hence,
$p$ has a neighbor, say $p_1$ in $P'$.
Thus the graph induced by $P'$ contains an edge $pp_1$.

Suppose for a moment that $pp_1 \notin E(P)$.
Let $r$ be the neighbor of $p_1$ in $P$, closest to $p$ in $P$.
Then $P+pp_1-p_1r$ is a path with the same vertex set as $P$ with its two extremes in $S$, and we are in another case.

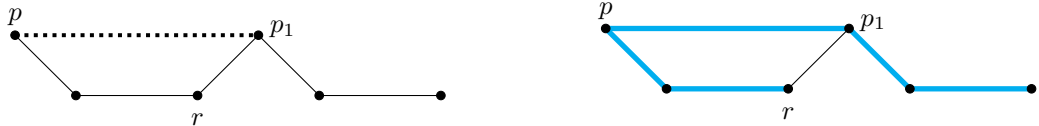
\begin{tikzpicture}[scale=0.8]
\draw(6,2.3) node {};
\draw[-] (0,2)--(1,1);
\draw[-] (1,1)--(3,1);
\draw[-] (3,1)--(4,2);
\draw[-] (4,2)--(5,1);
\draw[-] (5,1)--(7,1);
\draw[-,line width=1.5pt,dotted] (0,2)--(4,2);
\filldraw [black]  (0,2)    circle (2pt)
[black]  (1,1)    circle (2pt)
[black]  (3,1)    circle (2pt)
[black]  (4,2)    circle (2pt)
[black]  (5,1)    circle (2pt)
[black]  (7,1)    circle (2pt);
\draw (0,2.3) node {$p$};
\draw (3,0.6) node {$r$};
\draw (4.4,2.1) node {$p_1$};
\draw (8.5,0) node {};
\end{tikzpicture}
\begin{tikzpicture}[scale=0.8]
\draw(3.5,0.1) node {In $P+pp_1-p_1r$ the endpoints are not in $P'$};
\draw[-,line width=2pt,cyan] (0,2)--(1,1);
\draw[-,line width=2pt,cyan] (1,1)--(3,1);
\draw[-] (3,1)--(4,2);
\draw[-,line width=2pt,cyan] (4,2)--(5,1);
\draw[-,line width=2pt,cyan] (5,1)--(7,1);
\draw[-,line width=2pt,cyan] (0,2)--(4,2);
\filldraw [black]  (0,2)    circle (2pt)
[black]  (1,1)    circle (2pt)
[black]  (3,1)    circle (2pt)
[black]  (4,2)    circle (2pt)
[black]  (5,1)    circle (2pt)
[black]  (7,1)    circle (2pt);
\draw (0,2.3) node {$p$};
\draw (3,0.6) node {$r$};
\draw (4.4,2.1) node {$p_1$};
\draw (8.5,3) node {};
\end{tikzpicture}

Hence, we may assume that $pp_1 \in E(P)$.

Since $P'\cap Q'$ is connected, $p_1$ has a
 neighbor $q_1$ in $Q'$. Let $q$ be the other vertex of $Q'$.
If $q$ is adjacent to $p$ or $q_1$ then we can augment the path $P$. Hence, $qp_1$ is an edge. Thus,
the graph induced by $P' \cup	Q'$ is formed by the edges $pp_1$,
$p_1q_1$ and $p_1q$ (Figure).

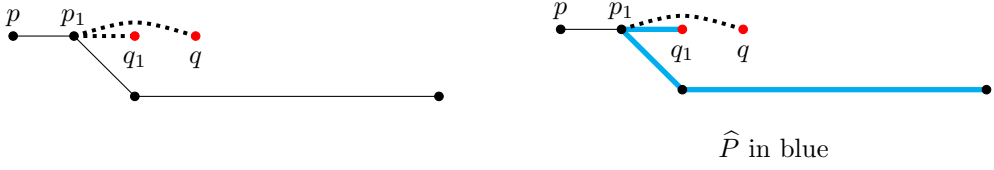
\begin{tikzpicture}[scale=0.8]
\draw(6,0) node {};
\draw[-] (0,2)--(1,2);
\draw[-] (1,2)--(2,1);
\draw[-] (2,1)--(7,1);
\draw[-,line width=1.5pt,dotted] (1,2)--(2,2);
\draw[-,line width=1.5pt,dotted] (1,2)..controls (2,2.3)..(3,2);
\filldraw [black]  (0,2)    circle (2pt)
[black]  (1,2)    circle (2pt)
[black]  (2,1)    circle (2pt)
[black]  (7,1)    circle (2pt);
\filldraw[red]  (2,2)    circle (2pt)
[red]  (3,2)    circle (2pt);

\draw (0,2.3) node {$p$};
\draw (1,2.3) node {$p_1$};
\draw (2,1.6) node {$q_1$};
\draw (3,1.6) node {$q$};
\draw (8.5,3) node {};
\end{tikzpicture}
\begin{tikzpicture}[scale=0.8]
\draw(3.5,0.1) node {$\widehat{P}$ in blue};
\draw[-] (0,2)--(1,2);
\draw[-,line width=2pt,cyan] (1,2)--(2,1);
\draw[-,line width=2pt,cyan] (2,1)--(7,1);
\draw[-,line width=2pt,cyan] (1,2)--(2,2);
\draw[-,line width=1.5pt,dotted] (1,2)..controls (2,2.3)..(3,2);
\filldraw [black]  (0,2)    circle (2pt)
[black]  (1,2)    circle (2pt)
[black]  (2,1)    circle (2pt)
[black]  (7,1)    circle (2pt);
\filldraw[red]  (2,2)    circle (2pt)
[red]  (3,2)    circle (2pt);

\draw (0,2.3) node {$p$};
\draw (1,2.3) node {$p_1$};
\draw (2,1.6) node {$q_1$};
\draw (3,1.6) node {$q$};
\draw (8.5,3) node {};
\end{tikzpicture}


Consider the path $\hat{P}=P-pp_1+p_1q_1$
and the graph $\hat{G}=G-p$.
It is clear that both $\hat{P}$ and $Q$ are longest paths in $\hat{G}$.
Note also that $\hat{P} \setminus Q=\{p_1\}$
and $Q \setminus \hat{P}=\{q\}$
and that $p_1q$ is an edge in $\hat{G}$.
As $n(\hat{G}) <  n(G)$ and
$|V(\hat{P}) \cap V(Q)|>|V(P) \cap V(Q)|$,
this is a contradiction to the minimality of
$G$.

The proof is similar for $Q$ instead of $P$.
\end{proof}

\begin{theorem} \label{teorema principal}
The case $\ell=6$ and $n=10$ is impossible.
\end{theorem}

\begin{proof}
As $P' \cup Q'$ is connected, we have, without loss of generality, three cases.

\textbf{Case 1}:

$P'=ab$ is an edge of $P$, $Q'=cd$ is an edge and $ac$ is an edge.

In that situation, $P\cap S$ consists on two paths $P_1$ and $P_2$.
Let $p_1$ and $p_2$ be the extremes of $P_1$ and $P_2$ that are not in $P''$.
Let $c'$ and $d'$ be the vertices adjacent to $c$ and $d$ that are not in $Q'$.
Then
$$
{d_P(c',p_1),d_P(c',p_2),d_P(d',p_1),d_P(d',p_2) \geq 2}.
$$
Suppose for a moment that both $c'$ and $d'$ are in the same subpath of $P$, say $P_1$.
Without loss of generality suppose that $d$ is closer to $p_1$ than $c$.
Note that $d_P(c',d')\geq 3$ (otherwise we can augment $P$). Then we obtain the contradiction.

$$
|P_1|\geq d_P(c',p_1) = d_P(c',d')+d_P(d',p_1)\geq 5.
$$

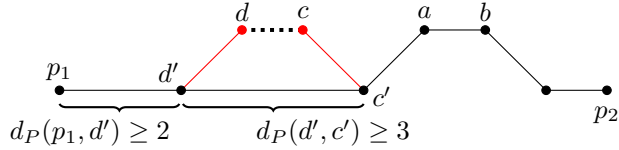
\begin{tikzpicture}[scale=0.8]
\draw(-3.5,0.1) node {};
\draw[-] (0,1)--(5,1);
\draw[-] (5,1)--(6,2);
\draw[-] (6,2)--(7,2);
\draw[-] (7,2)--(8,1);
\draw[-] (8,1)--(9,1);
\draw[-,red] (2,1)--(3,2);
\draw[-,red] (4,2)--(5,1);
\draw[-,line width=1.5pt,dotted] (3,2)--(4,2);
\filldraw [black]  (0,1)    circle (2pt)
[black]  (2,1)    circle (2pt)
[black]  (5,1)    circle (2pt)
[black]  (6,2)    circle (2pt)
[black]  (7,2)    circle (2pt)
[black]  (8,1)    circle (2pt)
[black]  (9,1)    circle (2pt);
\filldraw[red]  (3,2)    circle (2pt)
[red]  (4,2)    circle (2pt);
\draw [
    thick,
    decoration={
        brace,
        mirror,
        raise=0.5cm
    },
    decorate
](0,1.4)--(1.97,1.4);
\draw [
    thick,
    decoration={
        brace,
        mirror,
        raise=0.5cm
    },
    decorate
](2.03,1.4)--(5,1.4);
\draw (0,1.3) node {$p_1$};
\draw (1.8,1.3) node {$d'$};
\draw (3,2.3) node {$d$};
\draw (4,2.3) node {$c$};
\draw (5.3,0.9) node {$c'$};
\draw (6,2.3) node {$a$};
\draw (7,2.3) node {$b$};
\draw (9,0.6) node {$p_2$};
\draw (0.5,0.3) node {$d_P(p_1,d')\ge 2$};
\draw (4.5,0.3) node {$d_P(d',c')\ge 3$};

\draw (8.5,3) node {};
\end{tikzpicture}

Hence, without loss of generality, $c'$ is in $P_1$
and $d'$ is in $P_2$, which implies that $|P_1|,|P_2|\geq 2$ and $|P_1|=|P_2|=2$.

Without loss of generality, assume that
$P_1=u_0u_1u_2$, $P_2=v_0v_1v_2$,
$u_0$ is adjacent to $a$, and $v_0$ is adjacent to $b$
(see figure).

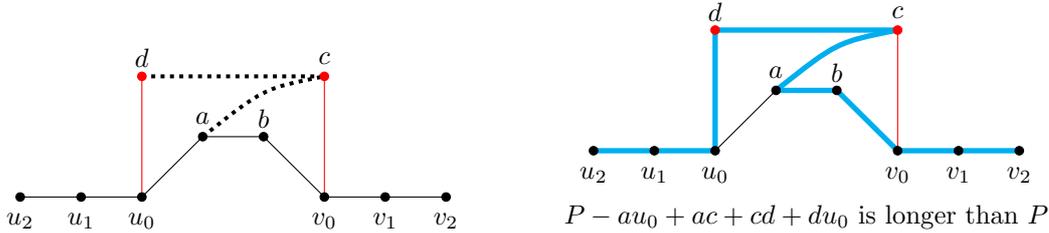
\begin{tikzpicture}[scale=0.8]
\draw(3.5,3.5) node {};
\draw[-] (0,1)--(2,1);
\draw[-] (2,1)--(3,2);
\draw[-] (3,2)--(4,2);
\draw[-] (4,2)--(5,1);
\draw[-] (5,1)--(7,1);
\draw[-,red] (2,1)--(2,3);
\draw[-,red] (5,3)--(5,1);
\draw[-,line width=1.5pt,dotted] (2,3)--(5,3);
\draw[-,line width=1.5pt,dotted] (3,2)..controls(4,2.8)..(5,3);
\filldraw [black]  (0,1)    circle (2pt)
[black]  (1,1)    circle (2pt)
[black]  (2,1)    circle (2pt)
[black]  (3,2)    circle (2pt)
[black]  (4,2)    circle (2pt)
[black]  (5,1)    circle (2pt)
[black]  (6,1)    circle (2pt)
[black]  (7,1)    circle (2pt);
\filldraw[red]  (2,3)    circle (2pt)
[red]  (5,3)    circle (2pt);
\draw (0,0.6) node {$u_2$};
\draw (1,0.6) node {$u_1$};
\draw (2,0.6) node {$u_0$};
\draw (5,0.6) node {$v_0$};
\draw (6,0.6) node {$v_1$};
\draw (7,0.6) node {$v_2$};
\draw (3,2.3) node {$a$};
\draw (4,2.3) node {$b$};
\draw (2,3.3) node {$d$};
\draw (5,3.3) node {$c$};
\draw (8.5,3) node {};
\end{tikzpicture}
\begin{tikzpicture}[scale=0.8]
\draw(3.5,3.5) node {};
\draw[-,line width=2pt,cyan] (0,1)--(2,1);
\draw[-] (2,1)--(3,2);
\draw[-,line width=2pt,cyan] (3,2)--(4,2);
\draw[-,line width=2pt,cyan] (4,2)--(5,1);
\draw[-,line width=2pt,cyan] (5,1)--(7,1);
\draw[-,line width=2pt,cyan] (2,1)--(2,3);
\draw[-,red] (5,3)--(5,1);
\draw[-,line width=2pt,cyan] (2,3)--(5,3);
\draw[-,line width=2pt,cyan] (3,2)..controls(4,2.8)..(5,3);
\filldraw [black]  (0,1)    circle (2pt)
[black]  (1,1)    circle (2pt)
[black]  (2,1)    circle (2pt)
[black]  (3,2)    circle (2pt)
[black]  (4,2)    circle (2pt)
[black]  (5,1)    circle (2pt)
[black]  (6,1)    circle (2pt)
[black]  (7,1)    circle (2pt);
\filldraw[red]  (2,3)    circle (2pt)
[red]  (5,3)    circle (2pt);
\draw (0,0.6) node {$u_2$};
\draw (1,0.6) node {$u_1$};
\draw (2,0.6) node {$u_0$};
\draw (5,0.6) node {$v_0$};
\draw (6,0.6) node {$v_1$};
\draw (7,0.6) node {$v_2$};
\draw (3,2.3) node {$a$};
\draw (4,2.3) node {$b$};
\draw (2,3.3) node {$d$};
\draw (5,3.3) node {$c$};
\draw (3.5,-0.1) node {$P - au_0 + ac + cd + du_0$ is longer than $P$};
\end{tikzpicture}

Note that $cu_0$ is not an edge, otherwise $P-au_0+ac+cu_0$ is a path longer than $P$.
As $d_P(c',u_2),d_P(c',v_2),d_P(d',u_2),d_P(d',v_2)
\geq 2$, we must have that $du_0$ and $cv_0$ are edges.
But then $P-au_0+ac+cd+du_0$ is a path longer than $P$, a contradiction.

\textbf{Case 2.}

$P'=p_ip_{i+1}$ is an edge of $P$, $Q'=\{c,d\}$ and $cd$ is not an edge in $G$.

Let $P=p_1p_2\dots p_8$. Then $i\in \{2,3,4,5,6\}$, and by symmetry it suffices to discard the cases $i=2$, $i=3$ and $i=4$.
Let $Q''=\{c_1,c_2,d_1,d_2\}$ be such that $cc_1,cc_2,dd_1,dd_2$ are edges of $Q$.

$\bullet$ If $i=2$, then $Q''\subset \{p_4,p_5,p_6,p_7\}$.
Moreover, $d_P(c_1,c_2)\ge 2$, $d_P(d_1,d_2)\ge 2$ and $\{c_1,c_2\}\ne \{d_1,d_2\}$, and so
there are only three possibilities:
\begin{itemize}
  \item[$\bullet\bullet$] $
\{\{c_1,c_2\},\{d_1,d_2\}\}=\{\{p_4,p_6\},\{p_4,p_7\}\},
{\vcenter{\hbox{
\begin{tikzpicture}[scale=0.5]
\draw(6,2.3) node {};
\draw[-] (0,1)--(1,2);
\draw[-] (1,2)--(2,2);
\draw[-] (2,2)--(3,1);
\draw[-] (3,1)--(7,1);
\draw[-,red] (3,1)--(4,2);
\draw[-,red] (5,1)--(4,2);
\draw[-,white,line width=2pt] (3,1)--(5,2);
\draw[-,red] (3,1)--(5,2);
\draw[-,red] (6,1)--(5,2);
\filldraw[red]  (4,2)    circle (2pt);
\filldraw[red]  (5,2)    circle (2pt);
\filldraw [black]  (0,1)    circle (2pt)
[black]  (1,2)    circle (2pt)
[black]  (2,2)    circle (2pt)
[black]  (3,1)    circle (2pt)
[black]  (4,1)    circle (2pt)
[black]  (5,1)    circle (2pt)
[black]  (6,1)    circle (2pt)
[black]  (7,1)    circle (2pt);
\draw (0,0.6) node {$p_1$};
\draw (0.7,2.3) node {$p_2$};
\draw (2.3,2.3) node {$p_3$};
\draw (3,0.6) node {$p_4$};
\draw (4,0.6) node {$p_{5}$};
\draw (5,0.6) node {$p_6$};
\draw (6,0.6) node {$p_7$};
\draw (7,0.6) node {$p_8$};
\draw (4.3,2.2) node {$c$};
\draw (5.3,2.2) node {$d$};
\draw (8.5,2.6) node {};
\end{tikzpicture}}}}
$
\item[$\bullet\bullet$] $
\{\{c_1,c_2\},\{d_1,d_2\}\}=\{\{p_4,p_6\},\{p_5,p_7\}\},
{\vcenter{\hbox{
\begin{tikzpicture}[scale=0.5]
\draw(6,2.3) node {};
\draw[-] (0,1)--(1,2);
\draw[-] (1,2)--(2,2);
\draw[-] (2,2)--(3,1);
\draw[-] (3,1)--(7,1);
\draw[-,red] (3,1)--(4,2);
\draw[-,red] (5,1)--(4,2);
\draw[-,white,line width=2pt] (4,1)--(5,2);
\draw[-,red] (4,1)--(5,2);
\draw[-,red] (6,1)--(5,2);
\filldraw[red]  (4,2)    circle (2pt);
\filldraw[red]  (5,2)    circle (2pt);
\filldraw [black]  (0,1)    circle (2pt)
[black]  (1,2)    circle (2pt)
[black]  (2,2)    circle (2pt)
[black]  (3,1)    circle (2pt)
[black]  (4,1)    circle (2pt)
[black]  (5,1)    circle (2pt)
[black]  (6,1)    circle (2pt)
[black]  (7,1)    circle (2pt);
\draw (0,0.6) node {$p_1$};
\draw (0.7,2.3) node {$p_2$};
\draw (2.3,2.3) node {$p_3$};
\draw (3,0.6) node {$p_4$};
\draw (4,0.6) node {$p_{5}$};
\draw (5,0.6) node {$p_6$};
\draw (6,0.6) node {$p_7$};
\draw (7,0.6) node {$p_8$};
\draw (4.3,2.2) node {$c$};
\draw (5.3,2.2) node {$d$};
\draw (8.5,2.6) node {};
\end{tikzpicture}}}}
$
\item[$\bullet\bullet$] $
\{\{c_1,c_2\},\{d_1,d_2\}\}=\{\{p_4,p_7\},\{p_5,p_7\}\},
{\vcenter{\hbox{
\begin{tikzpicture}[scale=0.5]
\draw(6,2.3) node {};
\draw[-] (0,1)--(1,2);
\draw[-] (1,2)--(2,2);
\draw[-] (2,2)--(3,1);
\draw[-] (3,1)--(7,1);
\draw[-,red] (3,1)--(4,2);
\draw[-,red] (6,1)--(4,2);
\draw[-,white,line width=2pt] (4,1)--(5,2);
\draw[-,red] (4,1)--(5,2);
\draw[-,red] (6,1)--(5,2);
\filldraw[red]  (4,2)    circle (2pt);
\filldraw[red]  (5,2)    circle (2pt);
\filldraw [black]  (0,1)    circle (2pt)
[black]  (1,2)    circle (2pt)
[black]  (2,2)    circle (2pt)
[black]  (3,1)    circle (2pt)
[black]  (4,1)    circle (2pt)
[black]  (5,1)    circle (2pt)
[black]  (6,1)    circle (2pt)
[black]  (7,1)    circle (2pt);
\draw (0,0.6) node {$p_1$};
\draw (0.7,2.3) node {$p_2$};
\draw (2.3,2.3) node {$p_3$};
\draw (3,0.6) node {$p_4$};
\draw (4,0.6) node {$p_{5}$};
\draw (5,0.6) node {$p_6$};
\draw (6,0.6) node {$p_7$};
\draw (7,0.6) node {$p_8$};
\draw (4.3,2.2) node {$c$};
\draw (5.3,2.2) node {$d$};
\draw (8.5,2.6) node {};
\end{tikzpicture}}}}
$
\end{itemize}

In the first case $c$ must be connected by an edge with $p_2$, since by assumption $P'\cup Q'$ is connected, $cd$ is not an edge, and
if $c$ is connected with $p_3$, then we can replace $p_3p_4$ with $p_3cp_4$, obtaining a longer path.
The following diagrams shows that with the edge $cp_2$ we obtain a path $\widehat P$, that is longer than $P$, discarding this case.

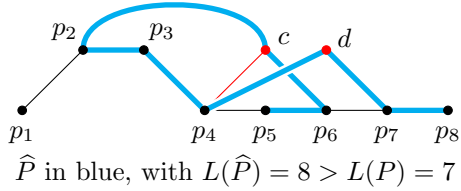
\begin{tikzpicture}[scale=0.8]
\draw(3.5,0) node {$\widehat{P}$ in blue, with $L(\widehat{P})=8>L(P)=7$};
\draw[-] (0,1)--(1,2);
\draw[-,line width=2pt,cyan] (1,2)--(2,2);
\draw[-,line width=2pt,cyan] (2,2)--(3,1);
\draw[-] (3,1)--(4,1);
\draw[-,line width=2pt,cyan] (4,1)--(5,1);
\draw[-] (5,1)--(6,1);
\draw[-,line width=2pt,cyan] (6,1)--(7,1);
\draw[-,red] (3,1)--(4,2);
\draw[-,line width=2pt,cyan] (5,1)--(4,2);
\draw[-,white,line width=4pt] (3,1)--(5,2);
\draw[-,line width=2pt,cyan] (3,1)--(5,2);
\draw[-,line width=2pt,cyan] (6,1)--(5,2);
\draw[-,line width=2pt,cyan] (1,2)..controls(1,3) and (4,3)..(4,2);
\filldraw[red]  (4,2)    circle (2pt);
\filldraw[red]  (5,2)    circle (2pt);
\filldraw [black]  (0,1)    circle (2pt)
[black]  (1,2)    circle (2pt)
[black]  (2,2)    circle (2pt)
[black]  (3,1)    circle (2pt)
[black]  (4,1)    circle (2pt)
[black]  (5,1)    circle (2pt)
[black]  (6,1)    circle (2pt)
[black]  (7,1)    circle (2pt);
\draw (0,0.6) node {$p_1$};
\draw (0.7,2.3) node {$p_2$};
\draw (2.3,2.3) node {$p_3$};
\draw (3,0.6) node {$p_4$};
\draw (4,0.6) node {$p_{5}$};
\draw (5,0.6) node {$p_6$};
\draw (6,0.6) node {$p_7$};
\draw (7,0.6) node {$p_8$};
\draw (4.3,2.2) node {$c$};
\draw (5.3,2.2) node {$d$};
\draw (8.5,2.6) node {};
\end{tikzpicture}

In the two other cases can be discarded by the paths $\widehat{P}$ in the following diagrams that are longer than $P$,
which concludes the case $i=2$.

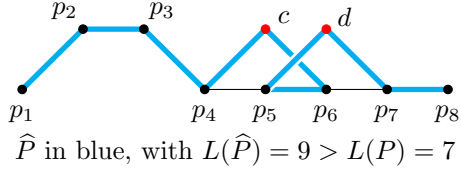
\begin{tikzpicture}[scale=0.8]
\draw(3.5,0) node {$\widehat{P}$ in blue, with $L(\widehat{P})=9>L(P)=7$};
\draw[-,line width=2pt,cyan] (0,1)--(1,2);
\draw[-,line width=2pt,cyan] (1,2)--(2,2);
\draw[-,line width=2pt,cyan] (2,2)--(3,1);
\draw[-] (3,1)--(4,1);
\draw[-,line width=2pt,cyan] (4,1)--(5,1);
\draw[-] (5,1)--(6,1);
\draw[-,line width=2pt,cyan] (6,1)--(7,1);
\draw[-,line width=2pt,cyan] (3,1)--(4,2);
\draw[-,line width=2pt,cyan] (5,1)--(4,2);
\draw[-,white,line width=4pt] (4,1)--(5,2);
\draw[-,line width=2pt,cyan] (4,1)--(5,2);
\draw[-,line width=2pt,cyan] (6,1)--(5,2);
\filldraw[red]  (4,2)    circle (2pt);
\filldraw[red]  (5,2)    circle (2pt);
\filldraw [black]  (0,1)    circle (2pt)
[black]  (1,2)    circle (2pt)
[black]  (2,2)    circle (2pt)
[black]  (3,1)    circle (2pt)
[black]  (4,1)    circle (2pt)
[black]  (5,1)    circle (2pt)
[black]  (6,1)    circle (2pt)
[black]  (7,1)    circle (2pt);
\draw (0,0.6) node {$p_1$};
\draw (0.7,2.3) node {$p_2$};
\draw (2.3,2.3) node {$p_3$};
\draw (3,0.6) node {$p_4$};
\draw (4,0.6) node {$p_{5}$};
\draw (5,0.6) node {$p_6$};
\draw (6,0.6) node {$p_7$};
\draw (7,0.6) node {$p_8$};
\draw (4.3,2.2) node {$c$};
\draw (5.3,2.2) node {$d$};
\draw (8.5,2.6) node {};
\end{tikzpicture}
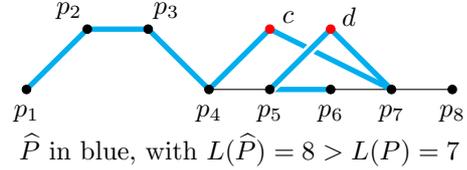
\begin{tikzpicture}[scale=0.8]
\draw(3.5,0) node {$\widehat{P}$ in blue, with $L(\widehat{P})=8>L(P)=7$};
\draw[-,line width=2pt,cyan] (0,1)--(1,2);
\draw[-,line width=2pt,cyan] (1,2)--(2,2);
\draw[-,line width=2pt,cyan] (2,2)--(3,1);
\draw[-] (3,1)--(4,1);
\draw[-,line width=2pt,cyan] (4,1)--(5,1);
\draw[-] (5,1)--(6,1);
\draw[-] (6,1)--(7,1);
\draw[-,line width=2pt,cyan] (3,1)--(4,2);
\draw[-,line width=2pt,cyan] (6,1)--(4,2);
\draw[-,white,line width=4pt] (4,1)--(5,2);
\draw[-,line width=2pt,cyan] (4,1)--(5,2);
\draw[-,line width=2pt,cyan] (6,1)--(5,2);
\filldraw[red]  (4,2)    circle (2pt);
\filldraw[red]  (5,2)    circle (2pt);
\filldraw [black]  (0,1)    circle (2pt)
[black]  (1,2)    circle (2pt)
[black]  (2,2)    circle (2pt)
[black]  (3,1)    circle (2pt)
[black]  (4,1)    circle (2pt)
[black]  (5,1)    circle (2pt)
[black]  (6,1)    circle (2pt)
[black]  (7,1)    circle (2pt);
\draw (0,0.6) node {$p_1$};
\draw (0.7,2.3) node {$p_2$};
\draw (2.3,2.3) node {$p_3$};
\draw (3,0.6) node {$p_4$};
\draw (4,0.6) node {$p_{5}$};
\draw (5,0.6) node {$p_6$};
\draw (6,0.6) node {$p_7$};
\draw (7,0.6) node {$p_8$};
\draw (4.3,2.2) node {$c$};
\draw (5.3,2.2) node {$d$};
\draw (8.5,2.6) node {};
\end{tikzpicture}

$\bullet$ If $i=3$, then $Q''\subset \{p_2,p_5,p_6,p_7\}$.
Moreover, $d_P(c_1,c_2)\ge 2$, $d_P(d_1,d_2)\ge 2$, $\{c_1,c_2\}\ne \{d_1,d_2\}$, and neither $\{c_1,c_2\}$ nor $\{d_1,d_2\}$
are equal to $\{p_2,p_5\}$, since then either $c$ or $d$ could not be connected with $p_3$ nor $p_4$.
Hence
there are only three possibilities:
\begin{itemize}
  \item[$\bullet\bullet$] $
\{\{c_1,c_2\},\{d_1,d_2\}\}=\{\{p_2,p_6\},\{p_2,p_7\}\},
{\vcenter{\hbox{
\begin{tikzpicture}[scale=0.5]
\draw(6,2.3) node {};
\draw[-] (0,1)--(1,1);
\draw[-] (1,1)--(2,2);
\draw[-] (2,2)--(3,2);
\draw[-] (4,1)--(7,1);
\draw[-,red] (1,1)--(4,2);
\draw[-,red] (1,1)--(6,2);
\draw[-,red] (6,1)--(6,2);
\draw[-,white,line width=2pt] (5,1)--(4,2);
\draw[-,red] (5,1)--(4,2);
\draw[-,white,line width=2pt] (3,2)--(4,1);
\draw[-] (3,2)--(4,1);
\filldraw[red]  (4,2)    circle (2pt);
\filldraw[red]  (6,2)    circle (2pt);
\filldraw [black]  (0,1)    circle (2pt)
[black]  (1,1)    circle (2pt)
[black]  (2,2)    circle (2pt)
[black]  (3,2)    circle (2pt)
[black]  (4,1)    circle (2pt)
[black]  (5,1)    circle (2pt)
[black]  (6,1)    circle (2pt)
[black]  (7,1)    circle (2pt);
\draw (0,0.6) node {$p_1$};
\draw (1,0.6) node {$p_2$};
\draw (1.7,2.3) node {$p_3$};
\draw (3.2,2.4) node {$p_4$};
\draw (4,0.6) node {$p_{5}$};
\draw (5,0.6) node {$p_6$};
\draw (6,0.6) node {$p_7$};
\draw (7,0.6) node {$p_8$};
\draw (4.3,2.2) node {$c$};
\draw (6.3,2.2) node {$d$};
\draw (8.5,2.6) node {};
\end{tikzpicture}}}}
$
\item[$\bullet\bullet$] $
\{\{c_1,c_2\},\{d_1,d_2\}\}=\{\{p_2,p_6\},\{p_5,p_7\}\},
{\vcenter{\hbox{
\begin{tikzpicture}[scale=0.5]
\draw(6,2.3) node {};
\draw[-] (0,1)--(1,1);
\draw[-] (1,1)--(2,2);
\draw[-] (2,2)--(3,2);
\draw[-] (4,1)--(7,1);
\draw[-,red] (1,1)--(4,2);
\draw[-,red] (4,1)--(6,2);
\draw[-,red] (6,1)--(6,2);
\draw[-,white,line width=2pt] (5,1)--(4,2);
\draw[-,red] (5,1)--(4,2);
\draw[-,white,line width=2pt] (3,2)--(4,1);
\draw[-] (3,2)--(4,1);
\filldraw[red]  (4,2)    circle (2pt);
\filldraw[red]  (6,2)    circle (2pt);
\filldraw [black]  (0,1)    circle (2pt)
[black]  (1,1)    circle (2pt)
[black]  (2,2)    circle (2pt)
[black]  (3,2)    circle (2pt)
[black]  (4,1)    circle (2pt)
[black]  (5,1)    circle (2pt)
[black]  (6,1)    circle (2pt)
[black]  (7,1)    circle (2pt);
\draw (0,0.6) node {$p_1$};
\draw (1,0.6) node {$p_2$};
\draw (1.7,2.3) node {$p_3$};
\draw (3.2,2.4) node {$p_4$};
\draw (4,0.6) node {$p_{5}$};
\draw (5,0.6) node {$p_6$};
\draw (6,0.6) node {$p_7$};
\draw (7,0.6) node {$p_8$};
\draw (4.3,2.2) node {$c$};
\draw (6.3,2.2) node {$d$};
\draw (8.5,2.6) node {};
\end{tikzpicture}}}}
$
\item[$\bullet\bullet$] $
\{\{c_1,c_2\},\{d_1,d_2\}\}=\{\{p_2,p_7\},\{p_5,p_7\}\},
{\vcenter{\hbox{
\begin{tikzpicture}[scale=0.5]
\draw(6,2.3) node {};
\draw[-] (0,1)--(1,1);
\draw[-] (1,1)--(2,2);
\draw[-] (2,2)--(3,2);
\draw[-] (4,1)--(7,1);
\draw[-,red] (1,1)--(4,2);
\draw[-,red] (4,1)--(6,2);
\draw[-,red] (6,1)--(6,2);
\draw[-,white,line width=2pt] (6,1)--(4,2);
\draw[-,red] (6,1)--(4,2);
\draw[-,white,line width=2pt] (3,2)--(4,1);
\draw[-] (3,2)--(4,1);
\filldraw[red]  (4,2)    circle (2pt);
\filldraw[red]  (6,2)    circle (2pt);
\filldraw [black]  (0,1)    circle (2pt)
[black]  (1,1)    circle (2pt)
[black]  (2,2)    circle (2pt)
[black]  (3,2)    circle (2pt)
[black]  (4,1)    circle (2pt)
[black]  (5,1)    circle (2pt)
[black]  (6,1)    circle (2pt)
[black]  (7,1)    circle (2pt);
\draw (0,0.6) node {$p_1$};
\draw (1,0.6) node {$p_2$};
\draw (1.7,2.3) node {$p_3$};
\draw (3.2,2.4) node {$p_4$};
\draw (4,0.6) node {$p_{5}$};
\draw (5,0.6) node {$p_6$};
\draw (6,0.6) node {$p_7$};
\draw (7,0.6) node {$p_8$};
\draw (4.3,2.2) node {$c$};
\draw (6.3,2.2) node {$d$};
\draw (8.5,2.6) node {};
\end{tikzpicture}}}}
$
\end{itemize}

The first two cases can be discarded by the paths $\widehat{P}$ in the following diagrams that are longer than $P$.

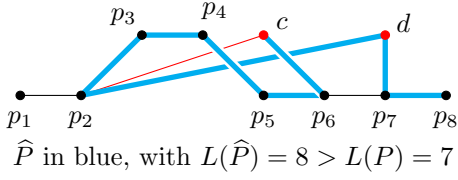
\begin{tikzpicture}[scale=0.8]
\draw(6,2.3) node {};
\draw[-] (0,1)--(1,1);
\draw[-,cyan,line width=2pt] (1,1)--(2,2);
\draw[-,cyan,line width=2pt] (2,2)--(3,2);
\draw[-,cyan,line width=2pt] (4,1)--(5,1);
\draw[-] (5,1)--(6,1);
\draw[-,cyan,line width=2pt] (6,1)--(7,1);
\draw[-,red] (1,1)--(4,2);
\draw[-,cyan,line width=2pt] (1,1)--(6,2);
\draw[-,cyan,line width=2pt] (6,1)--(6,2);
\draw[-,white,line width=4pt] (5,1)--(4,2);
\draw[-,cyan,line width=2pt] (5,1)--(4,2);
\draw[-,white,line width=4pt] (3,2)--(4,1);
\draw[-,cyan,line width=2pt] (3,2)--(4,1);
\filldraw[red]  (4,2)    circle (2pt);
\filldraw[red]  (6,2)    circle (2pt);
\filldraw [black]  (0,1)    circle (2pt)
[black]  (1,1)    circle (2pt)
[black]  (2,2)    circle (2pt)
[black]  (3,2)    circle (2pt)
[black]  (4,1)    circle (2pt)
[black]  (5,1)    circle (2pt)
[black]  (6,1)    circle (2pt)
[black]  (7,1)    circle (2pt);
\draw (0,0.6) node {$p_1$};
\draw (1,0.6) node {$p_2$};
\draw (1.7,2.3) node {$p_3$};
\draw (3.2,2.4) node {$p_4$};
\draw (4,0.6) node {$p_{5}$};
\draw (5,0.6) node {$p_6$};
\draw (6,0.6) node {$p_7$};
\draw (7,0.6) node {$p_8$};
\draw (4.3,2.2) node {$c$};
\draw (6.3,2.2) node {$d$};
\draw (8.5,2.6) node {};
\draw(3.5,0) node {$\widehat{P}$ in blue, with $L(\widehat{P})=8>L(P)=7$};
\end{tikzpicture}
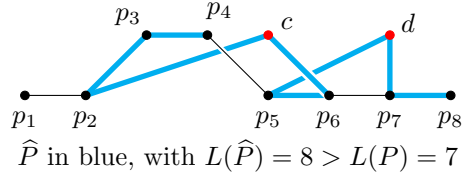
\begin{tikzpicture}[scale=0.8]
\draw(6,2.3) node {};
\draw[-] (0,1)--(1,1);
\draw[-,cyan,line width=2pt] (1,1)--(2,2);
\draw[-,cyan,line width=2pt] (2,2)--(3,2);
\draw[-,cyan,line width=2pt] (4,1)--(5,1);
\draw[-] (5,1)--(6,1);
\draw[-,cyan,line width=2pt] (6,1)--(7,1);
\draw[-,cyan,line width=2pt] (1,1)--(4,2);
\draw[-,cyan,line width=2pt] (4,1)--(6,2);
\draw[-,cyan,line width=2pt] (6,1)--(6,2);
\draw[-,white,line width=4pt] (5,1)--(4,2);
\draw[-,cyan,line width=2pt] (5,1)--(4,2);
\draw[-,white,line width=2pt] (3,2)--(4,1);
\draw[-] (3,2)--(4,1);
\filldraw[red]  (4,2)    circle (2pt);
\filldraw[red]  (6,2)    circle (2pt);
\filldraw [black]  (0,1)    circle (2pt)
[black]  (1,1)    circle (2pt)
[black]  (2,2)    circle (2pt)
[black]  (3,2)    circle (2pt)
[black]  (4,1)    circle (2pt)
[black]  (5,1)    circle (2pt)
[black]  (6,1)    circle (2pt)
[black]  (7,1)    circle (2pt);
\draw (0,0.6) node {$p_1$};
\draw (1,0.6) node {$p_2$};
\draw (1.7,2.3) node {$p_3$};
\draw (3.2,2.4) node {$p_4$};
\draw (4,0.6) node {$p_{5}$};
\draw (5,0.6) node {$p_6$};
\draw (6,0.6) node {$p_7$};
\draw (7,0.6) node {$p_8$};
\draw (4.3,2.2) node {$c$};
\draw (6.3,2.2) node {$d$};
\draw (8.5,2.6) node {};
\draw(3.5,0) node {$\widehat{P}$ in blue, with $L(\widehat{P})=8>L(P)=7$};
\end{tikzpicture}

In the third case $c$ must be connected by an edge with $p_4$, since by assumption $P'\cup Q'$ is connected, $cd$ is not an edge, and
if $c$ is connected with $p_3$, then we can replace $p_2p_3$ with $p_2cp_3$, obtaining a longer path.
The following diagrams shows that with the edge $cp_4$ we obtain a path $\widehat P$, that is longer than $P$, discarding the third case
and concluding the case $i=3$.

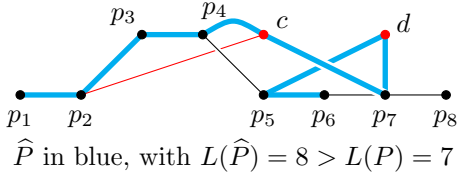
\begin{tikzpicture}[scale=0.8]
\draw(6,2.3) node {};
\draw[-,cyan,line width=2pt] (0,1)--(1,1);
\draw[-,cyan,line width=2pt] (1,1)--(2,2);
\draw[-,cyan,line width=2pt] (2,2)--(3,2);
\draw[-,cyan,line width=2pt] (4,1)--(5,1);
\draw[-] (5,1)--(6,1);
\draw[-] (6,1)--(7,1);
\draw[-,red] (1,1)--(4,2);
\draw[-,cyan,line width=2pt] (4,1)--(6,2);
\draw[-,cyan,line width=2pt] (6,1)--(6,2);
\draw[-,cyan,line width=2pt] (3,2)..controls (3.5,2.3)..(4,2);
\draw[-,white,line width=4pt] (6,1)--(4,2);
\draw[-,cyan,line width=2pt] (6,1)--(4,2);
\draw[-,white,line width=2pt] (3,2)--(4,1);
\draw[-] (3,2)--(4,1);
\filldraw[red]  (4,2)    circle (2pt);
\filldraw[red]  (6,2)    circle (2pt);
\filldraw [black]  (0,1)    circle (2pt)
[black]  (1,1)    circle (2pt)
[black]  (2,2)    circle (2pt)
[black]  (3,2)    circle (2pt)
[black]  (4,1)    circle (2pt)
[black]  (5,1)    circle (2pt)
[black]  (6,1)    circle (2pt)
[black]  (7,1)    circle (2pt);
\draw (0,0.6) node {$p_1$};
\draw (1,0.6) node {$p_2$};
\draw (1.7,2.3) node {$p_3$};
\draw (3.2,2.4) node {$p_4$};
\draw (4,0.6) node {$p_{5}$};
\draw (5,0.6) node {$p_6$};
\draw (6,0.6) node {$p_7$};
\draw (7,0.6) node {$p_8$};
\draw (4.3,2.2) node {$c$};
\draw (6.3,2.2) node {$d$};
\draw (8.5,2.6) node {};
\draw(3.5,0) node {$\widehat{P}$ in blue, with $L(\widehat{P})=8>L(P)=7$};
\end{tikzpicture}

$\bullet$ If $i=4$, then $Q''\subset \{p_2,p_3,p_6,p_7\}$.
Moreover, $d_P(c_1,c_2)\ge 2$, $d_P(d_1,d_2)\ge 2$, $\{c_1,c_2\}\ne \{d_1,d_2\}$, and neither $\{c_1,c_2\}$ nor $\{d_1,d_2\}$
are equal to $\{p_3,p_6\}$, since then either $c$ or $d$ could not be connected with $p_4$ nor $p_5$.
Hence
there are only three possibilities:

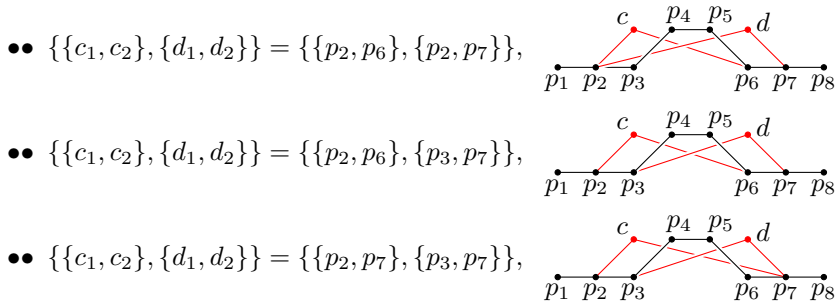
\begin{itemize}
  \item[$\bullet\bullet$] $
\{\{c_1,c_2\},\{d_1,d_2\}\}=\{\{p_2,p_6\},\{p_2,p_7\}\},
{\vcenter{\hbox{
\begin{tikzpicture}[scale=0.5]
\draw(6,2.3) node {};
\draw[-] (0,1)--(2,1);
\draw[-] (3,2)--(4,2);
\draw[-] (5,1)--(7,1);
\draw[-,red] (6,1)--(5,2);
\draw[-,red] (1,1)--(2,2);
\draw[-,red] (2,2)--(5,1);
\draw[-,white,line width=2pt] (1,1)--(5,2);
\draw[-,red] (1,1)--(5,2);
\draw[-,white,line width=2pt] (4,2)--(5,1);
\draw[-] (4,2)--(5,1);
\draw[-,white,line width=2pt] (2,1)--(3,2);
\draw[-] (2,1)--(3,2);
\filldraw[red]  (2,2)    circle (2pt);
\filldraw[red]  (5,2)    circle (2pt);
\filldraw [black]  (0,1)    circle (2pt)
[black]  (1,1)    circle (2pt)
[black]  (2,1)    circle (2pt)
[black]  (3,2)    circle (2pt)
[black]  (4,2)    circle (2pt)
[black]  (5,1)    circle (2pt)
[black]  (6,1)    circle (2pt)
[black]  (7,1)    circle (2pt);
\draw (0,0.6) node {$p_1$};
\draw (1,0.6) node {$p_2$};
\draw (2,0.6) node {$p_3$};
\draw (3.2,2.4) node {$p_4$};
\draw (4.3,2.4) node {$p_{5}$};
\draw (5,0.6) node {$p_6$};
\draw (6,0.6) node {$p_7$};
\draw (7,0.6) node {$p_8$};
\draw (1.7,2.3) node {$c$};
\draw (5.4,2.2) node {$d$};
\draw (8.5,2.6) node {};
\end{tikzpicture}}}}
$
  \item[$\bullet\bullet$] $
\{\{c_1,c_2\},\{d_1,d_2\}\}=\{\{p_2,p_6\},\{p_3,p_7\}\},
{\vcenter{\hbox{
\begin{tikzpicture}[scale=0.5]
\draw(6,2.3) node {};
\draw[-] (0,1)--(2,1);
\draw[-] (3,2)--(4,2);
\draw[-] (5,1)--(7,1);
\draw[-,red] (6,1)--(5,2);
\draw[-,red] (1,1)--(2,2);
\draw[-,red] (2,2)--(5,1);
\draw[-,white,line width=2pt] (2,1)--(5,2);
\draw[-,red] (2,1)--(5,2);
\draw[-,white,line width=2pt] (4,2)--(5,1);
\draw[-] (4,2)--(5,1);
\draw[-,white,line width=2pt] (2,1)--(3,2);
\draw[-] (2,1)--(3,2);
\filldraw[red]  (2,2)    circle (2pt);
\filldraw[red]  (5,2)    circle (2pt);
\filldraw [black]  (0,1)    circle (2pt)
[black]  (1,1)    circle (2pt)
[black]  (2,1)    circle (2pt)
[black]  (3,2)    circle (2pt)
[black]  (4,2)    circle (2pt)
[black]  (5,1)    circle (2pt)
[black]  (6,1)    circle (2pt)
[black]  (7,1)    circle (2pt);
\draw (0,0.6) node {$p_1$};
\draw (1,0.6) node {$p_2$};
\draw (2,0.6) node {$p_3$};
\draw (3.2,2.4) node {$p_4$};
\draw (4.3,2.4) node {$p_{5}$};
\draw (5,0.6) node {$p_6$};
\draw (6,0.6) node {$p_7$};
\draw (7,0.6) node {$p_8$};
\draw (1.7,2.3) node {$c$};
\draw (5.4,2.2) node {$d$};
\draw (8.5,2.6) node {};
\end{tikzpicture}}}}
$
  \item[$\bullet\bullet$] $
\{\{c_1,c_2\},\{d_1,d_2\}\}=\{\{p_2,p_7\},\{p_3,p_7\}\},
{\vcenter{\hbox{
\begin{tikzpicture}[scale=0.5]
\draw(6,2.3) node {};
\draw[-] (0,1)--(2,1);
\draw[-] (3,2)--(4,2);
\draw[-] (5,1)--(7,1);
\draw[-,red] (6,1)--(5,2);
\draw[-,red] (1,1)--(2,2);
\draw[-,red] (2,2)--(6,1);
\draw[-,white,line width=2pt] (2,1)--(5,2);
\draw[-,red] (2,1)--(5,2);
\draw[-,white,line width=2pt] (4,2)--(5,1);
\draw[-] (4,2)--(5,1);
\draw[-,white,line width=2pt] (2,1)--(3,2);
\draw[-] (2,1)--(3,2);
\filldraw[red]  (2,2)    circle (2pt);
\filldraw[red]  (5,2)    circle (2pt);
\filldraw [black]  (0,1)    circle (2pt)
[black]  (1,1)    circle (2pt)
[black]  (2,1)    circle (2pt)
[black]  (3,2)    circle (2pt)
[black]  (4,2)    circle (2pt)
[black]  (5,1)    circle (2pt)
[black]  (6,1)    circle (2pt)
[black]  (7,1)    circle (2pt);
\draw (0,0.6) node {$p_1$};
\draw (1,0.6) node {$p_2$};
\draw (2,0.6) node {$p_3$};
\draw (3.2,2.4) node {$p_4$};
\draw (4.3,2.4) node {$p_{5}$};
\draw (5,0.6) node {$p_6$};
\draw (6,0.6) node {$p_7$};
\draw (7,0.6) node {$p_8$};
\draw (1.7,2.3) node {$c$};
\draw (5.4,2.2) node {$d$};
\draw (8.5,2.6) node {};
\end{tikzpicture}}}}
$
\end{itemize}

The three cases can be discarded by the paths $\widehat{P}$ in the following diagrams that are longer than $P$.

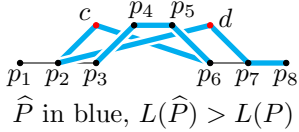
\noindent \begin{tikzpicture}[scale=0.5]
\draw(6,2.3) node {};
\draw[-] (0,1)--(2,1);
\draw[-,line width=2pt,cyan] (3,2)--(4,2);
\draw[-] (5,1)--(7,1);
\draw[-,line width=2pt,cyan] (6,1)--(7,1);
\draw[-,line width=2pt,cyan] (6,1)--(5,2);
\draw[-,line width=2pt,cyan] (1,1)--(2,2);
\draw[-,line width=2pt,cyan] (2,2)--(5,1);
\draw[-,white,line width=4pt] (1,1)--(5,2);
\draw[-,line width=2pt,cyan] (1,1)--(5,2);
\draw[-,white,line width=4pt] (4,2)--(5,1);
\draw[-,line width=2pt,cyan] (4,2)--(5,1);
\draw[-,white,line width=4pt] (2,1)--(3,2);
\draw[-,line width=2pt,cyan] (2,1)--(3,2);
\filldraw[red]  (2,2)    circle (2pt);
\filldraw[red]  (5,2)    circle (2pt);
\filldraw [black]  (0,1)    circle (2pt)
[black]  (1,1)    circle (2pt)
[black]  (2,1)    circle (2pt)
[black]  (3,2)    circle (2pt)
[black]  (4,2)    circle (2pt)
[black]  (5,1)    circle (2pt)
[black]  (6,1)    circle (2pt)
[black]  (7,1)    circle (2pt);
\draw (0,0.6) node {$p_1$};
\draw (1,0.6) node {$p_2$};
\draw (2,0.6) node {$p_3$};
\draw (3.2,2.4) node {$p_4$};
\draw (4.3,2.4) node {$p_{5}$};
\draw (5,0.6) node {$p_6$};
\draw (6,0.6) node {$p_7$};
\draw (7,0.6) node {$p_8$};
\draw (1.7,2.3) node {$c$};
\draw (5.4,2.2) node {$d$};
\draw (8.5,2.6) node {};
\draw(3.5,-0.3) node {$\widehat{P}$ in blue, $L(\widehat{P})>L(P)$};
\end{tikzpicture}
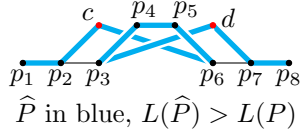
\begin{tikzpicture}[scale=0.5]
\draw(6,2.3) node {};
\draw[-,line width=2pt,cyan] (0,1)--(1,1);
\draw[-] (1,1)--(2,1);
\draw[-,line width=2pt,cyan] (3,2)--(4,2);
\draw[-] (5,1)--(7,1);
\draw[-,line width=2pt,cyan] (6,1)--(7,1);
\draw[-,line width=2pt,cyan] (6,1)--(5,2);
\draw[-,line width=2pt,cyan] (1,1)--(2,2);
\draw[-,line width=2pt,cyan] (2,2)--(5,1);
\draw[-,white,line width=4pt] (2,1)--(5,2);
\draw[-,line width=2pt,cyan] (2,1)--(5,2);
\draw[-,white,line width=4pt] (4,2)--(5,1);
\draw[-,line width=2pt,cyan] (4,2)--(5,1);
\draw[-,white,line width=4pt] (2,1)--(3,2);
\draw[-,line width=2pt,cyan] (2,1)--(3,2);
\filldraw[red]  (2,2)    circle (2pt);
\filldraw[red]  (5,2)    circle (2pt);
\filldraw [black]  (0,1)    circle (2pt)
[black]  (1,1)    circle (2pt)
[black]  (2,1)    circle (2pt)
[black]  (3,2)    circle (2pt)
[black]  (4,2)    circle (2pt)
[black]  (5,1)    circle (2pt)
[black]  (6,1)    circle (2pt)
[black]  (7,1)    circle (2pt);
\draw (0,0.6) node {$p_1$};
\draw (1,0.6) node {$p_2$};
\draw (2,0.6) node {$p_3$};
\draw (3.2,2.4) node {$p_4$};
\draw (4.3,2.4) node {$p_{5}$};
\draw (5,0.6) node {$p_6$};
\draw (6,0.6) node {$p_7$};
\draw (7,0.6) node {$p_8$};
\draw (1.7,2.3) node {$c$};
\draw (5.4,2.2) node {$d$};
\draw (8.5,2.6) node {};
\draw(3.5,-0.3) node {$\widehat{P}$ in blue, $L(\widehat{P})>L(P)$};
\end{tikzpicture}
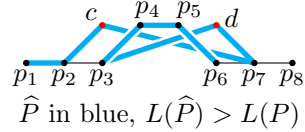
\begin{tikzpicture}[scale=0.5]
\draw(6,2.3) node {};
\draw[-,line width=2pt,cyan] (0,1)--(1,1);
\draw[-] (1,1)--(2,1);
\draw[-,line width=2pt,cyan] (3,2)--(4,2);
\draw[-] (5,1)--(7,1);
\draw[-,line width=2pt,cyan] (6,1)--(5,2);
\draw[-,line width=2pt,cyan] (1,1)--(2,2);
\draw[-,line width=2pt,cyan] (2,2)--(6,1);
\draw[-,white,line width=4pt] (2,1)--(5,2);
\draw[-,line width=2pt,cyan] (2,1)--(5,2);
\draw[-,white,line width=4pt] (4,2)--(5,1);
\draw[-,line width=2pt,cyan] (4,2)--(5,1);
\draw[-,white,line width=4pt] (2,1)--(3,2);
\draw[-,line width=2pt,cyan] (2,1)--(3,2);
\filldraw[red]  (2,2)    circle (2pt);
\filldraw[red]  (5,2)    circle (2pt);
\filldraw [black]  (0,1)    circle (2pt)
[black]  (1,1)    circle (2pt)
[black]  (2,1)    circle (2pt)
[black]  (3,2)    circle (2pt)
[black]  (4,2)    circle (2pt)
[black]  (5,1)    circle (2pt)
[black]  (6,1)    circle (2pt)
[black]  (7,1)    circle (2pt);
\draw (0,0.6) node {$p_1$};
\draw (1,0.6) node {$p_2$};
\draw (2,0.6) node {$p_3$};
\draw (3.2,2.4) node {$p_4$};
\draw (4.3,2.4) node {$p_{5}$};
\draw (5,0.6) node {$p_6$};
\draw (6,0.6) node {$p_7$};
\draw (7,0.6) node {$p_8$};
\draw (1.7,2.3) node {$c$};
\draw (5.4,2.2) node {$d$};
\draw (8.5,2.6) node {};
\draw(3.5,-0.3) node {$\widehat{P}$ in blue, $L(\widehat{P})>L(P)$};
\end{tikzpicture}

This discards the possibility that $i=4$ and concludes the Case 2.

\textbf{Case 3.}

 $P'=\{a,b\}$, such that $ab$ is not an edge of $P$,  $Q'=\{c,d\}$ and $cd$ is not an edge of $Q$.

Let $P=p_1p_2\dots p_8$ be the sequential order of the points in $P$. Assume that $P'=\{i,j\}$, and that $i<j$. Then we know that $i,j\notin \{1,8\}$
and that $j-i>1$. Thus
$$
\{i,j\}\in \{\{2,4\},\{2,5\},\{2,6\},\{2,7\},\{3,5\},\{3,6\},\{3,7\},\{4,6\},\{4,7\},\{5,7\}\}.
$$
By symmetry it suffices to discard the following cases
$$
\{i,j\}\in \{\{2,4\},\{2,5\},\{2,6\},\{2,7\},\{3,5\},\{3,6\}\}.
$$
Let $Q''=\{c_1,c_2,d_1,d_2\}$ be such that $cc_1,cc_2,dd_1,dd_2$ are edges of $Q$.
Then  $d_P(c_1,c_2)\ge 2$, $d_P(d_1,d_2)\ge 2$ and $\{c_1,c_2\}\ne \{d_1,d_2\}$.

$\bullet$ If $\{i,j\}=\{2,4\}$, then $Q''\subset \{p_3,p_5,p_6,p_7\}$.
If $p_3\in \{c_1,c_2\}$, then $c$ cannot be connected with $p_2$ nor with $p_4$ by an edge, since then, we could extend $P$. For example, if
$c_1=p_3$ and $c$ is connected with $p_2$, then we replace $p_2p_3$ by $p_2cp_3$ in $P$ and obtain a longer path. Similarly,
if $p_3\in \{d_1,d_2\}$, then $d$ cannot be connected with $p_2$ nor with $p_4$ by an edge. Hence, since $P'\cup Q'=\{p_2,p_4,c,d\}$ is connected,
$p_3$ cannot be in both $\{c_1,c_2\}$ and $\{d_1,d_2\}$. But it has to be in one of them, since otherwise $Q''\subset \{p_5,p_6,p_7\}$, which leads
to $\{c_1,c_2\}= \{d_1,d_2\}=\{p_5,p_7\}$, contradicting $\{c_1,c_2\}\ne \{d_1,d_2\}$.

Assume for example that  $p_3\in \{c_1,c_2\}$ but $p_3\notin \{d_1,d_2\}$. Then $c$
cannot be connected with $p_2$ nor with $p_4$ by an edge. Since $P'\cup Q'=\{p_2,p_4,c,d\}$ is connected, $c$ must be connected with
$d$ by an edge. Moreover, $\{d_1,d_2\}\subset \{p_5,p_6,p_7\}$, and so $\{d_1,d_2\}=\{p_5,p_7\}$, which yields the path
$\widehat{P}= p_1p_2p_3p_4p_5p_6p_7dc$, which is longer than $P$.

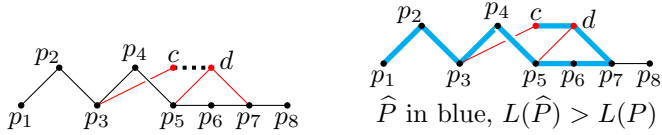
\begin{tikzpicture}[scale=0.5]
\draw(6,2.3) node {};
\draw[-] (0,1)--(1,2);
\draw[-] (1,2)--(2,1);
\draw[-] (2,1)--(3,2);
\draw[-] (4,1)--(5,1);
\draw[-] (5,1)--(6,1);
\draw[-] (6,1)--(7,1);

\draw[-,red] (2,1)--(4,2);
\draw[-,red] (4,1)--(5,2);
\draw[-,red] (5,2)--(6,1);
\draw[-,line width=1.5pt,dotted] (4,2)--(5,2);
\draw[-,white,line width=2pt] (3,2)--(4,1);
\draw[-] (3,2)--(4,1);
\filldraw[red]  (4,2)    circle (2pt);
\filldraw[red]  (5,2)    circle (2pt);
\filldraw [black]  (0,1)    circle (2pt)
[black]  (1,2)    circle (2pt)
[black]  (2,1)    circle (2pt)
[black]  (3,2)    circle (2pt)
[black]  (4,1)    circle (2pt)
[black]  (5,1)    circle (2pt)
[black]  (6,1)    circle (2pt)
[black]  (7,1)    circle (2pt);
\draw (0,0.6) node {$p_1$};
\draw (0.7,2.3) node {$p_2$};
\draw (2,0.6) node {$p_3$};
\draw (3,2.4) node {$p_4$};
\draw (4,0.6) node {$p_{5}$};
\draw (5,0.6) node {$p_6$};
\draw (6,0.6) node {$p_7$};
\draw (7,0.6) node {$p_8$};
\draw (4,2.3) node {$c$};
\draw (5.4,2.2) node {$d$};
\draw (8.5,2.6) node {};
\end{tikzpicture}
\begin{tikzpicture}[scale=0.5]
\draw(6,2.3) node {};
\draw[-,line width=2pt,cyan] (0,1)--(1,2);
\draw[-,line width=2pt,cyan] (1,2)--(2,1);
\draw[-,line width=2pt,cyan] (2,1)--(3,2);
\draw[-,line width=2pt,cyan] (4,1)--(5,1);
\draw[-,line width=2pt,cyan] (5,1)--(6,1);
\draw[-] (6,1)--(7,1);

\draw[-,red] (2,1)--(4,2);
\draw[-,red] (4,1)--(5,2);
\draw[-,line width=2pt,cyan] (5,2)--(6,1);
\draw[-,line width=2pt,cyan] (4,2)--(5,2);
\draw[-,white,line width=4pt] (3,2)--(4,1);
\draw[-,line width=2pt,cyan] (3,2)--(4,1);
\filldraw[red]  (4,2)    circle (2pt);
\filldraw[red]  (5,2)    circle (2pt);
\filldraw [black]  (0,1)    circle (2pt)
[black]  (1,2)    circle (2pt)
[black]  (2,1)    circle (2pt)
[black]  (3,2)    circle (2pt)
[black]  (4,1)    circle (2pt)
[black]  (5,1)    circle (2pt)
[black]  (6,1)    circle (2pt)
[black]  (7,1)    circle (2pt);
\draw (0,0.6) node {$p_1$};
\draw (0.7,2.3) node {$p_2$};
\draw (2,0.6) node {$p_3$};
\draw (3,2.4) node {$p_4$};
\draw (4,0.6) node {$p_{5}$};
\draw (5,0.6) node {$p_6$};
\draw (6,0.6) node {$p_7$};
\draw (7,0.6) node {$p_8$};
\draw (4,2.3) node {$c$};
\draw (5.4,2.2) node {$d$};
\draw (8.5,2.6) node {};
\draw(3.5,-0.3) node {$\widehat{P}$ in blue, $L(\widehat{P})>L(P)$};
\end{tikzpicture}

If $p_3\in \{d_1,d_2\}$ but $p_3\notin \{c_1,c_2\}$, then the same argument interchanging $c$ and $d$
yields a contradiction, concluding the case
$\{i,j\}=\{2,4\}$.

$\bullet$ If $\{i,j\}=\{2,5\}$, then $Q''\subset \{p_3,p_4,p_6,p_7\}$. If $\{c_1,c_2\}$ (respectively $\{d_1,d_2\}$) is equal to
$\{p_3,p_6\}$, then $c$ (respectively $d$) cannot be connected by an edge with $p_2$ nor with $p_5$, since that would extend $P$.
Since $P'\cup Q'=\{p_2,p_5,c,d\}$ is connected, there is an edge from $c$ to $d$. But then none of $c_1,d_1,d_1,d_2$ can be equal to
$p_7$, since then we could replace $p_8$ by $cd$ in $P$ and obtain a longer path. So
$$
\{\{c_1,c_2\},\{d_1,d_2\}\}=\{\{p_3,p_6\},\{p_4,p_6\}\}.
$$
The path $\widehat{P}$ in the following diagram, which is longer than $P$, shows that this is impossible.

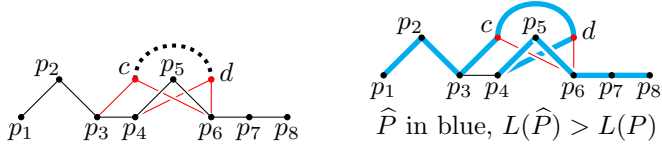
\begin{tikzpicture}[scale=0.5]
\draw(6,2.3) node {};
\draw[-] (0,1)--(1,2);
\draw[-] (1,2)--(2,1);
\draw[-] (2,1)--(3,1);
\draw[-] (5,1)--(7,1);

\draw[-,red] (2,1)--(3,2);
\draw[-,red] (3,1)--(5,2);
\draw[-,red] (5,2)--(5,1);
\draw[-,white,line width=2pt] (3,2)--(5,1);
\draw[-,red] (3,2)--(5,1);
\draw[-,line width=1.5pt,dotted] (3,2)..controls (3,3.2)and(5,3.2)..(5,2);
\draw[-,white,line width=2pt] (3,1)--(4,2);
\draw[-] (3,1)--(4,2);
\draw[-,white,line width=2pt] (4,2)--(5,1);
\draw[-] (4,2)--(5,1);
\filldraw[red]  (3,2)    circle (2pt);
\filldraw[red]  (5,2)    circle (2pt);
\filldraw [black]  (0,1)    circle (2pt)
[black]  (1,2)    circle (2pt)
[black]  (2,1)    circle (2pt)
[black]  (3,1)    circle (2pt)
[black]  (4,2)    circle (2pt)
[black]  (5,1)    circle (2pt)
[black]  (6,1)    circle (2pt)
[black]  (7,1)    circle (2pt);
\draw (0,0.6) node {$p_1$};
\draw (0.7,2.3) node {$p_2$};
\draw (2,0.6) node {$p_3$};
\draw (3,0.6) node {$p_4$};
\draw (4,2.3) node {$p_{5}$};
\draw (5,0.6) node {$p_6$};
\draw (6,0.6) node {$p_7$};
\draw (7,0.6) node {$p_8$};
\draw (2.7,2.3) node {$c$};
\draw (5.4,2.2) node {$d$};
\draw (8.5,2.6) node {};
\end{tikzpicture}
\begin{tikzpicture}[scale=0.5]
\draw(6,2.3) node {};
\draw[-,line width=2pt,cyan] (0,1)--(1,2);
\draw[-,line width=2pt,cyan] (1,2)--(2,1);
\draw[-] (2,1)--(3,1);
\draw[-,line width=2pt,cyan] (5,1)--(7,1);

\draw[-,line width=2pt,cyan] (2,1)--(3,2);
\draw[-,line width=2pt,cyan] (3,1)--(5,2);
\draw[-,red] (5,2)--(5,1);
\draw[-,white,line width=2pt] (3,2)--(5,1);
\draw[-,red] (3,2)--(5,1);
\draw[-,line width=2pt,cyan] (3,2)..controls (3,3.2)and(5,3.2)..(5,2);
\draw[-,white,line width=4pt] (3,1)--(4,2);
\draw[-,line width=2pt,cyan] (3,1)--(4,2);
\draw[-,white,line width=4pt] (4,2)--(5,1);
\draw[-,line width=2pt,cyan] (4,2)--(5,1);
\filldraw[red]  (3,2)    circle (2pt);
\filldraw[red]  (5,2)    circle (2pt);
\filldraw [black]  (0,1)    circle (2pt)
[black]  (1,2)    circle (2pt)
[black]  (2,1)    circle (2pt)
[black]  (3,1)    circle (2pt)
[black]  (4,2)    circle (2pt)
[black]  (5,1)    circle (2pt)
[black]  (6,1)    circle (2pt)
[black]  (7,1)    circle (2pt);
\draw (0,0.6) node {$p_1$};
\draw (0.7,2.3) node {$p_2$};
\draw (2,0.6) node {$p_3$};
\draw (3,0.6) node {$p_4$};
\draw (4,2.3) node {$p_{5}$};
\draw (5,0.6) node {$p_6$};
\draw (6,0.6) node {$p_7$};
\draw (7,0.6) node {$p_8$};
\draw (2.7,2.3) node {$c$};
\draw (5.4,2.2) node {$d$};
\draw (8.5,2.6) node {};
\draw(3.5,-0.3) node {$\widehat{P}$ in blue, $L(\widehat{P})>L(P)$};
\end{tikzpicture}

Thus the only possibilities for $\{c_1,c_2\}$ and $\{d_1,d_2\}$ are $\{p_3,p_7\}$, $\{p_4,p_6\}$ and $\{p_4,p_7\}$ and we have the following three
scenarios.

\begin{itemize}
\item[$\bullet\bullet$] $
\{\{c_1,c_2\},\{d_1,d_2\}\}=\{\{p_3,p_7\},\{p_4,p_6\}\},
{\vcenter{\hbox{
\begin{tikzpicture}[scale=0.5]
\draw(6,2.3) node {};
\draw[-] (0,1)--(1,2);
\draw[-] (1,2)--(2,1);
\draw[-] (2,1)--(3,1);
\draw[-] (5,1)--(7,1);

\draw[-,red] (2,1)--(3,2);
\draw[-,red] (3,1)--(5,2);
\draw[-,red] (5,2)--(5,1);
\draw[-,white,line width=2pt] (3,2)--(6,1);
\draw[-,red] (3,2)--(6,1);
\draw[-,white,line width=2pt] (3,1)--(4,2);
\draw[-] (3,1)--(4,2);
\draw[-,white,line width=2pt] (4,2)--(5,1);
\draw[-] (4,2)--(5,1);
\filldraw[red]  (3,2)    circle (2pt);
\filldraw[red]  (5,2)    circle (2pt);
\filldraw [black]  (0,1)    circle (2pt)
[black]  (1,2)    circle (2pt)
[black]  (2,1)    circle (2pt)
[black]  (3,1)    circle (2pt)
[black]  (4,2)    circle (2pt)
[black]  (5,1)    circle (2pt)
[black]  (6,1)    circle (2pt)
[black]  (7,1)    circle (2pt);
\draw (0,0.6) node {$p_1$};
\draw (0.7,2.3) node {$p_2$};
\draw (2,0.6) node {$p_3$};
\draw (3,0.6) node {$p_4$};
\draw (4,2.3) node {$p_{5}$};
\draw (5,0.6) node {$p_6$};
\draw (6,0.6) node {$p_7$};
\draw (7,0.6) node {$p_8$};
\draw (2.7,2.3) node {$c$};
\draw (5.4,2.2) node {$d$};
\draw (8.5,2.6) node {};
\end{tikzpicture}}}}
$
\item[$\bullet\bullet$] $
\{\{c_1,c_2\},\{d_1,d_2\}\}=\{\{p_3,p_7\},\{p_4,p_7\}\},
{\vcenter{\hbox{
\begin{tikzpicture}[scale=0.5]
\draw(6,2.3) node {};
\draw[-] (0,1)--(1,2);
\draw[-] (1,2)--(2,1);
\draw[-] (2,1)--(3,1);
\draw[-] (5,1)--(7,1);

\draw[-,red] (2,1)--(3,2);
\draw[-,red] (3,1)--(5,2);
\draw[-,red] (5,2)--(6,1);
\draw[-,white,line width=2pt] (3,2)--(6,1);
\draw[-,red] (3,2)--(6,1);
\draw[-,white,line width=2pt] (3,1)--(4,2);
\draw[-] (3,1)--(4,2);
\draw[-,white,line width=2pt] (4,2)--(5,1);
\draw[-] (4,2)--(5,1);
\filldraw[red]  (3,2)    circle (2pt);
\filldraw[red]  (5,2)    circle (2pt);
\filldraw [black]  (0,1)    circle (2pt)
[black]  (1,2)    circle (2pt)
[black]  (2,1)    circle (2pt)
[black]  (3,1)    circle (2pt)
[black]  (4,2)    circle (2pt)
[black]  (5,1)    circle (2pt)
[black]  (6,1)    circle (2pt)
[black]  (7,1)    circle (2pt);
\draw (0,0.6) node {$p_1$};
\draw (0.7,2.3) node {$p_2$};
\draw (2,0.6) node {$p_3$};
\draw (3,0.6) node {$p_4$};
\draw (4,2.3) node {$p_{5}$};
\draw (5,0.6) node {$p_6$};
\draw (6,0.6) node {$p_7$};
\draw (7,0.6) node {$p_8$};
\draw (2.7,2.3) node {$c$};
\draw (5.4,2.2) node {$d$};
\draw (8.5,2.6) node {};
\end{tikzpicture}}}}
$
\item[$\bullet\bullet$] $
\{\{c_1,c_2\},\{d_1,d_2\}\}=\{\{p_4,p_6\},\{p_4,p_7\}\},
{\vcenter{\hbox{
\begin{tikzpicture}[scale=0.5]
\draw(6,2.3) node {};
\draw[-] (0,1)--(1,2);
\draw[-] (1,2)--(2,1);
\draw[-] (2,1)--(3,1);
\draw[-] (5,1)--(7,1);

\draw[-,red] (3,1)--(3,2);
\draw[-,red] (3,1)--(5,2);
\draw[-,red] (5,2)--(5,1);
\draw[-,white,line width=2pt] (3,2)--(6,1);
\draw[-,red] (3,2)--(6,1);
\draw[-,white,line width=2pt] (3,1)--(4,2);
\draw[-] (3,1)--(4,2);
\draw[-,white,line width=2pt] (4,2)--(5,1);
\draw[-] (4,2)--(5,1);
\filldraw[red]  (3,2)    circle (2pt);
\filldraw[red]  (5,2)    circle (2pt);
\filldraw [black]  (0,1)    circle (2pt)
[black]  (1,2)    circle (2pt)
[black]  (2,1)    circle (2pt)
[black]  (3,1)    circle (2pt)
[black]  (4,2)    circle (2pt)
[black]  (5,1)    circle (2pt)
[black]  (6,1)    circle (2pt)
[black]  (7,1)    circle (2pt);
\draw (0,0.6) node {$p_1$};
\draw (0.7,2.3) node {$p_2$};
\draw (2,0.6) node {$p_3$};
\draw (3,0.6) node {$p_4$};
\draw (4,2.3) node {$p_{5}$};
\draw (5,0.6) node {$p_6$};
\draw (6,0.6) node {$p_7$};
\draw (7,0.6) node {$p_8$};
\draw (2.7,2.3) node {$c$};
\draw (5.4,2.2) node {$d$};
\draw (8.5,2.6) node {};
\end{tikzpicture}}}}
$
\end{itemize}

The path $\widehat{P}$ in each of the following diagrams shows that none of the first two cases is possible.

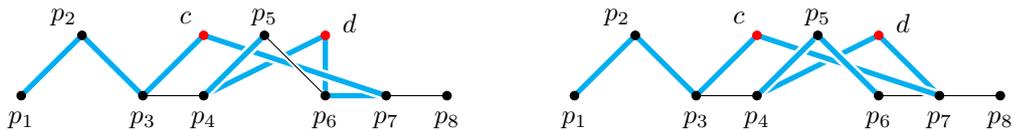
\begin{tikzpicture}[scale=0.8]
\draw(6,2.3) node {};
\draw[-,line width=2pt,cyan] (0,1)--(1,2);
\draw[-,line width=2pt,cyan] (1,2)--(2,1);
\draw[-] (2,1)--(3,1);
\draw[-,line width=2pt,cyan] (5,1)--(6,1);
\draw[-] (6,1)--(7,1);
\draw[-,line width=2pt,cyan] (2,1)--(3,2);
\draw[-,line width=2pt,cyan] (3,1)--(5,2);
\draw[-,line width=2pt,cyan] (5,2)--(5,1);
\draw[-,white,line width=4pt] (3,2)--(6,1);
\draw[-,line width=2pt,cyan] (3,2)--(6,1);
\draw[-,white,line width=4pt] (3,1)--(4,2);
\draw[-,line width=2pt,cyan] (3,1)--(4,2);
\draw[-,white,line width=2pt] (4,2)--(5,1);
\draw[-] (4,2)--(5,1);
\filldraw[red]  (3,2)    circle (2pt);
\filldraw[red]  (5,2)    circle (2pt);
\filldraw [black]  (0,1)    circle (2pt)
[black]  (1,2)    circle (2pt)
[black]  (2,1)    circle (2pt)
[black]  (3,1)    circle (2pt)
[black]  (4,2)    circle (2pt)
[black]  (5,1)    circle (2pt)
[black]  (6,1)    circle (2pt)
[black]  (7,1)    circle (2pt);
\draw (0,0.6) node {$p_1$};
\draw (0.7,2.3) node {$p_2$};
\draw (2,0.6) node {$p_3$};
\draw (3,0.6) node {$p_4$};
\draw (4,2.3) node {$p_{5}$};
\draw (5,0.6) node {$p_6$};
\draw (6,0.6) node {$p_7$};
\draw (7,0.6) node {$p_8$};
\draw (2.7,2.3) node {$c$};
\draw (5.4,2.2) node {$d$};
\draw (8.5,2.6) node {};
\end{tikzpicture}
\begin{tikzpicture}[scale=0.8]
\draw(6,2.3) node {};
\draw[-,line width=2pt,cyan] (0,1)--(1,2);
\draw[-,line width=2pt,cyan] (1,2)--(2,1);
\draw[-] (2,1)--(3,1);
\draw[-] (5,1)--(6,1);
\draw[-] (6,1)--(7,1);
\draw[-,line width=2pt,cyan] (2,1)--(3,2);
\draw[-,line width=2pt,cyan] (3,1)--(5,2);
\draw[-,line width=2pt,cyan] (5,2)--(6,1);
\draw[-,white,line width=4pt] (3,2)--(6,1);
\draw[-,line width=2pt,cyan] (3,2)--(6,1);
\draw[-,white,line width=4pt] (3,1)--(4,2);
\draw[-,line width=2pt,cyan] (3,1)--(4,2);
\draw[-,white,line width=4pt] (4,2)--(5,1);
\draw[-,line width=2pt,cyan] (4,2)--(5,1);
\filldraw[red]  (3,2)    circle (2pt);
\filldraw[red]  (5,2)    circle (2pt);
\filldraw [black]  (0,1)    circle (2pt)
[black]  (1,2)    circle (2pt)
[black]  (2,1)    circle (2pt)
[black]  (3,1)    circle (2pt)
[black]  (4,2)    circle (2pt)
[black]  (5,1)    circle (2pt)
[black]  (6,1)    circle (2pt)
[black]  (7,1)    circle (2pt);
\draw (0,0.6) node {$p_1$};
\draw (0.7,2.3) node {$p_2$};
\draw (2,0.6) node {$p_3$};
\draw (3,0.6) node {$p_4$};
\draw (4,2.3) node {$p_{5}$};
\draw (5,0.6) node {$p_6$};
\draw (6,0.6) node {$p_7$};
\draw (7,0.6) node {$p_8$};
\draw (2.7,2.3) node {$c$};
\draw (5.4,2.2) node {$d$};
\draw (8.5,2.6) node {};
\end{tikzpicture}

In the third case $p_5$ cannot connect with $c$ nor $d$, thus it has to be connected with $p_2$, since $P'\cup Q'$ is connected. But then
the path $\widehat{P}$ in second diagram shows that the third case is not possible, thus discarding the case $\{i,j\}=\{2,5\}$.

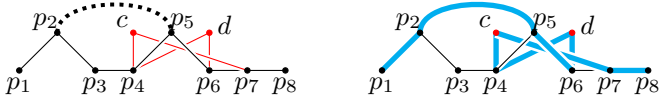
\begin{tikzpicture}[scale=0.5]
\draw(6,2.3) node {};
\draw[-] (0,1)--(1,2);
\draw[-] (1,2)--(2,1);
\draw[-] (2,1)--(3,1);
\draw[-] (5,1)--(7,1);

\draw[-,red] (3,1)--(3,2);
\draw[-,red] (3,1)--(5,2);
\draw[-,red] (5,2)--(5,1);
\draw[-,white,line width=2pt] (3,2)--(6,1);
\draw[-,red] (3,2)--(6,1);
\draw[-,white,line width=2pt] (3,1)--(4,2);
\draw[-] (3,1)--(4,2);
\draw[-,white,line width=2pt] (4,2)--(5,1);
\draw[-] (4,2)--(5,1);
\draw[-,line width=1.5pt,dotted] (1,2)..controls(1,3)and(4,3)..(4,2);
\filldraw[red]  (3,2)    circle (2pt);
\filldraw[red]  (5,2)    circle (2pt);
\filldraw [black]  (0,1)    circle (2pt)
[black]  (1,2)    circle (2pt)
[black]  (2,1)    circle (2pt)
[black]  (3,1)    circle (2pt)
[black]  (4,2)    circle (2pt)
[black]  (5,1)    circle (2pt)
[black]  (6,1)    circle (2pt)
[black]  (7,1)    circle (2pt);
\draw (0,0.6) node {$p_1$};
\draw (0.7,2.3) node {$p_2$};
\draw (2,0.6) node {$p_3$};
\draw (3,0.6) node {$p_4$};
\draw (4.3,2.3) node {$p_{5}$};
\draw (5,0.6) node {$p_6$};
\draw (6,0.6) node {$p_7$};
\draw (7,0.6) node {$p_8$};
\draw (2.7,2.3) node {$c$};
\draw (5.4,2.2) node {$d$};
\draw (8.5,2.6) node {};
\end{tikzpicture}
\begin{tikzpicture}[scale=0.5]
\draw(6,2.3) node {};
\draw[-,line width=2pt,cyan] (0,1)--(1,2);
\draw[-] (1,2)--(2,1);
\draw[-] (2,1)--(3,1);
\draw[-] (5,1)--(6,1);
\draw[-,line width=2pt,cyan] (6,1)--(7,1);
\draw[-,line width=2pt,cyan] (3,1)--(3,2);
\draw[-,line width=2pt,cyan] (3,1)--(5,2);
\draw[-,line width=2pt,cyan] (5,2)--(5,1);
\draw[-,white,line width=4pt] (3,2)--(6,1);
\draw[-,line width=2pt,cyan] (3,2)--(6,1);
\draw[-,white,line width=2pt] (3,1)--(4,2);
\draw[-] (3,1)--(4,2);
\draw[-,white,line width=4pt] (4,2)--(5,1);
\draw[-,line width=2pt,cyan] (4,2)--(5,1);
\draw[-,line width=2pt,cyan] (1,2)..controls(1,3)and(4,3)..(4,2);
\filldraw[red]  (3,2)    circle (2pt);
\filldraw[red]  (5,2)    circle (2pt);
\filldraw [black]  (0,1)    circle (2pt)
[black]  (1,2)    circle (2pt)
[black]  (2,1)    circle (2pt)
[black]  (3,1)    circle (2pt)
[black]  (4,2)    circle (2pt)
[black]  (5,1)    circle (2pt)
[black]  (6,1)    circle (2pt)
[black]  (7,1)    circle (2pt);
\draw (0,0.6) node {$p_1$};
\draw (0.7,2.3) node {$p_2$};
\draw (2,0.6) node {$p_3$};
\draw (3,0.6) node {$p_4$};
\draw (4.3,2.3) node {$p_{5}$};
\draw (5,0.6) node {$p_6$};
\draw (6,0.6) node {$p_7$};
\draw (7,0.6) node {$p_8$};
\draw (2.7,2.3) node {$c$};
\draw (5.4,2.2) node {$d$};
\draw (8.5,2.6) node {};
\end{tikzpicture}

$\bullet$ If $\{i,j\}=\{2,6\}$, then $Q''\subset \{p_3,p_4,p_5,p_7\}$.
If one of $c_1,c_2$ is equal to $p_3$, then $c$ cannot be connected by an edge with $p_2$ nor with $p_5$, since then one could extend $P$.
So it has to be connected with $d$, but then one of $c$ or $d$ has to be connected with $p_7$, and we can replace $p_8$ by the
edge $cd$ in $P$, obtaining a longer path. Thus $p_3\notin \{c_1,c_2\}$ and by the same argument, interchanging $c$ with $d$, we
obtain $p_3\notin \{d_1,d_2\}$.

Hence $Q''=\{c_1,c_2,d_1,d_2\}\subset\{p_4,p_5,p_7\}$ and there is only one possibility left:
$$
\{ \{c_1,c_2\},\{d_1,d_2\}\}=\{ \{p_4,p_7\},\{p_5,p_7\}\}.
$$
The blue path $\widehat{P}$ in the second diagram, which is longer than $P$,
shows that this is not possible, and so we have discarded the case $\{i,j\}=\{2,6\}$.

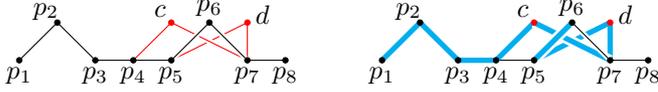
\begin{tikzpicture}[scale=0.5]
\draw(6,2.3) node {};
\draw[-] (0,1)--(1,2);
\draw[-] (1,2)--(2,1);
\draw[-] (2,1)--(4,1);
\draw[-] (6,1)--(7,1);

\draw[-,red] (3,1)--(4,2);
\draw[-,red] (4,1)--(6,2);
\draw[-,red] (6,2)--(6,1);
\draw[-,white,line width=2pt] (6,1)--(4,2);
\draw[-,red] (6,1)--(4,2);
\draw[-,white,line width=2pt] (4,1)--(5,2);
\draw[-] (4,1)--(5,2);
\draw[-,white,line width=2pt] (5,2)--(6,1);
\draw[-] (5,2)--(6,1);
\filldraw[red]  (4,2)    circle (2pt);
\filldraw[red]  (6,2)    circle (2pt);
\filldraw [black]  (0,1)    circle (2pt)
[black]  (1,2)    circle (2pt)
[black]  (2,1)    circle (2pt)
[black]  (3,1)    circle (2pt)
[black]  (4,1)    circle (2pt)
[black]  (5,2)    circle (2pt)
[black]  (6,1)    circle (2pt)
[black]  (7,1)    circle (2pt);
\draw (0,0.6) node {$p_1$};
\draw (0.7,2.3) node {$p_2$};
\draw (2,0.6) node {$p_3$};
\draw (3,0.6) node {$p_4$};
\draw (4,0.6) node {$p_{5}$};
\draw (5,2.4) node {$p_6$};
\draw (6,0.6) node {$p_7$};
\draw (7,0.6) node {$p_8$};
\draw (3.7,2.3) node {$c$};
\draw (6.4,2.2) node {$d$};
\draw (8.5,2.6) node {};
\end{tikzpicture}
\begin{tikzpicture}[scale=0.5]
\draw(6,2.3) node {};
\draw[-,line width=2pt,cyan] (0,1)--(1,2);
\draw[-,line width=2pt,cyan] (1,2)--(2,1);
\draw[-,line width=2pt,cyan] (2,1)--(3,1);
\draw[-] (3,1)--(4,1);
\draw[-] (6,1)--(7,1);

\draw[-,line width=2pt,cyan] (3,1)--(4,2);
\draw[-,line width=2pt,cyan] (4,1)--(6,2);
\draw[-,line width=2pt,cyan] (6,2)--(6,1);
\draw[-,white,line width=4pt] (6,1)--(4,2);
\draw[-,line width=2pt,cyan] (6,1)--(4,2);
\draw[-,white,line width=4pt] (4,1)--(5,2);
\draw[-,line width=2pt,cyan] (4,1)--(5,2);
\draw[-,white,line width=2pt] (5,2)--(6,1);
\draw[-] (5,2)--(6,1);
\filldraw[red]  (4,2)    circle (2pt);
\filldraw[red]  (6,2)    circle (2pt);
\filldraw [black]  (0,1)    circle (2pt)
[black]  (1,2)    circle (2pt)
[black]  (2,1)    circle (2pt)
[black]  (3,1)    circle (2pt)
[black]  (4,1)    circle (2pt)
[black]  (5,2)    circle (2pt)
[black]  (6,1)    circle (2pt)
[black]  (7,1)    circle (2pt);
\draw (0,0.6) node {$p_1$};
\draw (0.7,2.3) node {$p_2$};
\draw (2,0.6) node {$p_3$};
\draw (3,0.6) node {$p_4$};
\draw (4,0.6) node {$p_{5}$};
\draw (5,2.4) node {$p_6$};
\draw (6,0.6) node {$p_7$};
\draw (7,0.6) node {$p_8$};
\draw (3.7,2.3) node {$c$};
\draw (6.4,2.2) node {$d$};
\draw (8.5,2.6) node {};
\end{tikzpicture}

$\bullet$ If $\{i,j\}=\{2,7\}$, then $Q''\subset \{p_3,p_4,p_5,p_6\}$.
Thus the only possibilities for $\{c_1,c_2\}$ and $\{d_1,d_2\}$ are $\{p_3,p_5\}$, $\{p_3,p_6\}$ and $\{p_4,p_6\}$ and we have the following three
scenarios.

\begin{itemize}
\item[$\bullet\bullet$] $
\{\{c_1,c_2\},\{d_1,d_2\}\}=\{\{p_3,p_5\},\{p_3,p_6\}\},
{\vcenter{\hbox{
\begin{tikzpicture}[scale=0.5]
\draw(6,2.3) node {};
\draw[-] (0,1)--(1,2);
\draw[-] (1,2)--(2,1);
\draw[-] (2,1)--(5,1);
\draw[-,red] (2,1)--(3,2);
\draw[-,red] (2,1)--(4,2);
\draw[-,line width=2pt,white] (3,2)--(4,1);
\draw[-,red] (3,2)--(4,1);
\draw[-,white,line width=2pt] (5,1)--(4,2);
\draw[-,red] (5,1)--(4,2);
\draw[-] (5,1)--(6,2);
\draw[-] (6,2)--(7,1);
\filldraw[red]  (3,2)    circle (2pt);
\filldraw[red]  (4,2)    circle (2pt);
\filldraw [black]  (0,1)    circle (2pt)
[black]  (1,2)    circle (2pt)
[black]  (2,1)    circle (2pt)
[black]  (3,1)    circle (2pt)
[black]  (4,1)    circle (2pt)
[black]  (5,1)    circle (2pt)
[black]  (6,2)    circle (2pt)
[black]  (7,1)    circle (2pt);
\draw (0,0.6) node {$p_1$};
\draw (0.7,2.3) node {$p_2$};
\draw (2,0.6) node {$p_3$};
\draw (3,0.6) node {$p_4$};
\draw (4,0.6) node {$p_{5}$};
\draw (5,0.6) node {$p_6$};
\draw (6,2.4) node {$p_7$};
\draw (7,0.6) node {$p_8$};
\draw (2.7,2.3) node {$c$};
\draw (4.4,2.2) node {$d$};
\draw (8.5,2.6) node {};
\end{tikzpicture}}}}
$
\item[$\bullet\bullet$] $
\{\{c_1,c_2\},\{d_1,d_2\}\}=\{\{p_3,p_5\},\{p_4,p_6\}\},
{\vcenter{\hbox{
\begin{tikzpicture}[scale=0.5]
\draw(6,2.3) node {};
\draw[-] (0,1)--(1,2);
\draw[-] (1,2)--(2,1);
\draw[-] (2,1)--(5,1);

\draw[-,red] (2,1)--(3,2);
\draw[-,red] (3,1)--(4,2);
\draw[-,line width=2pt,white] (3,2)--(4,1);
\draw[-,red] (3,2)--(4,1);
\draw[-,white,line width=2pt] (5,1)--(4,2);
\draw[-,red] (5,1)--(4,2);
\draw[-] (5,1)--(6,2);
\draw[-] (6,2)--(7,1);
\filldraw[red]  (3,2)    circle (2pt);
\filldraw[red]  (4,2)    circle (2pt);
\filldraw [black]  (0,1)    circle (2pt)
[black]  (1,2)    circle (2pt)
[black]  (2,1)    circle (2pt)
[black]  (3,1)    circle (2pt)
[black]  (4,1)    circle (2pt)
[black]  (5,1)    circle (2pt)
[black]  (6,2)    circle (2pt)
[black]  (7,1)    circle (2pt);
\draw (0,0.6) node {$p_1$};
\draw (0.7,2.3) node {$p_2$};
\draw (2,0.6) node {$p_3$};
\draw (3,0.6) node {$p_4$};
\draw (4,0.6) node {$p_{5}$};
\draw (5,0.6) node {$p_6$};
\draw (6,2.4) node {$p_7$};
\draw (7,0.6) node {$p_8$};
\draw (2.7,2.3) node {$c$};
\draw (4.4,2.2) node {$d$};
\draw (8.5,2.6) node {};
\end{tikzpicture}}}}
$
\item[$\bullet\bullet$] $
\{\{c_1,c_2\},\{d_1,d_2\}\}=\{\{p_3,p_6\},\{p_4,p_6\}\}.
{\vcenter{\hbox{
\begin{tikzpicture}[scale=0.5]
\draw(6,2.3) node {};
\draw[-] (0,1)--(1,2);
\draw[-] (1,2)--(2,1);
\draw[-] (2,1)--(5,1);
\draw[-,red] (2,1)--(3,2);
\draw[-,red] (3,1)--(4,2);
\draw[-,line width=2pt,white] (3,2)--(5,1);
\draw[-,red] (3,2)--(5,1);
\draw[-,white,line width=2pt] (5,1)--(4,2);
\draw[-,red] (5,1)--(4,2);
\draw[-] (5,1)--(6,2);
\draw[-] (6,2)--(7,1);
\filldraw[red]  (3,2)    circle (2pt);
\filldraw[red]  (4,2)    circle (2pt);
\filldraw [black]  (0,1)    circle (2pt)
[black]  (1,2)    circle (2pt)
[black]  (2,1)    circle (2pt)
[black]  (3,1)    circle (2pt)
[black]  (4,1)    circle (2pt)
[black]  (5,1)    circle (2pt)
[black]  (6,2)    circle (2pt)
[black]  (7,1)    circle (2pt);
\draw (0,0.6) node {$p_1$};
\draw (0.7,2.3) node {$p_2$};
\draw (2,0.6) node {$p_3$};
\draw (3,0.6) node {$p_4$};
\draw (4,0.6) node {$p_{5}$};
\draw (5,0.6) node {$p_6$};
\draw (6,2.4) node {$p_7$};
\draw (7,0.6) node {$p_8$};
\draw (2.7,2.3) node {$c$};
\draw (4.4,2.2) node {$d$};
\draw (8.5,2.6) node {};
\end{tikzpicture}}}}
$
\end{itemize}

In the first case $d$ must be connected with $c$ by an edge. In fact, it cannot be connected with $p_7$, since otherwise we could replace
$p_6p_7$ with $p_6dp_7$ in $P$ and obtain a longer path than $P$, and similarly, if it is connected with $p_2$, then we replace $p_2p_3$
by $p_2dp_3$ in $P$, obtaining a longer path. Since $P'\cup Q'=\{p_2,p_7,c,d\}$ is connected, it has to be connected with $c$. But then
the blue path $\widehat{P}$ in the second diagram is longer than $P$, which discards the case
$\{\{c_1,c_2\},\{d_1,d_2\}\}=\{\{p_3,p_5\},\{p_3,p_6\}\}$.

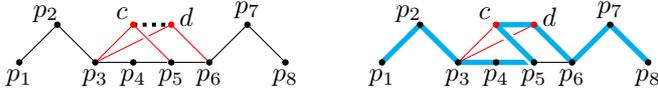
\begin{tikzpicture}[scale=0.5]
\draw(6,2.3) node {};
\draw[-] (0,1)--(1,2);
\draw[-] (1,2)--(2,1);
\draw[-] (2,1)--(5,1);
\draw[-,red] (2,1)--(3,2);
\draw[-,red] (2,1)--(4,2);
\draw[-,line width=2pt,white] (3,2)--(4,1);
\draw[-,red] (3,2)--(4,1);
\draw[-,line width=1.5pt,dotted] (3,2)--(4,2);
\draw[-,white,line width=2pt] (5,1)--(4,2);
\draw[-,red] (5,1)--(4,2);
\draw[-] (5,1)--(6,2);
\draw[-] (6,2)--(7,1);
\filldraw[red]  (3,2)    circle (2pt);
\filldraw[red]  (4,2)    circle (2pt);
\filldraw [black]  (0,1)    circle (2pt)
[black]  (1,2)    circle (2pt)
[black]  (2,1)    circle (2pt)
[black]  (3,1)    circle (2pt)
[black]  (4,1)    circle (2pt)
[black]  (5,1)    circle (2pt)
[black]  (6,2)    circle (2pt)
[black]  (7,1)    circle (2pt);
\draw (0,0.6) node {$p_1$};
\draw (0.7,2.3) node {$p_2$};
\draw (2,0.6) node {$p_3$};
\draw (3,0.6) node {$p_4$};
\draw (4,0.6) node {$p_{5}$};
\draw (5,0.6) node {$p_6$};
\draw (6,2.4) node {$p_7$};
\draw (7,0.6) node {$p_8$};
\draw (2.7,2.3) node {$c$};
\draw (4.4,2.2) node {$d$};
\draw (8.5,2.6) node {};
\end{tikzpicture}
\begin{tikzpicture}[scale=0.5]
\draw(6,2.3) node {};
\draw[-,line width=2pt,cyan] (0,1)--(1,2);
\draw[-,line width=2pt,cyan] (1,2)--(2,1);
\draw[-,line width=2pt,cyan] (2,1)--(4,1);
\draw[-] (4,1)--(5,1);
\draw[-,red] (2,1)--(3,2);
\draw[-,red] (2,1)--(4,2);
\draw[-,line width=4pt,white] (3,2)--(4,1);
\draw[-,line width=2pt,cyan] (3,2)--(4,1);
\draw[-,line width=2pt,cyan] (3,2)--(4,2);
\draw[-,cyan,line width=2pt] (5,1)--(4,2);
\draw[-,line width=2pt,cyan] (5,1)--(6,2);
\draw[-,line width=2pt,cyan] (6,2)--(7,1);
\filldraw[red]  (3,2)    circle (2pt);
\filldraw[red]  (4,2)    circle (2pt);
\filldraw [black]  (0,1)    circle (2pt)
[black]  (1,2)    circle (2pt)
[black]  (2,1)    circle (2pt)
[black]  (3,1)    circle (2pt)
[black]  (4,1)    circle (2pt)
[black]  (5,1)    circle (2pt)
[black]  (6,2)    circle (2pt)
[black]  (7,1)    circle (2pt);
\draw (0,0.6) node {$p_1$};
\draw (0.7,2.3) node {$p_2$};
\draw (2,0.6) node {$p_3$};
\draw (3,0.6) node {$p_4$};
\draw (4,0.6) node {$p_{5}$};
\draw (5,0.6) node {$p_6$};
\draw (6,2.4) node {$p_7$};
\draw (7,0.6) node {$p_8$};
\draw (2.7,2.3) node {$c$};
\draw (4.4,2.2) node {$d$};
\draw (8.5,2.6) node {};
\end{tikzpicture}

The second case is discarded by the path $\widehat{P}$ in the second diagram, since it is longer than $P$.

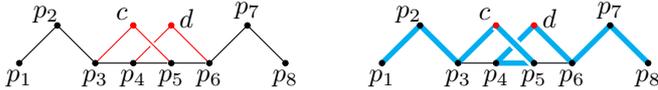
\begin{tikzpicture}[scale=0.5]
\draw(6,2.3) node {};
\draw[-] (0,1)--(1,2);
\draw[-] (1,2)--(2,1);
\draw[-] (2,1)--(5,1);
\draw[-,red] (2,1)--(3,2);
\draw[-,red] (3,1)--(4,2);
\draw[-,line width=2pt,white] (3,2)--(4,1);
\draw[-,red] (3,2)--(4,1);
\draw[-,white,line width=2pt] (5,1)--(4,2);
\draw[-,red] (5,1)--(4,2);
\draw[-] (5,1)--(6,2);
\draw[-] (6,2)--(7,1);
\filldraw[red]  (3,2)    circle (2pt);
\filldraw[red]  (4,2)    circle (2pt);
\filldraw [black]  (0,1)    circle (2pt)
[black]  (1,2)    circle (2pt)
[black]  (2,1)    circle (2pt)
[black]  (3,1)    circle (2pt)
[black]  (4,1)    circle (2pt)
[black]  (5,1)    circle (2pt)
[black]  (6,2)    circle (2pt)
[black]  (7,1)    circle (2pt);
\draw (0,0.6) node {$p_1$};
\draw (0.7,2.3) node {$p_2$};
\draw (2,0.6) node {$p_3$};
\draw (3,0.6) node {$p_4$};
\draw (4,0.6) node {$p_{5}$};
\draw (5,0.6) node {$p_6$};
\draw (6,2.4) node {$p_7$};
\draw (7,0.6) node {$p_8$};
\draw (2.7,2.3) node {$c$};
\draw (4.4,2.2) node {$d$};
\draw (8.5,2.6) node {};
\end{tikzpicture}
\begin{tikzpicture}[scale=0.5]
\draw(6,2.3) node {};
\draw[-,line width=2pt,cyan] (0,1)--(1,2);
\draw[-,line width=2pt,cyan] (1,2)--(2,1);
\draw[-,line width=2pt,cyan] (3,1)--(4,1);
\draw[-] (2,1)--(3,1);
\draw[-] (4,1)--(5,1);
\draw[-,line width=2pt,cyan] (2,1)--(3,2);
\draw[-,line width=2pt,cyan] (3,1)--(4,2);
\draw[-,line width=4pt,white] (3,2)--(4,1);
\draw[-,line width=2pt,cyan] (3,2)--(4,1);
\draw[-,white,line width=2pt] (5,1)--(4,2);
\draw[-,line width=2pt,cyan] (5,1)--(4,2);
\draw[-,line width=2pt,cyan] (5,1)--(6,2);
\draw[-,line width=2pt,cyan] (6,2)--(7,1);
\filldraw[red]  (3,2)    circle (2pt);
\filldraw[red]  (4,2)    circle (2pt);
\filldraw [black]  (0,1)    circle (2pt)
[black]  (1,2)    circle (2pt)
[black]  (2,1)    circle (2pt)
[black]  (3,1)    circle (2pt)
[black]  (4,1)    circle (2pt)
[black]  (5,1)    circle (2pt)
[black]  (6,2)    circle (2pt)
[black]  (7,1)    circle (2pt);
\draw (0,0.6) node {$p_1$};
\draw (0.7,2.3) node {$p_2$};
\draw (2,0.6) node {$p_3$};
\draw (3,0.6) node {$p_4$};
\draw (4,0.6) node {$p_{5}$};
\draw (5,0.6) node {$p_6$};
\draw (6,2.4) node {$p_7$};
\draw (7,0.6) node {$p_8$};
\draw (2.7,2.3) node {$c$};
\draw (4.4,2.2) node {$d$};
\draw (8.5,2.6) node {};
\end{tikzpicture}

The third case is symmetric to the first case, and so we have discarded the case $\{i,j\}=\{2,7\}$.

$\bullet$ If $\{i,j\}=\{3,5\}$, then $Q''\subset \{p_2,p_4,p_6,p_7\}$.
If one of $c_1,c_2$ is equal to $p_4$, then $c$ cannot be connected by an edge with $p_3$ nor with $p_5$, since then one could extend $P$.
So it has to be connected with $d$, but then none of $c$ or $d$ can be connected with $p_7$ nor with $p_2$, since we could
replace $p_8$ (respectively $p_1$) by the edge $cd$ in $P$, obtaining a longer path. This implies that $Q''\subset \{p_4,p_6\}$ which is impossible,
since $\{c_1,c_2\}\ne \{d_1,d_2\}$.

 Thus $p_4\notin \{c_1,c_2\}$ and by the same argument, interchanging $c$ with $d$, we
obtain $p_4\notin \{d_1,d_2\}$.

Hence $Q''=\{c_1,c_2,d_1,d_2\}\subset\{p_2,p_6,p_7\}$ and there is only one possibility left:
$$
\{ \{c_1,c_2\},\{d_1,d_2\}\}=\{ \{p_2,p_6\},\{p_2,p_7\}\}.
$$
The blue path $\widehat{P}$ in the second diagram, which is longer than $P$,
shows that this is not possible, and so we have discarded the case $\{i,j\}=\{3,5\}$.

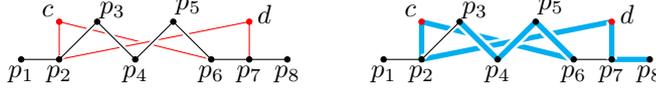
\begin{tikzpicture}[scale=0.5]
\draw(6,2.3) node {};
\draw[-] (0,1)--(1,1);
\draw[-] (5,1)--(7,1);
\draw[-,red] (1,1)--(1,2);
\draw[-,red] (5,1)--(1,2);
\draw[-,line width=2pt,white] (1,1)--(6,2);
\draw[-,red] (1,1)--(6,2);
\draw[-,white,line width=2pt] (6,1)--(6,2);
\draw[-,red] (6,1)--(6,2);
\draw[-,white,line width=2pt] (1,1)--(2,2);
\draw[-,white,line width=2pt] (2,2)--(3,1);
\draw[-,white,line width=2pt] (3,1)--(4,2);
\draw[-,white,line width=2pt] (4,2)--(5,1);
\draw[-] (1,1)--(2,2);
\draw[-] (2,2)--(3,1);
\draw[-] (3,1)--(4,2);
\draw[-] (4,2)--(5,1);
\filldraw[red]  (1,2)    circle (2pt);
\filldraw[red]  (6,2)    circle (2pt);
\filldraw [black]  (0,1)    circle (2pt)
[black]  (1,1)    circle (2pt)
[black]  (2,2)    circle (2pt)
[black]  (3,1)    circle (2pt)
[black]  (4,2)    circle (2pt)
[black]  (5,1)    circle (2pt)
[black]  (6,1)    circle (2pt)
[black]  (7,1)    circle (2pt);
\draw (0,0.6) node {$p_1$};
\draw (1,0.6) node {$p_2$};
\draw (2.4,2.3) node {$p_3$};
\draw (3,0.6) node {$p_4$};
\draw (4.4,2.4) node {$p_{5}$};
\draw (5,0.6) node {$p_6$};
\draw (6,0.6) node {$p_7$};
\draw (7,0.6) node {$p_8$};
\draw (0.7,2.3) node {$c$};
\draw (6.4,2.2) node {$d$};
\draw (8.5,2.6) node {};
\end{tikzpicture}
\begin{tikzpicture}[scale=0.5]
\draw(6,2.3) node {};
\draw[-] (0,1)--(1,1);
\draw[-] (5,1)--(6,1);
\draw[-,line width=2pt,cyan] (6,1)--(7,1);
\draw[-,line width=2pt,cyan] (1,1)--(1,2);
\draw[-,line width=2pt,cyan] (5,1)--(1,2);
\draw[-,line width=4pt,white] (1,1)--(6,2);
\draw[-,line width=2pt,cyan] (1,1)--(6,2);
\draw[-,white,line width=4pt] (6,1)--(6,2);
\draw[-,line width=2pt,cyan] (6,1)--(6,2);
\draw[-,white,line width=2pt] (1,1)--(2,2);
\draw[-,white,line width=4pt] (2,2)--(3,1);
\draw[-,white,line width=4pt] (3,1)--(4,2);
\draw[-,white,line width=4pt] (4,2)--(5,1);
\draw[-] (1,1)--(2,2);
\draw[-,line width=2pt,cyan] (2,2)--(3,1);
\draw[-,line width=2pt,cyan] (3,1)--(4,2);
\draw[-,line width=2pt,cyan] (4,2)--(5,1);
\filldraw[red]  (1,2)    circle (2pt);
\filldraw[red]  (6,2)    circle (2pt);
\filldraw [black]  (0,1)    circle (2pt)
[black]  (1,1)    circle (2pt)
[black]  (2,2)    circle (2pt)
[black]  (3,1)    circle (2pt)
[black]  (4,2)    circle (2pt)
[black]  (5,1)    circle (2pt)
[black]  (6,1)    circle (2pt)
[black]  (7,1)    circle (2pt);
\draw (0,0.6) node {$p_1$};
\draw (1,0.6) node {$p_2$};
\draw (2.4,2.3) node {$p_3$};
\draw (3,0.6) node {$p_4$};
\draw (4.4,2.4) node {$p_{5}$};
\draw (5,0.6) node {$p_6$};
\draw (6,0.6) node {$p_7$};
\draw (7,0.6) node {$p_8$};
\draw (0.7,2.3) node {$c$};
\draw (6.4,2.2) node {$d$};
\draw (8.5,2.6) node {};
\end{tikzpicture}

$\bullet$ If $\{i,j\}=\{3,6\}$, then $Q''\subset \{p_2,p_4,p_5,p_7\}$. We claim that $c$ and $d$ cannot be connected by an edge. In fact, if they are
connected, then none of the two can be connected with $p_2$, since we could replace $p_1$ by $cd$ in $P$ and obtain a longer path, and similarly
none of both can be connected with $p_7$. Hence $Q''\subset \{p_4,p_5\}$, which is impossible and proves the claim.

If $c$ is connected with $p_3$, then $\{c_1,c_2\}=\{p_5,p_7\}$, since a connection of $c$ with $p_2$ leads to a path longer than $P$,
replacing $p_2p_3$ by $p_2cp_3$ in $P$, and if $c$ is connected with $p_4$, then replacing $p_3p_4$ by $p_3cp_4$ in $P$ yields a path longer than $P$.

Similarly,
\begin{itemize}
  \item[-] If $d$ is connected with $p_3$, then $\{d_1,d_2\}=\{p_5,p_7\}$,
  \item[-] If $c$ is connected with $p_5$, then $\{c_1,c_2\}=\{p_2,p_4\}$,
  \item[-] If $d$ is connected with $p_5$, then $\{d_1,d_2\}=\{p_2,p_4\}$.
\end{itemize}

Since $\{c_1,c_2\}\ne \{d_1,d_2\}$, one of $c,d$ is connected with $p_3$ and the other with $p_5$, and we have
$$
\{\{c_1,c_2\}, \{d_1,d_2\}\}=\{\{p_2,p_4\}, \{p_5,p_7\}\}.
$$
The blue path $\widehat{P}$ in the second diagram, which is longer than $P$,
shows that this is not possible.

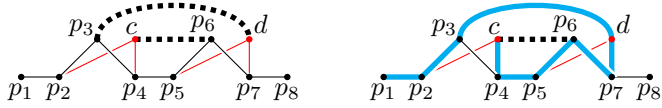
\begin{tikzpicture}[scale=0.5]
\draw(6,2.3) node {};
\draw[-] (0,1)--(1,1);
\draw[-] (3,1)--(4,1);
\draw[-] (6,1)--(7,1);
\draw[-,red] (1,1)--(3,2);
\draw[-,red] (3,1)--(3,2);
\draw[-,line width=2pt,white] (4,1)--(6,2);
\draw[-,red] (4,1)--(6,2);
\draw[-,white,line width=2pt] (6,1)--(6,2);
\draw[-,red] (6,1)--(6,2);
\draw[-,white,line width=2pt] (1,1)--(2,2);
\draw[-,white,line width=2pt] (2,2)--(3,1);
\draw[-,white,line width=2pt] (4,1)--(5,2);
\draw[-,white,line width=2pt] (5,2)--(6,1);
\draw[-] (1,1)--(2,2);
\draw[-] (2,2)--(3,1);
\draw[-] (4,1)--(5,2);
\draw[-] (5,2)--(6,1);
\draw[-,dotted,line width=2pt] (3,2)--(5,2);
\draw[-,dotted,line width=2pt] (2,2)..controls (2,3.2)and(6,3.2)..(6,2);
\filldraw[red]  (3,2)    circle (2pt);
\filldraw[red]  (6,2)    circle (2pt);
\filldraw [black]  (0,1)    circle (2pt)
[black]  (1,1)    circle (2pt)
[black]  (2,2)    circle (2pt)
[black]  (3,1)    circle (2pt)
[black]  (4,1)    circle (2pt)
[black]  (5,2)    circle (2pt)
[black]  (6,1)    circle (2pt)
[black]  (7,1)    circle (2pt);
\draw (0,0.6) node {$p_1$};
\draw (1,0.6) node {$p_2$};
\draw (1.6,2.3) node {$p_3$};
\draw (3,0.6) node {$p_4$};
\draw (4,0.6) node {$p_{5}$};
\draw (4.8,2.4) node {$p_6$};
\draw (6,0.6) node {$p_7$};
\draw (7,0.6) node {$p_8$};
\draw (2.9,2.3) node {$c$};
\draw (6.3,2.4) node {$d$};
\draw (8.5,2.6) node {};
\end{tikzpicture}
\begin{tikzpicture}[scale=0.5]
\draw(6,2.3) node {};
\draw[-,cyan,line width=2pt] (0,1)--(1,1);
\draw[-,cyan,line width=2pt] (3,1)--(4,1);
\draw[-] (6,1)--(7,1);
\draw[-,red] (1,1)--(3,2);
\draw[-,cyan,line width=2pt] (3,1)--(3,2);
\draw[-,line width=2pt,white] (4,1)--(6,2);
\draw[-,red] (4,1)--(6,2);
\draw[-,white,line width=4pt] (6,1)--(6,2);
\draw[-,cyan,line width=2pt] (6,1)--(6,2);
\draw[-,white,line width=4pt] (1,1)--(2,2);
\draw[-,white,line width=2pt] (2,2)--(3,1);
\draw[-,white,line width=4pt] (4,1)--(5,2);
\draw[-,white,line width=4pt] (5,2)--(6,1);
\draw[-,cyan,line width=2pt] (1,1)--(2,2);
\draw[-] (2,2)--(3,1);
\draw[-,cyan,line width=2pt] (4,1)--(5,2);
\draw[-,cyan,line width=2pt] (5,2)--(6,1);
\draw[-,dotted,line width=2pt] (3,2)--(5,2);
\draw[-,cyan,line width=2pt] (2,2)..controls (2,3.2)and(6,3.2)..(6,2);
\filldraw[red]  (3,2)    circle (2pt);
\filldraw[red]  (6,2)    circle (2pt);
\filldraw [black]  (0,1)    circle (2pt)
[black]  (1,1)    circle (2pt)
[black]  (2,2)    circle (2pt)
[black]  (3,1)    circle (2pt)
[black]  (4,1)    circle (2pt)
[black]  (5,2)    circle (2pt)
[black]  (6,1)    circle (2pt)
[black]  (7,1)    circle (2pt);
\draw (0,0.6) node {$p_1$};
\draw (1,0.6) node {$p_2$};
\draw (1.6,2.3) node {$p_3$};
\draw (3,0.6) node {$p_4$};
\draw (4,0.6) node {$p_{5}$};
\draw (4.8,2.4) node {$p_6$};
\draw (6,0.6) node {$p_7$};
\draw (7,0.6) node {$p_8$};
\draw (2.9,2.3) node {$c$};
\draw (6.3,2.4) node {$d$};
\draw (8.5,2.6) node {};
\end{tikzpicture}

Thus we have discarded the possibility $\{i,j\}=\{3,6\}$, which finishes Case 3. and concludes the proof of the
theorem.
\end{proof}

\begin{corollary}
  Assume that $P$ and $Q$ are two longest paths in a simple graph $G$. If $V(Q)\ne V(P)$ and $n=|V(G)|\le 10$ then $V(Q)\cap V(P)$ is a separator.
  Moreover, there is a graph $G$ with $n=|V(G)|= 11$ and two longest paths $P$ and $Q$ with $V(Q)\ne V(P)$, such that $V(Q)\cap V(P)$ is not a 
  separator. Consequently, $n=11$ is the minimal $n$ for which such a graph exists. 
\end{corollary}

\begin{proof}
  By the discussion in the introduction it suffices to discard the cases 
  \begin{itemize}
  \item $\ell=6$, $n=8$,
  \item $\ell=7$, $n=9$,
  \item $\ell=8$, $n=10$,
  \item $\ell=6$, $n=10$.
\end{itemize}
The first three cases are discarded in section~\ref{seccion P' =1}, whereas the last case is discarded in Theorem~\ref{teorema principal}. 
The statements about the example given in the introduction can be verified directly. 
\end{proof}

\end{document}